\documentclass[a4paper,11pt]{book}

\usepackage[utf8]{inputenc} 

\usepackage[T1]{fontenc}

\usepackage[english,francais]{babel}
\usepackage{amsmath}
\usepackage{amsfonts}
\usepackage{amssymb}
\usepackage{amsthm} 
\usepackage{mathrsfs}

\usepackage{subfig} 

\usepackage{tabularx} % Permet d'utiliser l'environnement tabularx
\newcolumntype{C}{>{\centering}X} % colonne de largeur adaptée, centrée
\usepackage{calc} % Pour pouvoir donner des formules dans les définitions
\usepackage[pdftex]{graphicx,color} % Pour inclure des graphiques
\usepackage{hyperref}
\usepackage[nottoc]{tocbibind}

\usepackage{multirow} %pour les tableaux à lignes fusionnées

\usepackage{xy} % pour xypic
\usepackage{bm}
\xyoption{all}
\xyoption{poly}
\xyoption{arc}

\usepackage{a4wide}

\usepackage{url} 
\usepackage{enumerate}

\usepackage{graphics}% pour utiliser rotatebox     

\usepackage{array} 

\theoremstyle{plain} % style plain
\newtheorem{thm}{Théorème}[chapter]
\newtheorem{cor}[thm]{Corollaire}
\newtheorem{lemme}[thm]{Lemme}
\newtheorem{prop}[thm]{Proposition}
\theoremstyle{definition}

\newtheorem{defn}[thm]{Définition}
\newtheorem{rqe}[thm]{Remarque}

\theoremstyle{plain} % style plain
\newtheorem{theo}{Theorem}[chapter]
\newtheorem{coro}[theo]{Corollary}
\newtheorem{propo}[theo]{Proposition}
\newtheorem{lemma}[theo]{Lemma}

\theoremstyle{definition}
\newtheorem{defi}[theo]{Definition}

\newtheorem{remark}[theo]{Remark}

\newcommand{\FS}{{\mathfrak{S}}} 
 
\newcommand{\CO}{{\mathcal{O}}} 
\newcommand{\CL}{{\mathcal{L}}} 
 \newcommand{\CH}{{\mathcal{H}}}
 \newcommand{\CR}{{\mathcal{R}}}
 
\newcommand{\CU}{{\mathcal{U}}} \newcommand{\CA}{{\mathcal{A}}}

\renewcommand{\CR}{{\mathcal{R}}}
\renewcommand{\CL}{{\mathcal{L}}} 

\newcommand{\CK}{{\mathcal{K}}}
 
\renewcommand{\CU}{{\mathcal{U}}} 
\newcommand{\CB}{{\mathcal{B}}}

\newcommand{\BQ}{{\mathbb{Q}}}
\newcommand{\BN}{{\mathbb{N}}}
\newcommand{\BC}{{\mathbb{C}}} 
\newcommand{\BZ}{{\mathbb{Z}}}
\newcommand{\BR}{{\mathbb{R}}}

\newcommand{\fp}{{\mathfrak{p}}}
\newcommand{\fq}{{\mathfrak{q}}}
\newcommand{\FD}{{\mathfrak{D}}}

\makeatletter
\def\hlinewd#1{%
\noalign{\ifnum0=`}\fi\hrule \@height #1 %
\futurelet\reserved@a\@xhline}
\makeatother
 \newenvironment{disarray}%
  {\everymath{\displaystyle\everymath{}}\array}%
  {\endarray}
\newcommand{\qg}{{\backslash}} % "quotient à gauche"
\newcommand{\eps}{\varepsilon} 
\newcommand{\D}{\Delta}
\newcommand{\fconj}{\stackrel{c}{\sim}}
\newcommand{\Lb}{\bar{\CL}}

\newcommand{\Vreg}{{V^{\mathrm{reg}}}}

\newcommand{\Enreg}{{E_n^{\mathrm{reg}}}}

\newcommand{\Ubr}{{U_{\mathrm{branch}}}}
\newcommand{\Vram}{{V_{\mathrm{ram}}}}

\newcommand{\eh}{{e^{2i\pi/h}}}
\newcommand{\tq}{{~|~}}
\newcommand{\ie}{{\emph{i.e.}~}}
\newcommand{\surj}{{\twoheadrightarrow}}
\newcommand{\inj}{{\hookrightarrow}}
\newcommand{\ssi}{{\Leftrightarrow}}

\newcommand{\slfrac}[2]{\left.#1\middle/#2\right.}

\newcommand{\DP}[3][]{\frac{\partial^{#1} #2}{\partial #3^{#1}}}

\newcommand{\Tr}{\operatorname{Tr}\nolimits}
\newcommand{\GL}{\operatorname{GL}\nolimits}

\newcommand{\OO}{\operatorname{O}\nolimits}
\newcommand{\U}{\operatorname{U}\nolimits}

\newcommand{\Red}{\operatorname{Red}\nolimits}

\newcommand{\Cat}{\operatorname{Cat}\nolimits}
\newcommand{\reg}{\operatorname{reg}\nolimits}

\newcommand{\im}{\operatorname{im}\nolimits}
\newcommand{\re}{\operatorname{re}\nolimits}
\newcommand{\Spec}{\operatorname{Spec}\nolimits}
\newcommand{\Ker}{\operatorname{Ker}\nolimits}

\renewcommand{\Im}{\operatorname{Im}\nolimits}

\newcommand{\codim}{\operatorname{codim}\nolimits}

\newcommand{\Disc}{\operatorname{Disc}\nolimits}

\renewcommand{\det}{\operatorname{det}\nolimits}

\renewcommand{\Red}{\operatorname{Red}\nolimits}

\newcommand{\Jac}{\operatorname{Jac}\nolimits}
\newcommand{\Fact}{\operatorname{\textsc{fact}}\nolimits}
\newcommand{\fact}{\operatorname{\underline{fact}}\nolimits}
\newcommand{\LL}{\operatorname{LL}\nolimits}
\newcommand{\lbl}{\operatorname{lbl}\nolimits}

\newcommand{\rg}{\operatorname{rg}\nolimits}
\newcommand{\rk}{\operatorname{rk}\nolimits}
\newcommand{\NCP}{\operatorname{\textsc{ncp}}\nolimits}
\newcommand{\cp}{\operatorname{cp}\nolimits}
\newcommand{\comp}{\operatorname{comp}\nolimits}
\newcommand{\Ch}{\operatorname{\textsc{ch}}\nolimits}

\newcommand{\<}{\preccurlyeq}
\newcommand{\Rleq}{\mathrel{\preccurlyeq_\CR}}
\newcommand{\Sleq}{\mathrel{\preccurlyeq_S}}
\newcommand{\Aleq}{\mathrel{\preccurlyeq_A}}
\newcommand{\Aleqst}{\mathrel{\prec_A}}
\newcommand{\lex}{\mathrel{\leq_{\mathrm{lex}}}}
\newcommand{\pr}{\operatorname{pr}\nolimits}
\newcommand{\Spram}{\operatorname{Spec_1^{\mathrm{ram}}}\nolimits}
\newcommand{\lcm}{\operatorname{lcm}\nolimits}
\newcommand{\grad}{\operatorname{grad}\nolimits}
\newcommand{\grdim}{\operatorname{grdim}\nolimits}

\usepackage{fancyhdr}
\makeatletter
\let\ps@plain=\ps@empty
\makeatother

\let\origdoublepage\cleardoublepage
\newcommand{\clearemptydoublepage}{%
  \clearpage
  {\pagestyle{empty}\origdoublepage}%
}
\let\cleardoublepage\clearemptydoublepage

\FrenchFootnotes % pour les notes de bas de page à la française
\AddThinSpaceBeforeFootnotes % pour avoir une espace fine entre le mot et l'appel de note

\newcommand*\chapterstar[1]{%
  \chapter*{#1}%
  \addcontentsline{toc}{chapter}{#1}%
  \markboth{#1}{#1}}

\title{Groupes de réflexion, géométrie du discriminant et
             partitions non-croisées}
\author{Vivien RIPOLL}

\begin{document}

% Page de garde
% page de garde thèse, à compiler avec manuscrit.tex
%

\pdfbookmark[0]{Page de garde}{garde}
% La commande pdfbookmark permet à hyperref d'ajouter la page de garde dans le
% menu du fichier compilé (dvi, pdf)

\thispagestyle{empty}
% pour ne pas avoir de numéro de page sur la page de garde -- le compteur de
% page est cependant à 1, c'est-à-dire que la numérotation commence à partir
% de la page de garde

\begin{center}
  \begin{tabularx}{\textwidth}{m{1.5cm}Cm{3.5cm}C}
	\includegraphics[height = 4cm,width = 1.5cm]{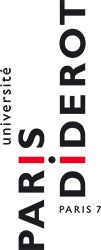}
	& \large Université Paris VII - \hbox{Denis Diderot}
	& \includegraphics[width = 3.5cm]{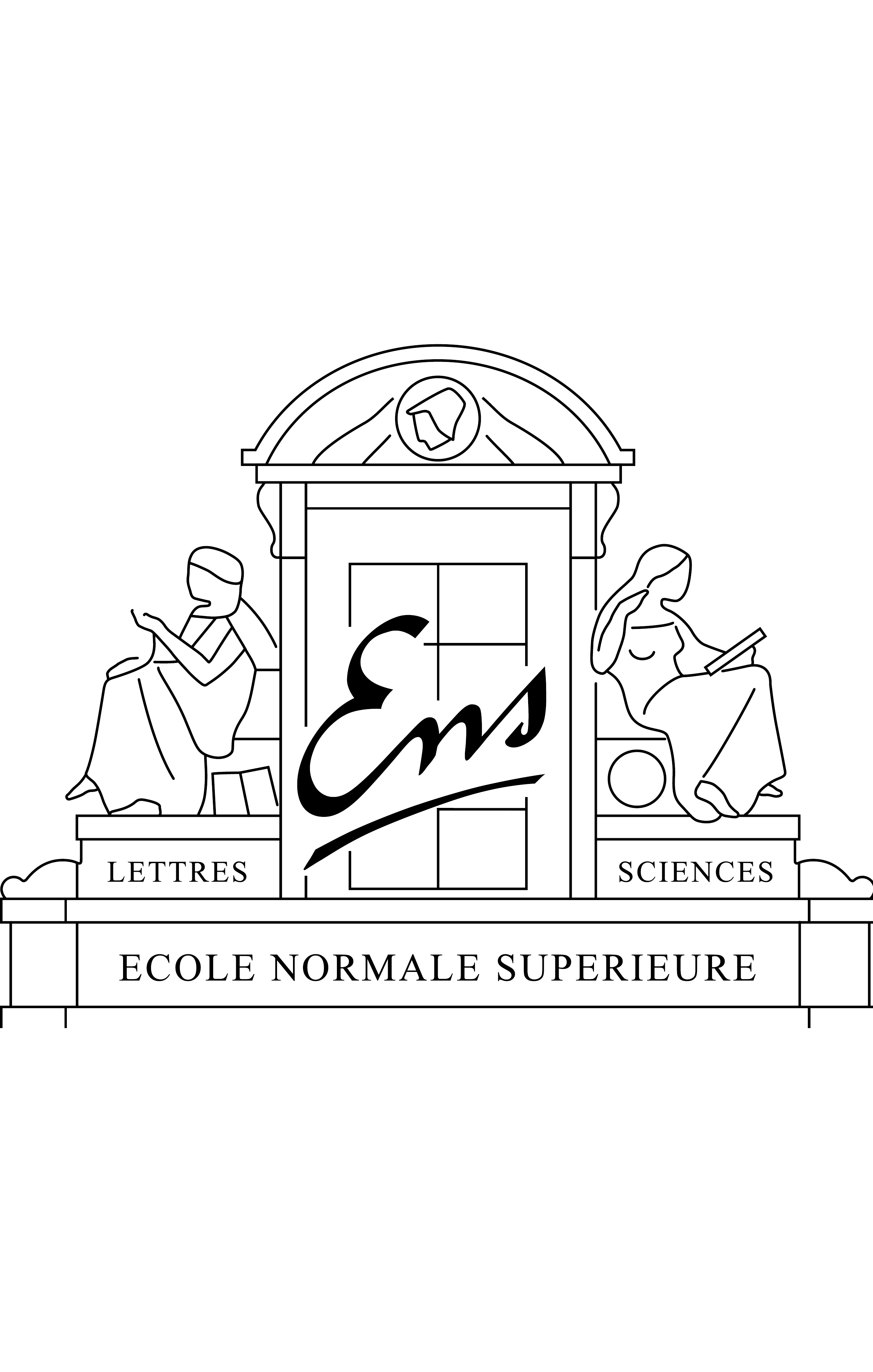}
	& \large École Normale Supérieure
  \end{tabularx}
\end{center}

\begin{center}
\vspace{\stretch{1}}
% Permet de créer un espace vertical de longueur variable (\stretch) et de "poids" 1
{\Large \textbf{École Doctorale Paris Centre}}

\vspace{\stretch{2}}
%Espace vertical variable de poids 2.

{\Huge \textsc{Thèse de doctorat}}

\vspace{\stretch{1}}

{\LARGE Discipline : Mathématiques}

\vspace{\stretch{3}}

{\large présentée par}
\vspace{\stretch{1}}

{\LARGE Vivien \textsc{Ripoll}}

\vspace{\stretch{2}} 
\hrule 
\vspace{\stretch{1}} 
\textbf{{\LARGE Groupes~de~réflexion, géométrie du discriminant\\et
    partitions non-croisées}}

\vspace{\stretch{1}}
\hrule
\vspace{\stretch{1}}

 {\Large 
 dirigée par David \textsc{Bessis}.}

\vspace{\stretch{5}}

{\Large Soutenue le 9 juillet 2010 devant le jury composé de :}

\vspace{\stretch{2}}
{\large
\begin{tabular}{lll}
M. David \textsc{Bessis} & École Normale Supérieure & (Directeur)\\
M. Cédric \textsc{Bonnafé} & Université de Franche-Comté & \\
M. Frédéric \textsc{Chapoton} & Université Lyon 1 & (Rapporteur)\\
M. Patrick \textsc{Dehornoy} & Université de Caen & \\
M. Christian \textsc{Krattenthaler} & Universität Wien & (Rapporteur)\\
M. François \textsc{Loeser} & École Normale Supérieure & \\
M. Jean \textsc{Michel} & Université Paris 7 & 
\end{tabular}
}

\end{center}

\newpage

\vspace*{\fill}

\noindent
Département de Mathématiques 
et Applications\\
École Normale Supérieure\\
45 rue d'Ulm \\
75\,005 Paris

\bigskip
\bigskip

\noindent
Université Paris Diderot - Paris 7\\
5 rue Thomas Mann\\
75\,205 Paris Cedex 13

\bigskip
\bigskip

\noindent
\'Ecole Doctorale de Sciences Mathématiques de Paris-Centre\\
Case 188 \\ 
4 place Jussieu \\
75\,252 Paris Cedex 05

% Résumé
\pdfbookmark[0]{Résumé}{resume}
    
\chapter*{Résumé}

\section*{Résumé}

%Ici le résumé. Ministère : 1700 caractères max, espaces compris.

Lorsque $W$ est un groupe de réflexion complexe bien engendré, le
\emph{treillis $\NCP_W$ des partitions non-croisées de type $W$} est un
objet combinatoire très riche, généralisant la notion de partitions
non-croisées d'un $n$-gone, et intervenant dans divers contextes algébriques
(monoïde de tresses dual, algèbres amassées...). De nombreuses propriétés
combinatoires de $\NCP_W$ sont démontrées au cas par cas, à partir de la
classification des groupes de réflexion. C'est le cas de la formule de
Chapoton, qui exprime le nombre de chaînes de longueur donnée dans le
treillis $\NCP_W$ en fonction des degrés invariants de $W$. Les travaux de
cette thèse sont motivés par la recherche d'une explication géométrique de
cette formule, qui permettrait une compréhension uniforme des liens entre
la combinatoire de $\NCP_W$ et la théorie des invariants de $W$. Le point
de départ est l'utilisation du \emph{revêtement de Lyashko-Looijenga}
($\LL$), défini à partir de la géométrie du discriminant de $W$.

Dans le chapitre 1, on raffine des constructions topologiques de Bessis,
permettant de relier les fibres de $\LL$ aux factorisations d'un élément de
Coxeter. On établit ensuite une propriété de transitivité de l'action
d'Hurwitz du groupe de tresses $B_n$ sur certaines factorisations. Le
chapitre 2 porte sur certaines extensions finies d'anneaux de polynômes, et
sur des propriétés concernant leurs jacobiens et leurs discriminants. Dans
le chapitre 3, on applique ces résultats au cas des extensions définies par
un revêtement $\LL$. On en déduit --- sans utiliser la classification --- des
formules donnant le nombre de factorisations sous-maximales d'un élément de
Coxeter de $W$ en fonction des degrés homogènes des composantes
irréductibles du discriminant et du jacobien de $\LL$.

\subsubsection*{Mots-clefs}
Groupes de réflexion complexes, partitions non-croisées, nombres de
Fuss-Catalan, formule de Chapoton, revêtement de Lyashko-Looijenga,
factorisations d'élément de Coxeter.

\begin{otherlanguage}{english}

  \vspace{1cm}

  \begin{center} \rule{\textwidth/3}{1pt} \end{center}

  \vspace{1cm}
\newpage
  \begin{center} \Large \bf Reflection groups, geometry of the discriminant
    and noncrossing partitions \end{center}
  \section*{Abstract}
  
 % Ici le résumé. Ministère : 1700 caractères max, espaces compris.
  
When $W$ is a well-generated complex reflection group, the \emph{noncrossing
partition lattice $\NCP_W$ of type $W$} is a very rich combinatorial object,
extending the notion of noncrossing partitions of an $n$-gon. This
structure appears in several algebraic setups (dual
braid monoid, cluster algebras...). Many combinatorial properties of
$\NCP_W$ are proved case-by-case, using the classification of reflection
groups. It is the case for Chapoton's formula, expressing the number
of multichains of a given length in the lattice $\NCP_W$, in terms of the
invariant degrees of $W$. This thesis work is motivated by the search
for a geometric explanation of this formula, which could lead to a
uniform understanding of the connections between the combinatorics of
$\NCP_W$ and the invariant theory of $W$. The starting point is to use the
\emph{Lyashko-Looijenga covering} ($\LL$), based on the geometry of the
discriminant of $W$. In the first chapter, some topological
constructions by Bessis are refined, allowing to relate the fibers of
LL with block factorisations of a Coxeter element. Then we prove a
transitivity property for the Hurwitz action of the braid group $B_n$ on
certain factorisations. Chapter 2 is devoted to certain finite
polynomial extensions, and to properties about their Jacobians and
discriminants. In Chapter 3, these results are applied to the
extension defined by the covering $\LL$. We deduce --- with a case-free
proof --- formulas for the number of submaximal factorisations of a
Coxeter element in $W$, in terms of the homogeneous degrees of the
irreducible components of the discriminant and Jacobian for $\LL$.

 \subsection*{Keywords} 
 Complex reflection groups, noncrossing partitions, Fuss-Catalan numbers,
 Chapoton's formula, Lyashko-Looijenga covering, factorisations of a
 Coxeter element.

\end{otherlanguage}

% Table des matières
\setcounter{tocdepth}{1} % pour régler sa profondeur - par défaut : 2
\pdfbookmark[0]{Table des matières}{tablematieres} % pour ajouter la table des matières dans l'``index'' du fichier compilé

\tableofcontents  % pour afficher la table des matières

% Introduction
%%%%%%%%%%%%%%%%%%%%%%%%%%%%%%%%%%%%%%%%%%%%%%%%%%%%%%%%%%%%%%%%%%%%%%%%%
% Intro thèse, à compiler avec les en-têtes de manuscrit.tex
%%%%%%%%%%%%%%%%%%%%%%%%%%%%%%%%%%%%%%%%%%%%%%%%%%%%%%%%%%%%%%%%%%%%%%%%%

\chapterstar{Introduction}
\label{intro}

Le point de départ des travaux de cette thèse était d'essayer de comprendre
une formule qui exprime le nombre de chaînes de longueur donnée dans le
treillis des partitions non-croisées associé à un groupe de réflexion $W$,
en fonction des degrés invariants de $W$. J'ai ainsi été amené à étudier
les relations entre la combinatoire du treillis et la géométrie de $W$, et
plus précisément à faire le lien entre certaines factorisations d'un
élément de Coxeter et le revêtement de Lyashko-Looijenga de $W$.

Ce chapitre introductif vise à présenter les principaux personnages :
\begin{itemize}
\item les groupes de réflexion complexes bien engendrés ;
\item le treillis des partitions non-croisées de type $W$, pour $W$ un tel
  groupe ;
\item le revêtement de Lyashko-Looijenga associé à $W$ ;
\item les factorisations par blocs d'un élément de Coxeter de $W$.
\end{itemize}

\medskip

La première partie introduit le treillis des partitions non-croisées de
type $W$, un objet combinatoire très riche, qui étend la notion de
partitions non-croisées d'un $n$-gone. Les parties \ref{sectalg},
\ref{sectcomb} et \ref{sectgeom} donnent un aperçu des aspects
combinatoires, algébriques, et géométriques de cet objet et des autres
structures étudiées. On y expose l'état de l'art et le contexte historique
dans lequel s'inscrit ce travail.

Dans la partie \ref{sectresult} on présente les principaux résultats de la
thèse.

\section{Le treillis des partitions non-croisées}
\label{partintroNCP}

\subsection{Partitions non-croisées classiques}

Soit $n$ un entier non nul. Considérons $n$ points du plan, disposés sur
les sommets d'un $n$-gone régulier, et étiquetés $1,\dots ,n$ dans le sens
des aiguilles d'une montre. Étant donnée une partition de l'ensemble $\{1,\dots,
n\}$, on peut tracer dans le polygone les enveloppes convexes des parts de
la partition. On dit que celle-ci est \emph{non-croisée} si ces enveloppes
convexes ne s'intersectent pas deux à deux (dans le cas contraire, elle est
dite croisée).

\begin{figure}[!h]

  \begin{center}

      %\begin{minipage}[c]{.5\linewidth}
        \resizebox{4.5cm}{!}{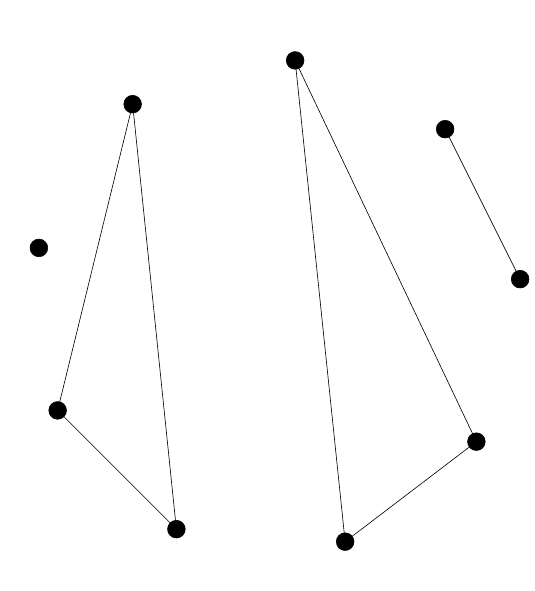}
      %\end{minipage}
        \hspace{3cm}
      %\begin{minipage}[c]{.5\linewidth}
        \resizebox{4.5cm}{!}{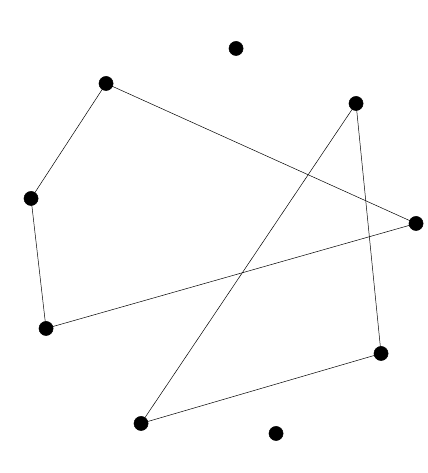}
      %\end{minipage} 

  \end{center}
\caption{La partition $P_1=\{\{1,4,5\} , \{2,3\} , \{6,7,9\}, \{8\}\}$
    est non-croisée ; la partition $P_2= \{
    \{1\},\{2,4,6\},\{5\},\{3,7,8,9\} \}$ est croisée.}
\label{figNCP}
\end{figure}

Si l'on préfère une définition combinatoire, on a la caractérisation
suivante : une partition de $\{1,\dots, n\}$ est
non-croisée si et seulement si pour $1\leq i <j<k<l\leq n$, si $i$ et $k$
sont dans la même part $P$, alors $j$ et $l$ sont soit dans des parts
distinctes, soit tous deux dans la part $P$ également.

\medskip

L'ensemble des partitions non-croisées d'un $n$-gone (que l'on notera
$\NCP_n$) est un ensemble partiellement ordonné, par l'ordre de raffinement
des partitions ($\alpha \leq \beta$ si $\alpha$ est une partition plus fine
que $\beta$). C'est de plus un treillis pour cet ordre, c'est-à-dire qu'il
existe des infs et des sups (cf. définition \ref{deftreillis}).

\medskip

Le treillis $\NCP_n$ est l'un des nombreux objets
combinatoires comptés par le nombre de Catalan :
\[ |\NCP_n | = \Cat(n) = \frac1{n+1} \binom{2n}{n} \ . \]

La première étude détaillée des partitions non-croisées date de 1972, par
Kreweras \cite{krew} (pour plus de détails sur l'historique du problème, on
renvoie à Stanley \cite[pp. 261--262]{stanley2} et Armstrong
\cite[p. 5]{Armstrong}). Le treillis est devenu depuis un objet classique en
combinatoire algébrique (cf. l'article de synthèse de Simion \cite{simion}). 

\medskip

Plus récemment, l'intérêt pour les partitions non-croisées s'est
diversifié. Elles sont étudiées en lien avec la théorie des probabilités
libres de Voiculescu (voir les synthèses de Speicher \cite{speicher} et Biane
\cite{biane}), et dans de nouveaux problèmes combinatoires (voir par
exemple les fonctions parkings \cite{stanley3}). Surtout, elles
interviennent dans de nombreuses situations algébriques (théorie des
groupes, théorie des représentations) sur lesquelles nous reviendrons plus
loin dans cette introduction. On renvoie à l'article \cite{mc} de McCammond
pour un tour d'horizon des diverses incarnations des partitions
non-croisées découvertes ces dernières années.

\medskip

Depuis la fin des années 90, la structure combinatoire a été généralisée
dans le contexte des groupes de réflexion finis (réels, puis complexes), le
cas du $n$-gone correspondant au type $W(A_{n-1})$, \ie au groupe de
permutations $\FS_n$. Pour expliquer cette généralisation, commençons par
donner une autre interprétation des partitions non-croisées.

A chaque partition $P$ (croisée ou non) d'un $n$-gone, on peut associer une
permutation $\sigma$ de $\FS_n$, qui est le produit des cycles donnés par
chaque part de $P$, ses éléments étant pris dans le sens des aiguilles
d'une montre : voir les diagrammes de la figure \ref{figperm}.

\begin{figure}[!h]
  
  % \begin{minipage}[c]{.4\linewidth}
  \begin{center}
    \resizebox{4.5cm}{!}{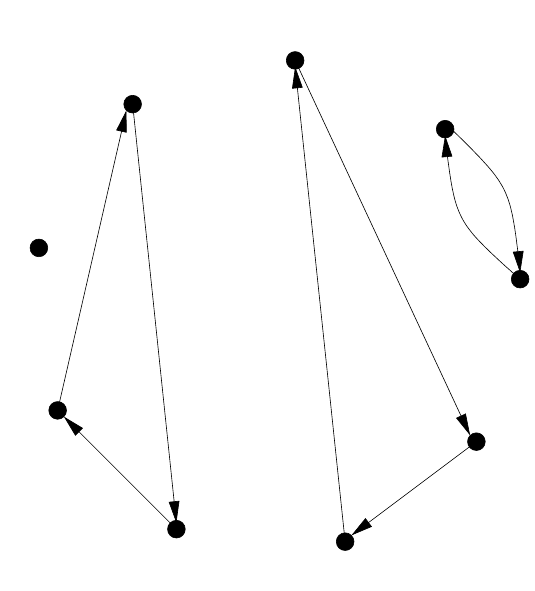}
    % \end{minipage}
    \hspace{3cm}
    % \begin{minipage}[c]{.4\linewidth}
    \resizebox{4.5cm}{!}{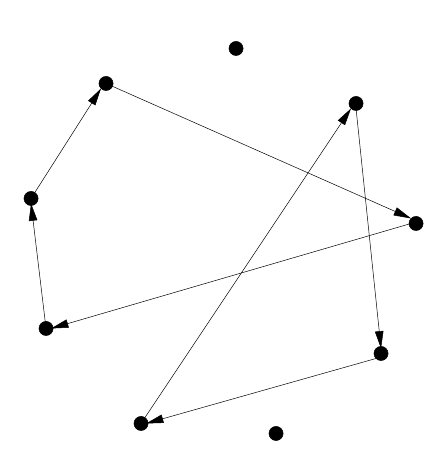}
    % \end{minipage} 
  \end{center}
  \caption{Les permutations $\sigma_1=(1\ 4\ 5)(2\ 3)(6\ 7\ 9)$ et
    $\sigma_2= (2\ 4\ 6)(3\ 7\ 8\ 9)$, correspondant aux partitions $P_1$
    et $P_2$ de la figure \ref{figNCP}.}
\label{figperm}
\end{figure}

On peut alors montrer que les partitions non-croisées correspondent
exactement aux permutations qui se situent sur une géodésique entre la
permutation identité et le $n$-cycle $c=(1\ 2\ \dots \ n)$ dans le graphe
de Cayley de $(\FS_n, T)$, où $T$ est l'ensemble de toutes les
transpositions de $\FS_n$. Ainsi, l'ensemble des partitions non-croisées
est en bijection avec l'ensemble des permutations $\sigma$ de $\FS_n$
telles que \[\ell_T(\sigma) + \ell_T (\sigma^{-1}c) = \ell_T (c) \ ,\] où
$\ell_T (\sigma)$ désigne le nombre minimal de transpositions nécessaires
pour écrire $\sigma$ (on a $\ell_T(\sigma)=n- \#\{\text{cycles de } \sigma
\}$). Sous cette forme, en voyant le groupe $\FS_n$ comme un groupe de
Coxeter de type $A$, on peut généraliser la notion de partitions
non-croisées à tous les groupes de Coxeter finis (et à certains groupes de
réflexion complexes).

\subsection{Généralisation aux groupes de réflexion}

\begin{center}
\textbf{Dans la suite, sauf mention explicite du contraire,
      tous les groupes de réflexion seront supposés \emph{finis}.}
\end{center}

\subsubsection*{Groupes de Coxeter finis}

Soit un groupe de réflexion réel fini $W$, c'est-à-dire un sous-groupe fini
de $\GL(V)$ (où $V$ est un espace vectoriel réel), engendré par des
réflexions. La théorie de Coxeter permet d'étudier ce type de groupe de
manière combinatoire, en fixant un ensemble de générateurs particuliers,
noté $S$, tel que $(W,S)$ ait une présentation très remarquable. On renvoie
aux livres \cite{Kane}, \cite{Humphreys}, ou \cite{Bourbaki} pour plus de
détails sur les groupes de Coxeter et les groupes de réflexion finis.

Pour construire les partitions non-croisées de type $W$, on munit au
contraire $W$ d'une autre partie génératrice, notée $\CR$, et constituée de
\emph{toutes} les réflexions de $W$. Dans le cas d'un groupe de type $A$,
on retrouve l'ensemble de toutes les transpositions, noté $T$ dans le
paragraphe précédent\footnote{L'ensemble de toutes les réflexions de $W$
  est d'ailleurs par extension noté $T$ par plusieurs auteurs
  (cf. \cite{Bessisdual} ou \cite{Armstrong}) ; on préfère ici la notation
  $\CR$ pour \og réflexions \fg.}.

\medskip

\begin{defn}
\label{defordreintro}
On note $\ell_\CR(w)$ la longueur d'un élément $w$ de $W$ en tant que mot
sur l'alphabet $\CR$, \ie
\[ \ell_\CR(w)= \min \{p\in \BN \tq \exists r_1,\dots, r_p \in \CR,
w=r_1\dots r_p \} . \]
Si $r_1\dots r_p =w$, avec $p=\ell_\CR(w)$, on dit que $(r_1,\dots, r_p)$
est une décomposition réduite de $w$. On définit un ordre partiel $\Rleq$ sur $W$,
en disant que $u$ \og divise \fg~  $v$ si et seulement si $u$ peut s'écrire en préfixe
d'une décomposition réduite de $v$, \ie :
\[ u \Rleq v \text{ si et seulement si } \ell_\CR(u) + \ell_\CR(u^{-1}v) =
\ell_\CR(v) \ . \]
\end{defn}

On notera simplement $\ell$ et $\<$, plutôt que $\ell_\CR$ et $\Rleq$, s'il
n'y a pas d'ambiguïté ; l'ordre $\<$ est parfois appelé ordre absolu sur
$W$ (et $\ell$ longueur absolue)\footnote{Il ne doit pas être confondu avec
  l'ordre plus classique $\Sleq$ sur les groupes de Coxeter, défini par la
  longueur $\ell_S$ relative à un système de réflexions fondamentales, ou
  encore avec l'ordre de Bruhat.}.

\begin{rqe}
\label{rqhasse}
Le diagramme de Hasse pour cet ordre partiel est donné par
le graphe de Cayley de $(W,\CR)$ (en effet, si $w\in W$ et $r\in \CR$, on
a nécessairement $\ell_\CR (rw)= \ell_\CR(w) \pm 1$, car les réflexions ont pour
déterminant $-1$).
\end{rqe}

\medskip

Le rôle du $n$-cycle dans le type $A$ va être joué ici par un \emph{élément
  de Coxeter} de $W$, c'est-à-dire un produit (dans n'importe quel ordre)
de toutes les réflexions par rapport aux murs d'une chambre donnée de
l'arrangement d'hyperplans de $W$ (cf. par exemple
\cite[29.1]{Kane}). L'ensemble des éléments de Coxeter de $W$ forme une classe
de conjugaison constituée d'éléments de longueur $\ell_\CR$
maximale dans $W$ (en type $A$ --- mais pas dans le cas général --- on
obtient tous les éléments de longueur maximale, \ie les cycles
maximaux). Fixons un élément de Coxeter $c$. On définit ainsi l'ensemble des
partitions non-croisées de type $W$ :
\[ \NCP_W(c)= \{ w \in W \tq w\< c\} \ .\] 

Soient $c$ et $c'$ deux éléments de Coxeter. Comme la longueur $\ell_\CR$
est invariante par conjugaison, et que $c$ et $c'$ sont conjugués,
$\NCP_W(c)$ et $\NCP_W(c')$ sont isomorphes en tant qu'ensembles
ordonnés. Du coup, tant qu'on peut travailler à isomorphisme près, on
notera simplement $\NCP_W$.

On peut montrer (et ce n'est pas évident, voir paragraphe
\ref{subparttreillis}), que, comme pour le type $A$, l'ensemble
partiellement ordonné $(\NCP_W, \<)$ est un treillis.

\subsubsection*{Groupes de réflexion complexes}

Dans \cite{BessisKPi1}, motivé par des questions géométriques sur les
groupes de tresses généralisés (cf. partie \ref{sectalg}), Bessis montre
que la définition de $\NCP_W$ peut s'étendre au cas où $W$ est un groupe de
réflexion complexe bien engendré. Rappelons quelques définitions (plus de
détails sont donnés au paragraphe \ref{subsectgrcdbm}).

\medskip

Si $V$ est un $\BC$-espace vectoriel de dimension $n\geq 1$, une
\emph{réflexion}\footnote{Certains auteurs utilisent le terme
  \emph{pseudo-réflexion} ; on préfère ici utiliser simplement
  \emph{réflexion} quand le cadre est un espace vectoriel complexe.} est un
élément $r$ de $\GL(V)$, d'ordre fini, et tel que $\Ker(r-1)$ soit un
hyperplan de $V$. Son unique valeur propre distincte de $1$ est donc une
racine de l'unité, qui n'est pas nécessairement $-1$ comme dans le cas
réel.

Un \emph{groupe de réflexion complexe} (fini) est un sous-groupe fini de
$\GL(V)$ engendré par des réflexions. Tout groupe de réflexion réel fini
peut bien sûr se voir comme un groupe complexe, en complexifiant l'espace
vectoriel sur lequel il agit. Cependant, on obtient avec cette définition
de nombreux nouveaux groupes de réflexion non-réels (et ce même dans le cas
où toutes les réflexions du groupe sont d'ordre $2$). La théorie de Coxeter
ne s'applique pas à ces groupes. Mais comme les définitions de l'ordre $\<$
et de la longueur $\ell$ 
vues plus haut sont relatives à l'ensemble $\CR$ de toutes les réflexions de
$W$, on peut les étendre telles quelles aux groupes de réflexion complexes. On
construit ainsi, comme dans le cas réel, un ensemble partiellement ordonné
$(W,\<)$, dont la structure provient de la longueur de réflexion
$\ell$.

\medskip

Notons que la situation est plus subtile que dans le cas réel : comme les
réflexions ne sont pas forcément de déterminant $-1$, on peut théoriquement
avoir $\ell_\CR(wr)=\ell_\CR(w)$ pour $w\in W$ et $r\in \CR$ (et ceci se
produit effectivement dans certains cas) ; la remarque \ref{rqhasse} sur le
lien entre le diagramme de Hasse de $\<$ et le graphe de Cayley de
$(W,\CR)$ n'est donc plus valide. On verra cependant que, si l'on se
restreint aux diviseurs d'un élément de Coxeter, ce problème ne se pose
plus.

\medskip

Si $W$ est un groupe de réflexion complexe qui agit de façon
essentielle\footnote{c'est-à-dire que $V^W=\{0\}$ : en tant que $W$-module,
  $V$ n'a pas de composante irréductible triviale.}  sur l'espace $V$ de
dimension $n$, on dit que $W$ est \emph{bien engendré} s'il peut être
engendré par $n$ réflexions (c'est le cas pour tous les groupes réels, mais
pas pour tous les groupes complexes). Si $W$ est bien engendré, il existe
une généralisation naturelle de la notion d'élément de Coxeter (voir la
définition \ref{defcox} du chapitre \ref{chaphur}). On peut ainsi donner la
définition générale de l'ensemble des partitions non-croisées de type $W$.

\begin{defn}
\label{defNCP}
Soit $W$ un groupe de réflexion complexe bien engendré, et $c$ un élément
de Coxeter de $W$. On note $\<$ l'ordre sur $W$ défini par la
$\CR$-longueur (cf. Déf. \ref{defordreintro}). Le treillis des partitions
non-croisées de type $W$ est l'intervalle $[1,c]$ pour l'ordre $\<$ :
\[ \NCP_W(c)= \{ w \in W \tq w\< c\} \ .\] 
\end{defn}

Il s'avère que la structure combinatoire de cet ensemble est alors très
similaire au cas réel ; en particulier $\NCP_W$ est encore un treillis. Au
paragraphe \ref{subsectgrcdbm} on donnera une interprétation algébrique de
ce treillis.

\subsubsection*{Historique et références}

Le treillis des partitions non-croisées de type $W$ est un objet à la
combinatoire très riche, et également très utile pour comprendre la
structure et la géométrie des groupes de réflexion ; on développe ces
questions en \ref{sectalg} et \ref{sectcomb}. Cette
généralisation est apparue à la fin des années 90, de façon indépendante
dans différents domaines mathématiques.

Travaillant en combinatoire algébrique, Reiner introduit des partitions
non-croisées de types $B$ et $D$ dans \cite{reiner}. Du côté de la théorie
des groupes, la construction du monoïde de Birman-Ko-Lee dans \cite{BKL}
suscite deux axes de recherche indépendants donnant une définition
algébrique de $\NCP_W$ : voir d'une part Bessis-Digne-Michel \cite{BDM}
(suivi de l'article \cite{Bessisdual}), d'autre part Brady-Watt
\cite{BradyA_n,BWordre,BWordre2}. Toujours dans la même période, des
problèmes en théorie des probabilités libres motivent également la
construction de certaines généralisations (Biane-Goodman-Nica \cite{BGN}).

En janvier 2005 une rencontre est organisée à l'American Institute of
Mathematics, rassemblant pour la première fois tous les chercheurs de
domaines divers travaillant sur ce même objet : il en résulte une
longue liste de problèmes ouverts (cf. \cite{NCP}).

\medskip

Pour plus de détails historiques, on renvoie au chapitre 4.1. dans
\cite{Armstrong}. Ici on va s'intéresser tout particulièrement aux
motivations algébriques liées aux groupes de réflexion et monoïdes de
tresses.

\section{Motivations algébriques}
\label{sectalg}

\subsection{Groupes de tresses}
\label{subsecttresses}

Soit $W \subseteq \GL(V)$ un groupe de réflexion complexe, et $\CR$
l'ensemble de toutes ses réflexions. On définit l'arrangement d'hyperplans
de $W$ :
\[\CA := \{ \Ker (r-1) \tq r\in \CR \} \ , \]
et l'espace des points réguliers :
\[ \Vreg := V - \bigcup_{H\in \CA} H \ .\] Le groupe $W$ agit
naturellement\footnote{On considèrera qu'il s'agit d'une action à gauche,
  et on notera les quotients à gauche.} sur $\Vreg$, et on définit :
\[ B(W) := \pi_1(W \qg \Vreg) \ ,\]
le groupe de tresses de $W$. Pour $W=\FS_n$ (\ie le groupe de
Coxeter $W(A_{n-1})$), on retrouve le groupe de tresses usuel à $n$ brins.

\bigskip

Dans le cas où $W$ est un groupe réel (complexifié), la théorie de Coxeter
permet de comprendre la structure de $B(W)$. Si l'on choisit une
chambre $C$ de l'arrangement réel (\ie une composante connexe de l'espace
des points réguliers \emph{réels}), alors l'ensemble $S \subseteq \CR$
formé des réflexions par rapport aux murs de $C$ engendre $W$, et l'on a
une \emph{présentation de Coxeter} :
\begin{equation*} W \simeq \big< S \ \big| \ \forall s\in S, s^2 =1 \; ; \;
  \forall s, t \in S\ (s\neq t),\ \underbrace{sts\dots}_{m_{s,t}} =
  \underbrace{tst\dots}_{m_{s,t}} \big>_{\text{groupe}}\
  \tag{1},\end{equation*} où $m_{s,t}$ est l'ordre du produit $st$. D'après
un théorème de Brieskorn (\cite{Brieskorn2}), le groupe de tresses $B(W)$
est isomorphe au groupe d'Artin-Tits de $W$, défini par la présentation
suivante\footnote{dans laquelle l'ensemble des générateurs est plus
  exactement une copie formelle de $S$.} :
\begin{equation*}A(W,S) := \big< S \ \big| \ \forall s, t \in S\ (s\neq
  t),\ 
  \underbrace{sts\dots}_{m_{s,t}} = \underbrace{tst\dots}_{m_{s,t}}
  \big>_{\text{groupe}}\ \tag{2}.\end{equation*} Deligne et Brieskorn-Saito,
dans les deux articles fondateurs \cite{Deligne,BriSaito}, ont étudié la
structure de $B(W)$, notamment en construisant une forme normale dans
$A(W,S)$ (y résolvant ainsi le problème des mots), et en établissant que le
\emph{monoïde d'Artin-Tits} $A_+(W,S)$ (défini par la présentation (2),
mais en tant que monoïde) se plonge dans son groupe des fractions $A(W,S)$.

\medskip

Cette méthode n'offre par contre aucune prise sur les groupes de réflexion
complexes, puisqu'alors la théorie de Coxeter ne s'applique pas (on n'a
plus de domaine fondamental naturel comme dans la géométrie réelle). Ce
n'est qu'au début des années 2000 qu'a été contruit un substitut pour le
monoïde d'Artin-Tits qui, lui, peut se généraliser.

\subsection{Groupes de réflexion complexes et monoïde de tresses dual}
\label{subsectgrcdbm}

Lorsqu'on passe des groupes de Weyl finis (groupes de réflexion définis sur
$\BQ$) aux groupes de Coxeter finis (définis sur $\BR$), peu de nouveaux
exemples se présentent (dans la classification des groupes irréductibles,
ce sont les types $H_3$, $H_4$, et les groupes diédraux). Par contre, si
l'on autorise les groupes de réflexion à être définis sur $\BC$, la
situation change drastiquement. La classification des groupes de réflexion
complexes irréductibles a été établie en 1954 par Shephard-Todd
(\cite{Shephard} ; voir aussi Cohen \cite{cohen}) ; elle contient :
\begin{itemize}
\item une série infinie à $3$ paramètres : $G(de,e,r)$ est le groupe des
  matrices monomiales\footnote{Un seul coefficient non nul sur chaque
    ligne, ainsi que sur chaque colonne ; autrement dit, une matrice
    monomiale est le produit d'une matrice diagonale inversible et d'une
    matrice de permutation.} de $\GL_r(\BC)$, dont les
  coefficients non nuls sont dans $\mu_{de}$, avec un produit dans
  $\mu_d$ (où $\mu_n$ désigne l'ensemble des racines $n$-ièmes de
  l'unité). On y retrouve les séries infinies réelles : $G(1,1,n)\simeq
  W(A_{n-1})$, $G(2,1,n)\simeq  W(B_n)$, $G(2,2,n)\simeq W(D_n)$,
  $G(e,e,2)\simeq W(I_2(e))$. 
\item $34$ groupes exceptionnels (numérotés historiquement $G_4,\dots,
  G_{37}$), de petit rang : $19$ sont de rang $2$, et le rang maximal est $8$
  (atteint pour $G_{37}$ qui n'est autre que le groupe $W(E_8)$).
\end{itemize}

Contrairement au cas réel, cette classification n'apparaît pas comme la
solution d'un problème combinatoire simple. Elle a été obtenue par des
méthodes \emph{ad hoc}, et reste encore aujourd'hui assez mal comprise. En
particulier, on ne peut pas toujours trouver un système générateur minimal
de réflexions qui soit \og naturel \fg, et trouver des présentations satisfaisantes
pour ces groupes est un problème difficile (voir les travaux de
Broué-Malle-Rouquier \cite{broumarou1,broumarou}, où sont construits des
diagrammes \og à la Coxeter \fg\ pour les groupes complexes). Encore
aujourd'hui, on ne dispose pas d'une vision globale satisfaisante sur ces
groupes, et de nombreuses propriétés importantes sont démontrées au cas par
cas (comme des coïncidences constatées sur tous les groupes de la
classification), et non comprises en profondeur. Les travaux de cette thèse
visent d'ailleurs à améliorer la compréhension conceptuelle de certaines de
ces propriétés empiriques.

\bigskip

L'étude des groupes de réflexion complexes connaît depuis une quinzaine
d'années un regain d'intérêt, motivé par de nouvelles questions en théorie
des représentations. De la même façon que les groupes de Weyl sont les \og
squelettes \fg\ de la théorie de Lie, il semblerait en effet qu'on puisse
construire des structures algébriques (appelées \emph{Spetses}) associées
aux groupes de réflexion complexes (travaux de Broué, Malle, Michel,
Rouquier, présentés dans \cite{broue}).

Pour une présentation exhaustive de la théorie des groupes de réflexion
complexes, on renvoie au récent livre de Lehrer-Taylor \cite{unitary}.

\bigskip

En 1998, Birman-Ko-Lee \cite{BKL} donnent une nouvelle présentation pour le
groupe de tresses usuel, qui est réinterprétée par Bessis-Digne-Michel
\cite{BDM} dans le langage alors tout récent des \emph{monoïdes de Garside}
(voir partie suivante). Ceci est le point de départ de la construction dans
\cite{Bessisdual} d'un
nouveau monoïde de tresses pour tous les groupes de Coxeter finis (voir
aussi \cite{BWordre2}) : le \emph{monoïde de
  tresses dual}. Celui-ci vérifie des propriétés structurelles analogues à
celles du monoïde d'Artin-Tits (c'est un monoïde de Garside), mais il a
l'avantage de ne pas utiliser la structure de Coxeter du groupe : il est en
effet construit à partir de la paire $(W,\CR)$, où $\CR$ est l'ensemble de
\emph{toutes} les réflexions. Cela permet de le définir également pour
certains groupes de réflexion non réels (ceux qui sont bien engendrés),
cf. \cite{BessisKPi1}.

\medskip

Pour une présentation synthétique du monoïde dual et de ses applications,
on renvoie à l'introduction de \cite{bessishdr}. Donnons ici simplement sa
définition. Celle-ci s'inspire d'une présentation alternative du monoïde
d'Artin-Tits (établie par Tits) :
\begin{equation*}A_+ (W,S)\simeq \big< W\ |\ w.w'=ww' \text{ (produit dans }
W\text{) si } \ell_S(w)+\ell_S(w')=\ell_S(ww') \big>_{\text{monoïde}},
\tag{3} \end{equation*} 
où $\ell_S$ désigne la longueur relativement à l'ensemble $S$ des réflexions
fondamentales (définition analogue à Déf. \ref{defordreintro}). On définit
alors le monoïde dual de tresses par :
\begin{equation*}M(W,c):=\big< \NCP_W(c) |\ w.w'=ww' \text{ (produit dans }
  W\text{) si } \ell_\CR(w)+\ell_\CR(w')=\ell_\CR(ww')
  \big>_{\text{monoïde} } ,\tag{4} \end{equation*} où $c$ est un élément de
Coxeter et où $\NCP_W(c)$, qui sert ici de système générateur, est l'ensemble
des diviseurs de $c$ pour l'ordre $\Rleq$ (Déf. \ref{defNCP}). On
montre que le monoïde dual se plonge dans son groupe des fractions, qui est
$B(W)$, mais qu'il n'est pas (dans le cas réel) isomorphe au monoïde
d'Artin-Tits.

\subsection{Structures de Garside}
\label{subsectgarside}

Le monoïde d'Artin-Tits et le monoïde dual sont deux exemples de \emph{monoïdes
de Garside}. Cette notion a été introduite à la fin des années 90 par
Dehornoy et Paris \cite{Deh2,DehPa}, dans le but d'axiomatiser les
propriétés essentielles de $A_+(W,S)$, de sorte que les méthodes de preuve
et les résultats de Deligne et Brieskorn-Saito s'appliquent dans un cadre
plus général. 

Il n'est pas utile de donner ici la définition précise d'un monoïde de
Garside\footnote{Le nom donné à cette structure provient du premier cas
  historique, celui du monoïde de tresses classique de type $A$, étudié par
  Garside dans \cite{Garside}.} --- d'autant que celle-ci n'est pas tout à
fait stabilisée ---, mais signalons qu'un axiome fondamental est une
propriété de treillis pour un système générateur du monoïde. Les
générateurs de la présentation (3) de $A_+ (W,S)$ constituent l'ensemble
ordonné $(W,\Sleq)$ (où $\Sleq$ est l'ordre de divisibilité associé à $S$),
dont on peut montrer géométriquement qu'il est un treillis. Pour le cas du
monoïde dual (présentation (4)), il faut introduire un élément de Coxeter
(l'ensemble $(W,\Rleq)$ entier n'est en général pas un treillis). Le
système générateur est alors précisément $\NCP_W(c)$, et on comprend
l'importance algébrique de la propriété de treillis pour l'ensemble des
partitions non-croisées de type $W$ (voir paragraphe
\ref{subparttreillis}).

\bigskip

La construction du monoïde de tresses dual pour les groupes non réels a
donc des applications importantes, puisqu'elle permet d'adapter les preuves
de Deligne et Brieskorn-Saito et de démontrer des propriétés fondamentales
de $W$ et $B(W)$ (problème des mots, problème de conjugaison...) dans des
cas nouveaux. Une application importante est la démonstration de la
propriété $K(\pi,1)$ pour les arrangements d'hyperplans des groupes de
réflexion complexes \cite{BessisKPi1}.

\bigskip

\begin{rqe}
  On ne traite dans cette thèse que de groupes de réflexion
  finis. Cependant, des monoïdes duaux ont également été construits 
  pour les groupes d'Artin-Tits de type non sphérique :
  \cite{charneypeifer,Digne,Bessis,BCKM}. 
\end{rqe}

La théorie de Garside continue par ailleurs de se développer, et de nombreuses
généralisations de structures de Garside ont été définies ces dernières années
(groupoïdes, catégories) : voir par exemple \cite{BessisGarside},
\cite{krammer}, et l'article de synthèse \cite{Gar}.

\section{Combinatoire de Coxeter-Fuss-Catalan}
\label{sectcomb}

Fixons un groupe de réflexion complexe bien engendré $W$. L'ensemble
ordonné $\NCP_W$ possède une combinatoire particulièrement riche, explorée
depuis quelques années (au moins lorsque $W$ est réel). Suivant une
terminologie d'Armstrong, on pourra parler de \emph{combinatoire de
  Coxeter-Fuss-Catalan}, car des généralisations des nombres de
Fuss-Catalan interviennent.

On donne dans cette partie quelques exemples (non exhaustifs) des problèmes
posés par cette combinatoire. Pour un exposé plus complet, et pour d'autres
références, on pourra se reporter à la thèse d'Armstrong \cite{Armstrong},
en particulier les chapitres 1 et 5.

\subsection{Propriété de treillis}
\label{subparttreillis}

Il est temps de rappeler la définition d'un treillis.

\begin{defn}
  \label{deftreillis}
  Soit $(E,\leq)$ un ensemble partiellement ordonné. On dit que $E$ est
  un treillis si :
  \begin{itemize}
  \item[(i)] $E$ est borné : $\exists a,b \in E,\ \forall x \in E,\ a \leq
    x \leq b$ ;
  \item[(ii)] pour tout $x,y\in E$, $x$ et $y$ admettent un infimum $x
    \wedge y$ ;
  \item[(iii)] pour tout $x,y\in E$, $x$ et $y$ admettent un supremum $x
    \vee y$.
  \end{itemize}
\end{defn}

Lorsque $E$ est fini, deux quelconques de ces propriétés entraînent la
troisième. 

\bigskip

Comme expliqué au paragraphe \ref{subsectgarside}, la propriété de treillis
pour $\NCP_W$ est essentielle pour montrer que le monoïde dual de tresses
de $W$ est un monoïde de Garside.

Pour l'ensemble $\NCP_n$ des partitions non-croisées d'un $n$-gone, on
vérifie aisément (i) et (ii). 

Par contre, dans le cas général, une méthode de preuve uniforme est
difficile à trouver. Historiquement, la démonstration a dans un premier
temps été faite au cas par cas, avant que Brady-Watt \cite{BWtreillis}
donnent une preuve générale --- mais seulement dans le cas où $W$ est réel
---, en construisant un complexe simplicial dont les simplexes représentent
les éléments de $\NCP_W$ (pour un exposé détaillé, voir le mémoire
\cite{M2}). Depuis, d'autres démonstrations uniformes indépendantes, mais
toujours dans le cas réel, ont été publiées (Ingalls-Thomas \cite{IngThom},
Reading \cite{readingshard}).

\medskip

Dans le cas où $W$ n'est pas réel, aucune démonstration sans cas par cas
n'a été obtenue à l'heure actuelle. Les méthodes citées ci-dessus ne se
généralisent pas facilement au cas complexe, car elles reposent de façon
essentielle sur la théorie de Coxeter.

\subsection{Sur la terminologie \og non-croisées \fg}

Historiquement, l'ensemble $\NCP_W$ est une généralisation algébrique des
partitions non-croisées d'un $n$-gone. On peut se demander s'il est
possible de voir effectivement les éléments de $\NCP_W$ comme certains
objets combinatoires sans croisement, ou plus généralement s'il existe une
façon naturelle de représenter géométriquement les éléments de cette
structure. C'est le cas pour les types $B$ et $D$ (Athanasiadis-Reiner
\cite{athareiner}) et $G(e,e,r)$ (Bessis-Corran \cite{BessisCorran}).

Un cas plus général (pour les séries réelles infinies et pour certains
groupes exceptionnels) est traité par les récents travaux de
Reading \cite{readingcoxeter}, qui donne des représentations graphiques des
sous-groupes non-croisés en utilisant le plan de Coxeter.

\subsection{Nombres de Catalan associés}
\label{subpartcatalan}

Le calcul du cardinal de $\NCP_W$ fait intervenir des généralisations des
nombres de Catalan, s'exprimant à partir des degrés invariants de
$W$. Rappelons la définition de ces degrés.

Soit $W\subseteq \GL(V)$ un groupe de réflexion (complexe ou réel) de rang
$n$ ; il agit naturellement sur l'algèbre symétrique $S(V^*) \simeq
\BC[v_1,\dots, v_n]$ (où $v_1,\dots, v_n$ est une base de $V$). En vertu du
théorème de Chevalley-Shephard-Todd, l'algèbre des invariants $S(V^*)^W$
est une algèbre de polynômes ; de plus, elle admet un système de $n$
générateurs homogènes $f_1,\dots, f_n$, et leurs degrés $d_1\leq \dots \leq
d_n$ ne dépendent pas du choix des invariants $f_i$. Ce sont les
\emph{degrés invariants} de $W$ ; le degré maximal $d_n$ est appelé nombre de Coxeter (noté $h$), et est
aussi l'ordre d'un élément de Coxeter.

\bigskip

Ces degrés invariants interviennent dans la combinatoire du treillis
$\NCP_W$, avec des formules souvent étonnamment simples. Ainsi, pour tout
groupe de réflexion complexe irréductible bien engendré $W$, on a l'égalité
suivante :
\begin{equation*} |\NCP_W|=\prod_{i=1}^n \frac{d_i + h}{d_i} \ .
  \tag{5} \end{equation*}

Aussi simple soit-elle, cette formule n'a jusqu'ici pas été démontrée de
manière uniforme : elle a simplement été vérifiée pour chaque groupe
irréductible de la classification. Le seul progrès vers une meilleure
explication est une formule de récurrence établie par Reading dans
\cite{reading} (valable dans le cas réel), et qui permet de vérifier
l'égalité au cas par cas de façon simple.

\medskip

Par extension du cas du type $A$ (où $|NCP_n|$ est le $n$-ième nombre de
Catalan $\frac1{n+1} \binom{2n}{n}$), le membre de droite de l'égalité (5)
est appelé \emph{nombre de Catalan de type $W$}, et noté  $\Cat(W)$.

\medskip

Ce nombre apparaît dans de nombreux autres contextes. Dans le contexte des
algèbres amassées (\og cluster algebras \fg), introduites par
Fomin-Zelevinsky dans \cite{FominZel}, il compte le nombre de sommets
dans l'associaèdre généralisé de type $W$, pour $W$ réel (voir
Fomin-Reading \cite{gcccc}). Quand $W$ est un groupe de Weyl, le nombre
$\Cat(W)$ intervient aussi dans la combinatoire des \og partitions
non-emboîtées \fg\ (``nonnested partitions''), dans celle des antichaînes
de racines, de l'arrangement de Shi étendu, des triangles de Chapoton
\cite{Atha,Atha2,Atha3,chapoton,chapoton2}... De nombreuses questions se
posent en lien avec ces objets combinatoires, car l'on a souvent seulement
une énumération au cas par cas, qui, de façon mystérieuse, donne $\Cat(W)$,
et très peu de preuves uniformes ou de bijections explicites entre les
différents objets comptés par le même nombre.

\medskip

La formule (5) n'est cependant qu'un cas particulier de la formule de
Chapoton sur le nombre de chaînes dans $\NCP_W$.

\subsection{Nombres de chaînes, formule de Chapoton}

\begin{defn}
  Pour $k \in \BN$, le $k$-ième \emph{nombre de Fuss-Catalan de type $W$}
  est défini par la formule suivante :
 \[ \Cat^{(k)}(W):= \prod_{i=1}^n \frac{d_i + kh}{d_i} \]
\end{defn}

\begin{thm}[Formule de Chapoton]
  Soit $W$ un groupe de réflexion complexe irréductible bien
  engendré. Alors le nombre de $k$-chaînes $w_1\< \dots w_k \< c$ dans
  $\NCP_W(c)$ est égal à $\Cat^{(k)}(W)$.
\end{thm}

Cette formule apparaît pour la première fois (dans le cas $W$ réel) dans
l'article de Chapoton \cite{chapoton} ; mais elle se vérifie au cas par
cas en utilisant de nombreux travaux (Edelman, Athanasiadis, Reiner,
Bessis...). On renvoie à \cite[3.5]{Armstrong} pour les détails
historiques.

\medskip

La plupart des remarques de la partie précédente s'appliquent encore ici :
les nombres de Fuss-Catalan de type $W$ apparaissent dans d'autres
contextes combinatoires, souvent sans qu'on en ait une explication
satisfaisante.

\bigskip

La formule de Chapoton a pour conséquence deux formules \og classiques \fg\
:
\begin{enumerate}[(i)]
\item $|\NCP_W| = \prod \frac{d_i + h}{d_i}$, comme décrit en
  \ref{subpartcatalan} ;
\item $|\Red_{\CR}(c)|= n!\ h^n /\ |W|$, où $\Red_{\CR}(c)$ désigne l'ensemble
  des décompositions réduites de l'élément de Coxeter $c$.
\end{enumerate}

Toutes deux avaient été observées avant la formule de Chapoton, également
avec des preuves reposant sur la classification. Historiquement, la formule
(ii) pour le cas réel a été conjecturée par Looijenga dans
\cite{looijenga}, puis démontrée au cas par cas par Deligne
\cite{deligneletter}. A ma connaissance, aucune autre preuve n'est connue à
ce jour. 

\medskip

La motivation initiale des travaux présentés ici était de progresser vers
une explication conceptuelle de la formule de Chapoton. On verra qu'au-delà
de cet enjeu qui pourrait sembler anecdotique, une compréhension profonde
de cette formule exige d'élucider une structure
algébro-géométrico-combinatoire riche et subtile.

L'un des principaux résultats obtenus est le calcul du nombre de
factorisations d'un élément de Coxeter (de longueur $n$) en $(n-1)$ blocs,
où un seul facteur est de longueur $2$, les autres étant des réflexions. La
preuve est de nature géométrique, et ne demande pas de nouvelle analyse au
cas par cas si l'on admet la formule (i) (cf. Remarque \ref{rqfinale}).

\subsection{Raffinements divers}

On peut munir l'ensemble des $k$-chaînes d'un ordre naturel, construit à
partir de l'ordre produit sur $(W,\<)$. Il est alors appelé
l'ensemble $NC^{(k)}(W)$ des partitions non-croisées $k$-divisibles de type
$W$ --- terminologie inspirée par le cas du type $A$, où
$NC^{(k)}(A_{n-1})$ est isomorphe à l'ensemble des partitions de
$\NCP_{kn}$ dont toutes les parts sont de cardinal multiple de $k$. Sa
combinatoire très riche fait l'objet des chapitres 3 et 4 de
\cite{Armstrong}. En particulier, on peut obtenir des formules encore plus
fines, où le rang du premier terme de la chaîne est spécifié (nombres de
Fuss-Narayana).

\bigskip

Signalons un autre raffinement de ces nombres, à base de $q$-analogues :
ainsi, Stump étudie dans \cite{stump} des $q,t$-nombres de Fuss-Catalan,
qui sont en rapport avec l'algèbre des coinvariants diagonaux (travaux de
Gordon \cite{gordon}, Gordon-Griffeth \cite{GorGri}).

Un autre problème combinatoire sur $\NCP_W$, lié aux $q$-analogues, est le
\og phénomène de crible cyclique \fg\ (``cyclic sieving phenomenon''
défini par Reiner-Stanton-White \cite{sieving}) : on renvoie aux travaux de
Bessis-Reiner \cite{bessisreiner} et Krattenthaler-Müller
\cite{KraMu2}.

\subsection{Factorisations d'un élément de Coxeter}
\label{subsectfactintro}

Si $w_1\< \dots \< w_k$ est une chaîne de $\NCP_W$, on peut construire une
factorisation $c=u_1\dots u_{k+1}$ en posant $u_i=w_{i-1}^{-1}w_{i}$ (avec
$w_0=1$ et $w_{k+1}=c$) ; on a alors $\ell(u_1) + \dots +
\ell(u_{k+1})=\ell(c)$. Inversement, à toute telle factorisation de $c$ on
peut associer une chaîne dans $\NCP_W$. Ce type de factorisation est une
extension naturelle de la notion de  décompositions réduites de $c$
($\Red_{\CR}(c)$).

Dans le chapitre \ref{chaphur} de cette thèse, on va construire
géométriquement des factorisations de $c$ (en passant par le groupe de
tresses $B(W)$), qui seront des factorisations strictes (aucun des facteurs
n'est trivial). Mais il existe des formules de passage simples entre les
factorisations strictes et les factorisations larges (ou entre les chaînes
strictes et les chaînes larges) : dans l'annexe \ref{annexchapo} on
détaille ces relations.

\medskip

La formule de Chapoton permet donc de calculer le nombre de
factorisations d'un élément de Coxeter en un nombre donné de facteurs. La
combinatoire de telles factorisations (qui, dans le cas du groupe
symétrique, est déjà très riche) offre de nombreux problèmes intéressants
(cf. les \og nombres de décompositions \fg\ de Krattenthaler-Müller
\cite{KraMu}).

\section{Géométrie du discriminant}
\label{sectgeom}

La plupart des travaux en combinatoire sur le treillis $\NCP_W$ concernent
le cas des groupes réels. Très souvent, les méthodes utilisées reposent sur
la théorie de Coxeter (par exemple sur l'existence d'un diagramme de
Coxeter qui admet un $2$-coloriage, ce qui permet de choisir un élément de
Coxeter bien adapté), et ne sont donc pas transposables aux groupes non
réels. 

En fait, la notion même d'élément de Coxeter dans un groupe de réflexion
complexe $W$ (non réel) n'a pas encore de définition combinatoire
satisfaisante. C'est plutôt une notion de nature géométrique (s'appliquant
aux valeurs propres et vecteurs propres de l'élément,
cf. Déf. \ref{defcox}), contrairement au cas où $W$ est un groupe de
Coxeter.

Pour ces raisons, si l'on veut comprendre de manière globable la
structure combinatoire de $\NCP_W$ pour l'ensemble des groupes de
réflexion complexes bien engendrés, on est amené à tenter une approche
géométrique. 

C'est l'approche qui a été menée dans ce travail de thèse.

\subsection{Le revêtement de Lyashko-Looijenga}

Le point de départ est l'utilisation d'un revêtement, construit à partir de
la géométrie de l'hypersurface $\CH$ associée au discriminant de $W$. Le
\emph{revêtement de Lyashko-Looijenga} ($\LL$) est introduit par Bessis
dans \cite{BessisKPi1}, généralisant une définition de Looijenga et Lyashko
(cf. partie \ref{partLL}). Décrivons grossièrement sa construction. On part
de l'espace quotient $W \qg V$, dont les fonctions coordonnées sont les
invariants $f_1,\dots, f_n$. On considère l'hypersurface $\CH$, image de
l'union des hyperplans par la surjection $V \ \surj \ W \qg V$. Le revêtement
$\LL$ permet d'étudier cette hypersurface via les fibres de la projection
$(f_1,\dots, f_n) \mapsto (f_1,\dots, f_{n-1})$. Il associe à chaque
$(n-1)$-uplet $(f_1,\dots, f_{n-1})$ le multi-ensemble des intersections
(avec multiplicités) de sa fibre avec $\CH$.

Le point de départ des travaux de \cite{BessisKPi1} est l'observation que
le degré du revêtement $\LL$ est égal au nombre de décompositions réduites
d'un élément de Coxeter. La relation entre la géométrie de
$\LL$ et les factorisations d'un élément de Coxeter dont on a parlé plus
haut va en fait beaucoup plus loin : en quelque sorte, ces factorisations
encodent naturellement les fibres de $\LL$. Les travaux de cette thèse
découlent d'un raffinement de ces observations.

\subsection{Travaux de Saito}

Signalons que ces questions sur la géométrie du discriminant sont liées aux
travaux de Kyoji Saito sur la structure de variété de Frobenius (``flat
structure'') de $W \qg V$
\cite{saito,saitoorbifold,saitopolyhedra,saitosemialg}. Même si ceux-ci
concernent les groupes réels, beaucoup de propriétés peuvent s'étendre sans
difficultés aux groupes complexes bien engendrés.

\section{Principaux résultats de la thèse}
\label{sectresult}

Comme on l'a dit plus haut, les travaux présentés ici étaient initialement
motivés par la recherche d'une explication conceptuelle de la formule de
Chapoton. Le point de départ a été l'utilisation des résultats récents de
Bessis \cite{BessisKPi1} : ses constructions topologiques (les \og tunnels
\fg) permettent en effet de caractériser les fibres du revêtement $\LL$ à
l'aide de certaines factorisations d'un élément de Coxeter $c$. La
combinatoire de ces factorisations est à son tour liée aux nombres de
chaînes dans $\NCP_W$ : on explique ces relations dans l'annexe
\ref{annexchapo}.

\subsection{Morphisme de Lyashko-Looijenga, factorisations primitives,
  action d'Hurwitz}

Soit $W$ un groupe de réflexion complexe bien engendré, irréductible, de
rang $n$. Notons $\Delta_W$ le discriminant de $W$ (cf. paragraphe
\ref{subpartLL0}) ; c'est un polynôme invariant, donc dans $\BC[f_1,\dots,
f_n]$ (où $f_1,\dots, f_n$ est un système d'invariants fondamentaux de
degrés $d_1\leq \dots \leq d_n$). On pose $Y:= \Spec \BC[f_1,\dots,
f_{n-1}]$, de sorte que le quotient $W \qg V$ s'identifie à $Y \times
\BC$. Le morphisme de Lyashko-Looijenga associe à $y\in Y$ l'ensemble des
points d'intersection (avec multiplicité) de la droite (complexe) $\{(y,x)
\tq x \in \BC \}$ avec l'hypersurface $\CH=\{\Delta_W=0 \}$ ; autrement
dit, $\LL(y)$ est le multi-ensemble des racines de $\Delta_W(y,x)$ vu comme
polynôme en $x$ (cf. Déf. \ref{defLL}). Si l'on préfère le voir comme
un morphisme algébrique explicite, $LL$ envoie $y$ sur les coefficients en
$x$ du polynôme $\Delta_W(y,x)$.

\medskip

Dans le chapitre \ref{chaphur} (qui est une reproduction de l'article
\cite{Ripollfacto}), on commence par raffiner les constructions de
\cite{BessisKPi1}. On définit, en passant par le groupe de tresses
$B(W)$, une application
\[ \fact : Y \to \Fact(c) \ ,\]
où $\Fact(c)$ est l'ensemble des \emph{factorisations par blocs} d'un
élément de Coxeter $c$ :
\[ \Fact(c) := \{(w_1,\dots, w_p) \in (W-\{1\})^p \tq w_1\dots
w_p = c,\text{ et } \ell(w_1) + \dots + \ell(w_p) = \ell(c) \} \ .\]
On obtient alors une reformulation d'un théorème fondamental de Bessis :

\begin{center} L'application $ \qquad \begin{array}{lcl}
    Y & \xrightarrow{\LL \times \fact} & E_n  \times_{\cp(n)}  \Fact(c)\\
    y & \mapsto &(\LL(y) \enspace ,\enspace \fact(y))
  \end{array}  \qquad $
  est bijective
\end{center}
(où  $E_n  \times_{\cp(n)}  \Fact(c)$ est un produit fibré au-dessus de
l'ensemble des compositions --- partitions ordonnées --- de $n$).

Les conséquences de ce théorème sont étudiées dans \cite{BessisKPi1}, mais
essentiellement sur la partie non ramifiée $Y'$ de $Y$ (\ie les $y$ tels
que le multi-ensemble $\LL(y)$ n'ait que des points simples) : la
restriction à cette partie est un revêtement non ramifié. En particulier,
le théorème implique que le degré de ce revêtement est égal au cardinal de
l'ensemble $\Red_\CR(c)$ des \emph{décompositions réduites} de $c$ (qui
sont les factorisations par blocs en $n$ réflexions). De plus, on peut
construire deux actions du groupe de tresses usuel $B_n$ :
\begin{itemize}
\item l'action de Galois (ou de monodromie) sur $Y'$, définie par le
  revêtement $\LL$ ;
\item l'action d'Hurwitz sur $\Red_\CR(c)$ (celle-ci conjugue les facteurs
  les uns avec les autres, cf. Déf. \ref{defhur}).
\end{itemize}
Via l'application $\fact$, ces deux actions sont compatibles.

\medskip

Dans le chapitre \ref{chaphur}, on se demande ce qu'il se passe lorsqu'on
n'est plus sur la partie non ramifiée. On traite particulièrement le cas où
le multi-ensemble $\LL(y)$ n'a qu'un seul point multiple (de multiplicité
quelconque) ; c'est aussi le cas où $\fact(y)$ contient un seul facteur de
longueur quelconque, tous les autres étant des réflexions. On parle alors
de \emph{factorisation primitive} de $c$.

On montre que le morphisme $\LL$ peut être vu comme un \og revêtement
ramifié stratifié \fg\ (Thm. \ref{thmrevetement}). Ensuite, on étudie les
orbites des factorisations primitives sous l'action d'Hurwitz : le groupe
$B_k$ agit naturellement sur l'ensemble des factorisations de $c$ en $k$
blocs, en conjuguant les facteurs les uns avec les autres. Il était déjà
connu que sur l'ensemble des décompositions réduites de $c$, l'action
d'Hurwitz est transitive. Le résultat principal du chapitre est le suivant
:

\begin{thm}[Thm. \ref{thmintro}]
  Soit $W$ un groupe de réflexion complexe fini, bien engendré. Soit $c$ un
  élément de Coxeter de $W$, et $F_u=(u,r_{p+1},\dots, r_n)$,
  $F_v=(v,r_{p+1}',\dots, r_n')$ deux factorisations primitives de $c$ (où
  les $r_i, r_i'$ sont des réflexions). Alors $F_u$ et $F_v$ sont dans la
  même orbite d'Hurwitz sous $B_n$ si et seulement si $u$ et $v$ sont
  conjugués dans $W$.
\end{thm}

La démonstration nécessite l'étude de la connexité de certaines strates de
l'hypersurface $\CH$ du discriminant.

Au passage, on approfondit la notion d'\emph{élément de Coxeter
  parabolique} dans les groupes bien engendrés, en montrant que les
propriétés classiques du cas réel s'étendent (Prop. \ref{propcoxpar} et
\ref{propcoxpar2}) :
\begin{itemize}
\item un élément de Coxeter parabolique peut être défini soit comme un élément
  de Coxeter d'un sous-groupe parabolique, soit comme un diviseur (pour
  $\<$) d'un élément de Coxeter de $W$ ;
\item le sous-groupe associé à un élément de Coxeter parabolique $w$ est
  engendré par toute factorisation de $w$.
\end{itemize}

\subsection{Extension d'anneaux de polynômes, Jacobiens, discriminants,
  groupes de réflexion virtuels}

Dans le chapitre \ref{chapjac}, on se place dans un cadre plus général, où
l'on étudie des questions d'algèbre commutative. Celles-ci ont au départ
été motivées par des constatations empiriques sur le morphisme de
Lyashko-Looijenga. En effet, pour pouvoir obtenir des formules explicites
sur la combinatoire des factorisations par blocs (voir partie suivante),
j'ai eu besoin d'examiner les degrés homogènes des composantes irréductibles
du lieu ramifié de $\LL$. Ce lieu ramifié est défini par le
\emph{discriminant de Lyashko-Looijenga} :
\[ D_{\LL}=\Disc(\Delta_W(f_1,\dots, f_n) \ ;\ f_n) \qquad \text{(polynôme
  de } \BC[f_1,\dots, f_{n-1}]\text{ ).}\] Or, en calculant dans quelques
exemples le polynôme $D_{\LL}$ ainsi que le Jacobien $J_{\LL}$ (déterminant
de la matrice jacobienne de $\LL$, en tant que morphisme algébrique), j'ai
observé qu'ils avaient une forme remarquable :
\[ D_{\LL} = \prod Q^{e_Q} \qquad ; \qquad J_{\LL} =\prod Q^{e_Q-1} \ ,\] 
où les facteurs irréductibles $Q$ sont les mêmes. J'ai cherché à
établir ces formules en toute généralité.

Ces propriétés évoquent le couple discriminant-Jacobien $(\Delta_W,J_W)$ issu de
la théorie des invariants d'un groupe de réflexion complexe $W$, qui est de la
forme
\[ \Delta_W = \prod_{H\in \CA} \alpha_H^{e_H} \qquad ; \qquad J_W
=\prod_{H\in \CA} \alpha_H^{e_H-1} \ .\] 

Mais la démonstration de ces propriétés n'est pas du tout transposable au cas de
$\LL$, car elle repose sur l'utilisation de l'action du groupe $W$.

\medskip

Une question naturelle est donc de comprendre dans quelle mesure ce type
de situation est général. On se place dans ce chapitre le cadre d'une
\emph{extension polynomiale finie graduée} (cf. Déf. \ref{defext}) :
\[ A=\BC[Y_1,\dots, Y_n] \subseteq \BC[X_1,\dots, X_n]=B \ ,\] où les
variables $X_i$, $Y_j$ sont homogènes pondérées, et où $B$ est fini sur $A$
en tant qu'anneau. Ainsi, les résultats obtenus s'appliquent à la fois :
\begin{itemize}
\item à une extension galoisienne, de la forme $\BC[V]^W \subseteq \BC[V]$
  ; dans ce cas, d'après le théorème de Chevalley-Shephard-Todd, $W$ est un
  groupe de réflexion complexe ;
\item à l'extension définie par un morphisme de Lyashko-Looijenga.
\end{itemize}

\medskip

Pour une telle extension, on commence par donner la factorisation en
irréductibles du Jacobien, en fonction des idéaux ramifiés de l'extension
et de leurs indices de ramification. Pour pouvoir définir un \og
discriminant \fg\ de l'extension, qui se comporte comme $\Delta_W$ ou
$D_{\LL}$ dans les cas précités, on introduit la notion d'\emph{extension
  bien ramifiée}. C'est une notion plus faible que celle d'extension
galoisienne, où pour chaque idéal premier $\fp$ de hauteur $1$ de $A$, les
idéaux de $B$ au-dessus de $\fp$ sont soit tous ramifiés, soit tous non
ramifiés (cf. Déf. \ref{defwellram}). On peut donner de nombreuses
caractérisations de cette propriété (Prop. \ref{propwell}).

Le résultat principal du chapitre \ref{chapjac} est le théorème suivant :

\begin{thm}[Thm. \ref{thmintrojac}]
Soit $A \subseteq B$ une extension polynomiale finie
graduée. Alors :
\begin{itemize}
\item le Jacobien $J_{B/A}$ de l'extension vérifie :
  \[ J = \prod_{Q \in \Spram(B)} Q^{e_Q-1} \quad \text{(à un scalaire
    multiplicatif près),}\]

  où $\Spram(B)$
  désigne l'ensemble des polynômes ramifiés de $B$ (à association près), et
  $e_Q$ l'indice de ramification de $Q$ ;
\item si l'extension est bien ramifiée, alors :
\[ (J)\cap A = \left(\prod_{Q \in \Spram(B)} Q^{e_Q}\right) \qquad \text{(en
  tant qu'idéal de } A \text{) .}\]
\end{itemize}
\end{thm}

Ainsi, dans le cas où l'extension est bien ramifiée, le polynôme $\prod
Q^{e_Q}$ peut être appelé \og discriminant \fg\ de l'extension, par
analogie avec la situation galoisienne, où il s'agit du discriminant du
groupe de réflexion associé. 

\bigskip

\begin{rqe}
Les résultats de ce chapitre sont des conséquences assez élémentaires
de propriétés d'algèbre commutative, mais je n'ai pas trouvé de
références satisfaisantes dans la littérature. Souvent, le point de
vue est soit beaucoup trop général (algèbre commutative des extensions
d'anneaux) soit trop spécifique (extensions galoisiennes). Ici, on se
place dans un cadre intermédiaire, où les preuves utilisant une action
de groupe (c'est-à-dire toute la théorie des invariants des groupes de
réflexion) ne s'appliquent pas, et où les résultats généraux
concernant les extensions d'anneaux se simplifient drastiquement.
\end{rqe}

\begin{rqe}
  Le vocabulaire et les propriétés énoncés dans ce chapitre sont utilisés
  dans le cas des extensions de Lyashko-Looijenga dans le chapitre
  suivant. Cependant, dans une optique plus générale, la situation décrite
  semble prometteuse, car elle offre une généralisation intrigante de la
  théorie des invariants d'un groupe de réflexion. Même s'il n'y a pas
  d'action de groupe, on a ainsi un Jacobien et un discriminant, qui se
  comportent de la même manière que dans le cas galoisien. Une telle
  extension peut ainsi être appelée \emph{groupe de réflexion
    virtuel}\footnote{terminologie proposée par David Bessis.} ; il serait
  intéressant de voir dans quelle mesure les analogies peuvent se
  poursuivre.
\end{rqe}

\subsection{Le morphisme de Lyashko-Looijenga vu comme groupe de réflexion
  virtuel ; factorisations sous-maximales}

Dans le chapitre \ref{chapsubmax}, on applique les propriétés données au
chapitre précédent, et on montre que les extensions de Lyashko-Looijenga
sont bien ramifiées. 

En utilisant les résultats du chapitre \ref{chaphur}, on peut décrire les
facteurs irréductibles du discriminant $D_{\LL}=\Disc(\Delta_W(f_1,\dots,
f_n) \ ;\ f_n) $ : il y en a un par classe de conjugaison de sous-groupes
paraboliques de rang $2$ de $W$ (ou d'élements de Coxeter paraboliques de
longueur $2$). On peut ainsi écrire
\[ D_{\LL} = \prod_{\Lambda \in \Lb_2} D_{\Lambda}^{r_{\Lambda}}\ , \] où
$\Lb_2$ désigne l'ensemble des classes de conjugaison de paraboliques de
rang $2$.

On obtient le théorème suivant :

\begin{thm}[Thm. \ref{thmLL}]
  Soit $A \subseteq B$ l'extension associée au morphisme de
  Lyashko-Looijenga $\LL$ d'un groupe de réflexion complexe, bien engendré,
  irréductible. Pour $\Lambda \in \Lb_2$, soit $w$ un élément de Coxeter
  parabolique de $W$, qui correspond à $\Lambda$.

  Alors, $r_\Lambda$ est le nombre de décompositions réduites de $w$ en
  deux réflexions. De plus :
  \begin{enumerate}[(a)]
  \item les polynômes ramifiés de l'extension sont les $D_{\Lambda}$ ;
  \item le Jacobien $J_{\LL}$ de l'extension est égal (à un scalaire près)
    à :
    \[ \prod_{\Lambda\in \Lb_2 } D_{\Lambda}^{r_{\Lambda} -1} \ ; \]
\item le polynôme $D_{\LL}$ engendre l'idéal $(J_{\LL})\cap A$.
\end{enumerate}
\end{thm}

\begin{rqe}
  La formule (b) pour le Jacobien avait déjà été observée (dans le cas $W$
  réel) par K. Saito (formule 2.2.3 dans \cite{saitopolyhedra}) ; mais sa
  démonstration était au cas par cas. Ici on donne une preuve géométrique
  générale, qui repose sur les liens entre $\LL$ et l'application $\fact$.
\end{rqe}

\medskip

D'autre part, en raffinant des propriétés du chapitre \ref{chaphur}, on
parvient à calculer le nombre de factorisations sous-maximales d'un élément
de Coxeter de type $\Lambda$ --- \ie, constituées de $(n-2)$ réflexions et
un facteur de longueur $2$, ce dernier étant dans la classe de conjugaison
$\Lambda$ :

\begin{thm}[Thm. \ref{thmfact}]
Soit $\Lambda \in \Lb_2$, et $c$ un élément de Coxeter de $W$. Alors le
nombre de factorisations sous-maximales de $c$ de type $\Lambda$ est :
\[ |\Fact_{n-1}^{\Lambda}(c)| = \frac{(n-1)!\ h^{n-1}}{|W|} \deg
    D_\Lambda \ .\]
\end{thm}

Ce théorème fait donc le lien entre la combinatoire des factorisations par
blocs d'un élément de Coxeter, et la géométrie de $W \qg V$ : les $\deg
D_{\Lambda}$ sont en fait les degrés des projetés sur $\Spec
\BC[f_1,\dots, f_{n-1}]$ des orbites de plats de codimension $2$ de $V$.

\medskip

Dans l'annexe \ref{annexfactodisc}, on calcule explicitement les valeurs
associées à cette formule pour tous les groupes irréductibles bien
engendrés. On retrouve ainsi des formules de \cite{KraMu} pour les cas
réels ; dans les autres cas, les résultats sont nouveaux.

\bigskip

Les deux théorèmes précédents permettent de calculer le nombre total de
factorisations sous-maximales :

\begin{thm}[\ref{thmsubmax}]
  Soit $W$ un groupe de réflexion complexe, bien engendré, irréductible, et
  $d_1\leq\dots \leq d_n=h$ ses degrés invariants. Alors, le nombre de
  factorisations d'un élément de Coxeter en $(n-1)$ blocs est :
    \[ |\Fact_{n-1}(c)| = \frac{(n-1)!\ h^{n-1}}{|W|}
\left( \frac{(n-1)(n-2)}{2}h + \sum_{i=1}^{n-1} d_i \right) .\] 
\end{thm}

\begin{rqe}
\label{rqfinale}
  Cette égalité est en réalité une conséquence logique de la formule de
  Chapoton. En effet, comme expliqué au paragraphe \ref{subsectfactintro},
  on peut passer des nombres de chaînes aux nombres de factorisations par
  blocs (voir Annexe \ref{annexchapo}).

  Cependant, on en donne ici une preuve directe, alors que la formule de
  Chapoton n'est connue qu'au cas par cas. En fait, dans toute la thèse
  (exceptée l'annexe \ref{annexfactodisc} sur les calculs explicites de
  discriminant), on n'utilise jamais la classification des groupes de
  réflexion. On se repose tout de même sur des propriétés pour lesquelles
  il n'y a pas encore de démonstrations générales (en particulier les
  premières propriétés de $\LL$ données dans \cite{BessisKPi1}) ; mais, au
  final, on parvient ici à passer, à l'aide d'une preuve géométrique sans
  cas par cas, de la connaissance de $|\Red_\CR(c)|$ à celle de
  $|\Fact_{n-1}(c)|$.
\end{rqe}

% Fin intro.
%%%%%%%%%%%%%%%%%%%%%%%%%%%%%%%%%%%%%%%%%%%%%%%%%%%%%%%%%%%%%%%%%%%%%%%

% Orbites d'Hurwitz, 1er chapitre 
%%%%%%%%%%%%%%%%%%%%%%%%%%%%%%%%%%%%%%%%%%%%%%%%%%%%%%%
% Chapitre de thèse, à partir de l'article de J. Alg.
% à compiler avec les en-têtes de manuscrit.tex
%%%%%%%%%%%%%%%%%%%%%%%%%%%%%%%%%%%%%%%%%%%%%%%%%%%%%%%

\chapter[Orbites d'Hurwitz des factorisations primitives d'un élément de
Coxeter]{Orbites d'Hurwitz des factorisations primitives d'un élément de
  Coxeter\protect\footnote{Ce chapitre reproduit intégralement le texte d'un article publié dans \emph{Journal of
      Algebra}, \textbf{323} (5), mars 2010.}}
\label{chaphur}

\noindent \textsc{Résumé}. On considère l'action d'Hurwitz du groupe de
tresses usuel sur les factorisations d'un élément de Coxeter $c$ d'un
groupe de réflexions complexe bien engendré $W$. Il est bien connu que
l'action d'Hurwitz est transitive sur l'ensemble des décompositions
réduites de $c$ en réflexions. On démontre ici une propriété similaire pour
les factorisations primitives de $c$, \ie celles dont tous les facteurs
sauf un sont des réflexions. Cette étude est motivée par la recherche d'une
explication géométrique de la formule de Chapoton sur le nombre de chaînes
de longueur donnée dans le treillis des partitions non croisées
$\NCP_W$. La démonstration présentée repose sur les propriétés du
revêtement de Lyashko-Looijenga et sur la géométrie du discriminant de $W$.

  \textbf{Mots-clés :} action d'Hurwitz, groupe de réflexions complexe, élément
  de Coxeter, treillis des partitions non croisées, revêtement de
  Lyashko-Looijenga.
  
  \medskip
  \selectlanguage{english}
  \noindent \textsc{Abstract}. We study the Hurwitz action of the classical
  braid group on factorisations of a Coxeter element $c$ in a
  well-generated complex reflection group $W$. It is well known that the
  Hurwitz action is transitive on the set of reduced decompositions of $c$
  in reflections. Our main result is a similar property for the primitive
  factorisations of $c$, \ie factorisations with only one factor which is
  not a reflection. The motivation is the search for a geometric proof of
  Chapoton's formula for the number of chains of given length in the
  non-crossing partitions lattice $\NCP_W$. Our proof uses the properties
  of the Lyashko-Looijenga covering and the geometry of the discriminant of
  $W$.

  \textbf{Keywords:} Hurwitz action, complex reflection group, Coxeter
  element, non-crossing partitions lattice, Lyashko-Looijenga covering.
 \selectlanguage{francais}

\newpage

\section*{Introduction}

Soit $V$ un $\BC$-espace vectoriel de dimension finie, et $W \subseteq
\GL(V)$ un groupe de réflexions complexe fini bien engendré (définition
en partie \ref{subpartcox}). On note $\CR$ l'ensemble de
toutes les réflexions de $W$.  On définit la longueur d'un élément $w$ de $W$ :
\[ \ell(w) := \min \{k \in \BN \tq \exists r_1,\dots, r_k \in \CR,\
w=r_1\dots r_k \} .\]
On note $\Red (w)$ l'ensemble des décompositions réduites de $w$, \ie des
mots de longueur $\ell(w)$ représentant $w$. Ce sont donc des factorisations
minimales de $w$ en réflexions. Plus généralement, on peut
définir des \emph{factorisations par blocs} de $w$ :

\begin{defn}
\label{deffact}
Soit $w\in W$, de longueur $n \geq 1$. Soit $1 \leq p \leq n$. Une
\emph{factorisation en $p$ blocs} de $w$ est un $p$-uplet $(u_1,
\dots, u_p)$ d'éléments de $W - \{1\}$ tels que $w=u_1\dots u_p$, avec
$\ell(w)=\ell(u_1) +\dots + \ell(u_p)$. On note $\Fact_p(w)$ l'ensemble de
ces factorisations.
\end{defn}

Si $w$ est de longueur $n$, on a donc
$\Fact_n(w)=\Red (w)$.

\medskip

Ci-dessous on définit l'action d'Hurwitz du groupe de tresses $B_p$ sur les
factorisations en $p$ blocs d'un élément de $W$. Une factorisation est dite
\emph{primitive} lorsqu'elle ne comporte qu'un seul facteur qui n'est pas
une réflexion (définition \ref{deffactprim}). Dans cet article on détermine les
orbites sous l'action d'Hurwitz des factorisations primitives d'un élément
de Coxeter de $W$ (définition \ref{defcox}).

\subsection*{Action d'Hurwitz}
~ 

Considérons un groupe $G$, et un entier $n\geq 1$. On fait agir le
groupe de tresses à $n$ brins, noté $B_n$, sur le produit cartésien $G^n$ :

\begin{defn}[Action d'Hurwitz]
\label{defhur}
  Soient $\bm{\sigma_1},\dots, \bm{\sigma_{n-1}}$ les générateurs standards de
  $B_n$. L'action d'Hurwitz (à droite) de $B_n$ sur $G^n$ est
  définie par :
  \[ (g_1,\dots, g_n) \cdot \bm{\sigma_i} \ := \ (g_1, \dots, g_{i-1},\
  g_{i+1},\  g_{i+1} ^{-1} g_i g_{i+1},\  g_{i+2}, \dots, g_n) \] pour tout
  $(g_1,\dots, g_n) \in G^n$ et tout $i\in \{1,\dots,n-1\}$.
\end{defn}

On vérifie aisément que cette définition donne bien une action du groupe de
tresses. Cela vient essentiellement du fait que la conjugaison fait partie
de la classe des opérations auto-distributives :
$f*(g*h)=(f*g)*(f*h)$ (ici $f*g=f^{-1}gf$). L'intérêt de ces opérations a
été décrit par Brieskorn (ensembles automorphes \cite{Brieskorn1}) et plus
récemment par Dehornoy (LD-systèmes \cite{LD}).

\subsection*{Factorisations primitives}
~ 

L'action d'Hurwitz préserve les fibres de l'application produit $(g_1,
\dots, g_n) \mapsto g_1\dots g_n $. Ainsi, si l'on revient au cadre décrit
plus haut, où $W$ est un groupe de réflexions complexe muni de sa partie
génératrice naturelle $\CR$, on obtient une action d'Hurwitz du groupe de
tresses $B_p$ sur l'ensemble des factorisations en $p$ blocs de $w$
(définition \ref{deffact}), pour tout $w\in W$ et $1\leq p \leq \ell(w)$.

Considérons $(u_1,\dots,u_p) \in \Fact_p(w)$. Le $p$-uplet $(\ell(u_1),\dots,
\ell(u_p))$ définit donc une composition (\ie une partition ordonnée) de
l'entier $n=\ell(w)$. Si $\mu$ est une composition de $n$ en $p$ termes, on note
$\Fact_\mu(w)$ l'ensemble des factorisations dont les longueurs des
facteurs forment la composition $\mu$.

Si l'on fait agir $B_p$ sur un élément de $\Fact_\mu(w)$, la composition de
$n$ associée peut être modifiée, mais pas la partition de $n$ (non
ordonnée) sous-jacente. En effet, l'ensemble $\CR$ des réflexions étant
invariant par conjugaison, la longueur $\ell$ est également invariante par
conjugaison.

On dit qu'une factorisation en $p$ blocs est de forme $\lambda$, pour
$\lambda$ partition de $n$ (notation :
$\lambda \vdash n$), si le $p$-uplet non ordonné des longueurs des
facteurs constitue la partition $\lambda$. On note $\Fact_\lambda (w)$
l'ensemble des factorisations de forme $\lambda$. Ainsi, l'action
d'Hurwitz de $B_p$ sur $\Fact_p (w)$ stabilise les ensembles
$\Fact_\lambda (w)$, pour $\lambda \vdash n$ et $\#\lambda =p$
($\#\lambda$ désignant le nombre de parts de la partition $\lambda$).

Pour $\lambda \vdash n$, on note classiquement $\lambda=k_1^{p_1}\dots
k_r^{p_r}$ (où $k_1 > \dots > k_r \geq1$ et $p_i \geq 1$), si la partition
est constituée de $p_1$ fois l'entier $k_1$, $\dots$, $p_r$ fois l'entier
$k_r$.

\begin{defn}
  \label{deffactprim}
  Une \emph{factorisation primitive} de $w$ est une factorisation par blocs
  de $w$, de forme $k^1 1^{n-k}\vdash n$, pour $k \geq 2$ : tous les blocs
  sauf un doivent être des réflexions. L'unique facteur de longueur
  strictement supérieure à $1$ est appelé \emph{facteur long} de la
  factorisation primitive.
\end{defn}

On étudie ici les factorisations d'un \emph{élément de Coxeter}, au sens de
\cite{BessisKPi1} (voir définition \ref{defcox}). Dans le cas réel, on
retrouve la notion classique d'élément de Coxeter. Dans la suite de
l'introduction, on précisera l'intérêt fondamental des éléments de Coxeter
dans certaines constructions algébriques relatives à un groupe de
réflexions complexe bien engendré.

Considérons un élément de Coxeter $c$ de $W$, et notons $n$ sa
longueur. Le cas le plus élémentaire des factorisations de forme $1^n
\vdash n$, \ie des décompositions réduites de $c$, est connu : on
sait que l'action d'Hurwitz sur $\Red(c)=\Fact_{1^n}(c)$ est
transitive. Dans un premier temps cette propriété a été prouvée pour les
groupes réels (\cite{deligneletter}, puis \cite[Prop.\ 1.6.1]{Bessisdual}),
avec une preuve générale, puis vérifiée au cas par cas pour le reste des
groupes complexes \cite[Prop.\ 7.5]{BessisKPi1}.

Le cas que nous traitons ici peut se voir comme l'étape suivante,
c'est-à-dire la détermination des orbites d'Hurwitz des factorisations
primitives de $c$. Le résultat principal de cet article est le suivant :

\begin{thm}[Orbites d'Hurwitz primitives]
\label{thmintro}
  Soit $W$ un groupe de réflexions complexe fini, bien engendré. Soit $c$
  un élément de Coxeter de $W$. Alors, deux factorisations primitives de
  $c$ sont dans la même orbite d'Hurwitz si et seulement si leurs facteurs
  longs sont conjugués.
\end{thm}

L'idée de la preuve est d'utiliser certaines constructions de Bessis
\cite{BessisKPi1} pour obtenir les factorisations de $c$ à partir de la
géométrie du discriminant de $W$ dans l'espace-quotient $W \qg V$ (qui
s'identifie à $\BC^n$ si on a choisi des invariants fondamentaux $f_1,
\dots, f_n$). On interprète ensuite l'action d'Hurwitz en termes
géométriques.

Dans une première partie on rappelle les définitions classiques et les
propriétés de l'ordre de divisibilité dans le treillis des partitions non
croisées généralisées, noté $\NCP_W(c)$. La partie \ref{partfconj} présente
une notion de forte conjugaison pour les éléments de $\NCP_W(c)$, qui
permet de formuler une version plus forte du théorème principal \ref{thmintro}.

Par la suite, on considère l'hypersurface $\CH$ de $\BC^n$, d'équation le
discriminant de $W$, qui admet une stratification naturelle (par les
orbites de plats). On projette cette hypersurface sur $\BC^{n-1}$ (en
oubliant la coordonnée correspondant à l'invariant de plus haut degré
$f_n$), et on en déduit une stratification du lieu de bifurcation $\CK$ de
$\CH$. En partie \ref{partLL} on rappelle les propriétés du revêtement de
Lyashko-Looijenga $\LL$ , puis dans la partie suivante on modifie légèrement
les constructions de \cite{BessisKPi1}, en définissant une application
$\fact$ qui produit des factorisations associées à la géométrie de
$\CH$. Dans la partie \ref{partLLstrat}, on démontre des propriétés plus
fines du morphisme $\LL$, qui apparaît alors comme un revêtement ramifié
«~stratifié~». On en déduit que l'ensemble des orbites d'Hurwitz des
factorisations primitives est en bijection avec l'ensemble des composantes
connexes par arcs de certains ouverts des strates de $\CK$. La
démonstration du théorème revient ainsi à un problème de connexité dans
$\CK$ (traité en parties \ref{partecp} et \ref{partirred}), et est achevée
en partie \ref{partrefl} par l'étude du cas particulier des réflexions.

Dans la partie \ref{partecp}, on détaille les propriétés de la
stratification de $\CH$. À cette occasion, des conséquences remarquables de
certaines propriétés énoncées dans \cite{BessisKPi1} sont exposées, en
particulier concernant les sous-groupes paraboliques de $W$, leurs classes
de conjugaison et leurs éléments de Coxeter.

\bigskip

La suite de cette introduction présente les motivations du problème.

\subsection*{Treillis des partitions non croisées, monoïde de
  tresses dual}
~ 

Fixons un élément de Coxeter $c$ dans $W$. On note $\<$ l'ordre
partiel sur $W$ associé à la $R$-longueur $\ell$ (voir définition
\ref{defordre}) :
\[u \< v \mathrm{\ si\ et\ seulement\ si\ } \ell (u) + \ell (u  ^{-1} v) = \ell
(v) .\]

On considère l'ensemble des diviseurs de $c$ : $\NCP_W(c):= \{ w \in W \tq
w\< c \} $ (la notation $\NCP$ signifie «~non-crossing partitions~»,
cf. interprétation ci-dessous).

L'ensemble partiellement ordonné $(\NCP_W (c), \<)$ permet de construire
les générateurs (appelés les \emph{simples}) du \emph{monoïde de tresses
  dual} de $W$ (cf. \cite{BKL} pour le type $A$, \cite{Bessisdual} et
\cite{BWordre2} pour les groupes réels, puis \cite{BessisKPi1} pour le cas
général).

\medskip

L'étude de la combinatoire de $\NCP_W(c)$ est au départ motivée par les
conséquences sur les propriétés algébriques du monoïde dual. Ainsi, une des
propriétés remarquables de $\NCP_W(c)$ est qu'il forme un treillis pour l'ordre
$\<$ : preuve générale dans le cas réel par Brady et Watt \cite{BWtreillis}
(voir aussi \cite{IngThom}), le reste au cas par cas, cf. \cite[Lemme
8.9]{BessisKPi1}. C'est essentiellement cette propriété de treillis qui
donne au monoïde de tresses dual sa structure de monoïde de Garside (comme
défini dans \cite{Deh2}).

Mais la richesse de la structure combinatoire de $\NCP_W(c)$ en a fait plus
récemment un objet d'étude en soi (voir par exemple \cite{Armstrong}). Dans
le cas du groupe $W(A_{n-1})$, ce treillis correspond à l'ensemble des
partitions non croisées d'un $n$-gone, et est à la base de la construction
du monoïde de Birman-Ko-Lee \cite{BKL,BDM}. Pour les groupes de type
$G(e,e,r)$, on peut également interpréter l'ensemble des diviseurs de $c$
comme certains types de partitions non croisées \cite{BessisCorran}. Par
extension, l'ensemble $\NCP_W(c)$ des diviseurs d'un élément de Coxeter
dans un groupe bien engendré est appelé le \emph{treillis des partitions
  non croisées généralisées} (d'où la notation $\NCP$).

\subsection*{\texorpdfstring{Formule de Chapoton pour les chaînes de $\NCP_W(c)$ et
  factorisations de $c$}{Formule de Chapoton pour les chaînes de NCPW(c) et
  factorisations de c}}
~ 

Si $W$ n'est pas irréductible et s'écrit $W_1 \times W_2$, alors le
treillis $\NCP_W(c)$ est le produit des treillis $\NCP_{W_1}(c_1)$ et
$\NCP_{W_2}(c_2)$ (où $c_i$ désigne la composante de $c$ sur $W_i$), muni
de l'ordre produit. Dans cette partie on supposera que $W$ est
irréductible.

Notons $d_1 \leq \dots \leq d_n = h$ les degrés invariants de $W$. Une
formule, énoncée initialement par Chapoton, exprime le nombre de chaînes de
longueur donnée dans le treillis $\NCP_W(c)$ en fonction de ces
degrés. Notons $Z_W$ le polynôme \emph{Zêta} de $\NCP_W(c)$ : c'est la
fonction sur $\BN$ telle que pour tout $N$, $Z_W(N)$ est le nombre de
chaînes larges (ou multi-chaînes) $w_1 \< \dots \< w_{N-1}$ dans
$(\NCP_W(c),\<)$ (de façon générale $Z_W$ est toujours un polynôme : voir
par exemple Stanley \cite[Chap. 3.11]{stanley}).

\begin{prop}
\label{propchapoton}
Soit $W$ un groupe de réflexions complexe bien engendré,
irréductible. Alors, avec les notations ci-dessus, on a :
\[ Z_W (N) = \prod_{i=1}^n \frac{d_i + (N-1)h}{d_i} .\] 
\end{prop}

Dans le cas des groupes réels, cette formule a été observée par Chapoton
\cite[Prop.\ 9]{chapoton}, utilisant des calculs de Reiner et Athanasiadis
\cite{reiner,athareiner}. La généralisation aux groupes complexes est
énoncée par Bessis \cite[Prop.\ 13.1]{BessisKPi1} et utilise des résultats
de Bessis et Corran \cite{BessisCorran}.  La preuve est au cas par cas :
méthodes \emph{ad hoc} pour les séries infinies, et logiciel GAP pour les
types exceptionnels. Pour le cas réel, une démonstration plus directe
(utilisant une formule de récurrence générale qui permet de faire la
vérification au cas par cas plus simplement) a été récemment publiée par
Reading \cite{reading}.  L'apparente simplicité de la formule de Chapoton
motive toujours la recherche d'une explication complètement générale.

En posant $N=2$ dans l'expression de $Z_W(N)$ on obtient naturellement le
cardinal de $\NCP_W(c)$ :
\[ \left| \NCP_W (c) \right| = \prod_{i=1}^n \frac{d_i + h}{d_i}. \] Même
pour cette formule, la seule preuve connue à l'heure actuelle est au cas
par cas. L'entier $ \left| \NCP_W (c) \right|$, parfois noté $\Cat(W)$, est
appelé \emph{nombre de Catalan de type $W$} : pour $W=A_{n-1}$ on obtient
en effet le nombre de Catalan classique $\frac{1}{n+1} \binom{2n}{n}$.
 
\medskip

Les factorisations de $c$ sont bien évidemment reliées aux chaînes de
$\NCP_W(c)$. Partant d'une chaîne $w_1 \< ... \< w_p \< c$, on peut obtenir
une factorisation par blocs $c=u_1 \dots u_{p+1}$ en posant $u_1 = w_1$,
$u_i = w_{i-1}^{-1} w_i$ pour $ 2\leq i \leq p$, et $u_{p+1}=w_p ^{-1}
c$. Inversement, à partir d'une telle factorisation de $c$ on peut obtenir
une chaîne de $\NCP_W(c)$. Ainsi, dénombrer les chaînes dans $\NCP_W(c)$
revient à compter le nombre de \emph{factorisations par blocs} de $c$ de
tailles fixées.

Toutefois, on a choisi ici de travailler sur des factorisations
«~strictes~» (pas de facteurs triviaux), alors que la formule de Chapoton
concerne les chaînes larges. Ce choix est motivé par le fait que ce sont
des factorisations strictes qui apparaissent géométriquement via
l'application «~$\fact$~» (construite en partie \ref{partlbl}). Les nombres de
factorisations strictes et les nombres de chaînes larges sont alors liés
par la formule suivante (on utilise les formules de passage entre nombres
de chaînes strictes et nombres de chaînes larges, données par exemple dans
\cite[Chap. 3.11]{stanley}) :

\[ Z_W (N) = \sum_{k\geq 1} \left| \Fact_k(c) \right| \binom{N}{k} . \]

Une piste pour une démonstration générale de la formule de Chapoton est donc
l'étude de la combinatoire des factorisations par blocs.

\subsection*{\texorpdfstring{Remerciements.}{Remerciements}}
Je tiens à remercier vivement David Bessis pour ses conseils et son aide
sur de nombreux points de cet article. Merci également à Emmanuel Lepage
dont une remarque judicieuse a permis de simplifier la partie
\ref{subpartcpa}.

\section{\texorpdfstring{Ordre de divisibilité dans $\NCP_W$}{Ordre de divisibilité dans NCPW}}
\label{partordre}

\subsection{Éléments de Coxeter}
\label{subpartcox}
~ 

Soit $V$ un espace vectoriel complexe de dimension $n\geq 1$, et $W
\subseteq \GL(V)$ un groupe de réflexions complexe. On suppose que $W$ est
essentiel, \ie que $V^W=\{0\}$. Dans ce cas, on dit que $W$ est un groupe
de réflexions \emph{bien engendré} si un ensemble de $n$ réflexions suffit
à l'engendrer. Désormais on supposera toujours que $W$ est un groupe de
réflexions complexe bien engendré.

Soit $\alpha$ une racine de l'unité. On dit qu'un élément de $W$ est
\emph{$\alpha$-régulier} (au sens de Springer \cite{Springer}) s'il possède
un vecteur propre régulier (\ie n'appartenant à aucun des hyperplans de
réflexions) pour la valeur propre $\alpha$. Si $W$ est irréductible, notons
$d_1 \leq \dots \leq d_n$ les degrés invariants de $W$. On note $h$ et on
appelle \emph{nombre de Coxeter} le plus grand degré $d_n$. D'après la
théorie de Springer (\cite{Springer}, résumé dans
\cite[Thm.\ 1.9]{BessisKPi1}), il existe dans $W$ des éléments
$\eh$-réguliers. Dans ce cadre, en généralisant le cas réel, on pose la
définition suivante.

\begin{defn}
  \label{defcox}
  Soit $W$ un groupe de réflexions complexe fini bien engendré. Si $W$ est
  irréductible, un \emph{élément de Coxeter} de $W$ est un élément
  $\eh$-régulier. Dans le cas général, on appelle
  \emph{élément de Coxeter} un produit d'éléments de Coxeter des
  composantes irréductibles de $W$.
\end{defn} 

\subsection{Ordre de divisibilité dans $W$}
\label{subpartordre}
~ 

On définit ci-dessous l'ordre de divisibilité à gauche dans $W$ (parfois
appelé ordre absolu). Rappelons que pour $w \in W$, $\ell(w)$ désigne la
longueur d'une décomposition réduite de $w$ en réflexions de $W$.

\begin{defn}
  \label{defordre}
  Soient $u,v \in W$. On pose :
  \[u \< v \mathrm{\ si\ et\ seulement\ si\ } \ell (u) + \ell (u  ^{-1} v) = \ell
(v)\ .\]
\end{defn}

La relation $\<$ munit $W$ d'une structure d'ensemble partiellement ordonné
: $u$ divise~$v$ si et seulement si $u$ peut s'écrire comme un préfixe
d'une décomposition réduite de $v$. Comme $\CR$ est invariant par
conjugaison, on a aussi $u \< v$ si et seulement si $u$ est un suffixe (ou
même un sous-facteur) d'une décomposition réduite de $v$. Ainsi, l'ordre
$\<$ coïncide avec l'ordre de divisibilité à droite dans $W$.

\medskip

Considérons un élément de Coxeter $c$ de $W$. On note $\NCP_W(c)$
l'ensemble des diviseurs de $c$ dans $W$. D'après la théorie de Springer,
tous les éléments de Coxeter de $W$ sont conjugués. D'autre part, l'ordre
$\<$ est invariant par conjugaison, donc si $c'=aca^{-1}$, alors
$\NCP_W(c')=a\NCP_W(c)a^{-1}$. Ainsi, la structure de $\NCP_W(c)$ ne dépend
pas du choix de l'élément de Coxeter $c$ ; on notera souvent simplement
$\NCP_W$.

\medskip

Dans $\NCP_W$, la longueur correspond toujours à la codimension de l'espace
des points fixes :

\begin{prop}
  \label{proplongueur}
  Soit $w\in \NCP_W$. Alors :
  \[ \ell(w) = \codim \Ker(w-1) .\]
\end{prop}

\begin{rqe}
  Cela revient à dire que pour les éléments de $\NCP_W$, la longueur
  relativement à $\CR$ est la même que la longueur relativement à
  l'ensemble de toutes les réflexions du groupe unitaire $\U(V)$. Dans le
  cas des groupes réels, cette propriété est vraie pour tout élément de $W$
  (cf. \cite[Lemme 2.8]{carter}, ou \cite[Prop.\ 2.2]{BWordre2}).
\end{rqe}

\begin{proof}[Démonstration :]
  On a toujours $\ell(w) \geq \codim \Ker(w-1)$. En effet, si $w=r_1\dots
  r_k$ alors $\Ker(w-1) \supseteq {\bigcap_i \Ker(r_i-1)}$. Donc pour
  conclure il suffit de montrer que $\ell(c)=n$ (considérer $w$ et
  $w^{-1}c$). L'inégalité $\ell(c) \geq n$ est évidente car
  $\Ker(c-1)=\{0\}$ d'après la théorie de Springer. On conclut en notant
  qu'on peut construire géométriquement des factorisations de $c$ en $n$
  réflexions (\cite[Lemme 7.2]{BessisKPi1} ou partie \ref{partlbl} de cet
  article).
\end{proof}

\subsection{Théorème de Brady-Watt} 
\label{subpartbw}
~ 

Une conséquence de la proposition \ref{proplongueur} est que l'ordre~$\<$
dans $\NCP_W$ est la restriction de l'ordre partiel sur $\U(V)$ étudié par
Brady et Watt dans \cite{BWordre} : $u \< v$ si et seulement si $\Im(u-1)
\oplus \Im(u^{-1}v -1)= \Im(v)$. On en déduit en particulier :

\begin{prop}
  \label{propordre}
  Soit $w \in \NCP_W$. Si $(r_1,\dots, r_p)$ est une décomposition réduite
  de $w$, alors $\Ker(w-1)= \Ker(r_1-1) \cap \dots \cap \Ker(r_k-1)$.
\end{prop}

Les résultats de \cite{BWordre} sont donnés dans le cadre de $\OO(V)$ pour
$V$ espace euclidien, mais restent valables sans modifications dans $\U(V)$
avec $V$ hermitien complexe. On pourra donc appliquer le théorème suivant :

\begin{thm}[Brady-Watt]
  \label{thmBW}
  Pour tout $g \in \U(V)$, l'application
  \[ \begin{array}{rcl} \left(\{ f \in \U(V) \tq f \< g \},\ \< \right)
    &\rightarrow & \left(\{\mathrm{s.e.v.\ de\ } V \mathrm{\ contenant\ }
      \Ker(g-1) \},\ \supseteq \right)\\ f & \mapsto & \Ker(f-1)
    \end{array} \]
    est un isomorphisme d'ensembles partiellement ordonnés.
  \end{thm}

Par conséquent, l'application $w \mapsto \Ker(w-1)$ est injective sur $
\NCP_W $.

\section
{\texorpdfstring{Action
    d'Hurwitz sur les factorisations primitives et conjugaison forte dans
    $\NCP_W$}{Action d'Hurwitz sur les factorisations primitives et
    conjugaison forte}}
\label{partfconj}

Soit $p \in \{1,\dots, n\}$. Le groupe de tresses classique $B_p$ agit par
action d'Hurwitz sur l'ensemble $\Fact _p (c)$ des factorisations en $p$
blocs de $c$, selon la définition \ref{defhur}.
Dans le cas où $p=n$, l'action de $B_n$ est transitive sur $\Fact_n (c)
=\Red(c)$. Ce n'est pas vrai dans le cas général, puisque des invariants
évidents apparaissent :
\begin{itemize}
\item la partition de $n$ associée au $p$-uplet des longueurs des facteurs ;
\item le multi-ensemble des classes de conjugaison des facteurs (car
  l'action d'Hurwitz conjugue les facteurs).
\end{itemize}

Plus précisément, l'action d'Hurwitz conjugue les facteurs par d'autres
facteurs «~dont les longueurs s'ajoutent~». Cela motive l'introduction
d'une notion plus forte de conjugaison.

\begin{defn}
\label{deffconj}
Pour $w,w' \in \NCP_W(c)$, on définit la relation «~$w$ et $w'$ sont
fortement conjugués dans $\NCP_W(c)$~» (notée $w \fconj w'$) comme la
clôture transitive et symétrique de la relation suivante :
\[w \stackrel{1}{\sim} w' \text{\ s'il\ existe\ } x \in W \text{\ tel\ que\
} xw=w'x , \text{\ avec\ } xw \in \NCP_W(c) \text{\ et\ } \ell(xw)=\ell(x)
+ \ell(w) .\]
\end{defn}

\begin{rqe}
\label{rqhur}
  Bien sûr la relation $ \stackrel{1}{\sim}$ reflète l'incidence, sur un
  facteur, de l'action d'Hurwitz d'une tresse élémentaire $\bm{\sigma_i}$
  sur une factorisation de $c$. Par transitivité, si une tresse $\beta$
  transforme une factorisation $\xi$ en une factorisation $\xi'$, et envoie
  le numéro d'un facteur $w$ de $\xi$ sur celui d'un facteur $w'$ de
  $\xi'$, alors $w$ et $w'$ sont fortement conjugués.
\end{rqe}

On voit facilement qu'on obtient la même relation d'équivalence si on
impose que les conjugateurs $x$ de la définition soient toujours des
réflexions de $\NCP_W$. En effet, si $w$ et $w'$ sont tels que $w
\stackrel{1}{\sim} w'$, notons $x$ un élément de $W$ tel que $xw=w'x
\in \NCP_W$, avec $\ell(xw)=\ell(x) + \ell(w)$. Si l'on décompose $x$ en
produit minimal de réflexions ($x=r_1\dots r_k$), et si l'on note $u
\stackrel{\mathrm{ref}}{\sim} v$ la relation «~$u$ est élémentairement
conjugué à $v$ par une réflexion~», alors on a : $w
\stackrel{\mathrm{ref}}{\sim} r_k w r_k ^{-1} \stackrel{\mathrm{ref}}{\sim}
r_{k-1} r_k w r_k ^{-1} r_{k-1}^{-1} \stackrel{\mathrm{ref}}{\sim} \dots
\stackrel{\mathrm{ref}}{\sim} xwx^{-1}=w'$.
Cette remarque permet de faire le lien avec les orbites d'Hurwitz
de \emph{factorisations primitives} de $c$, où tous les facteurs sauf un
sont des réflexions (cf. définition \ref{deffactprim}).

\begin{prop}
  \label{propfconjhur}
  Soient $u,v \in \NCP_W$, de longueurs strictement supérieures à
  $1$. Alors les propriétés suivantes sont équivalentes :
  \begin{enumerate}[(i)]
  \item $u$ et $v$ sont fortement conjugués dans $\NCP_W$ ;
  \item il existe $\xi=(u,r_{2},\dots,r_k)$ et
    $\xi'=(v,r_{2}',\dots,r_k')$ deux factorisations primitives de
    $c$, avec $u$ (resp. $v$) facteur long de $\xi$ (resp. $\xi'$),
    telles que $\xi$ et $\xi'$ soient dans la même orbite d'Hurwitz
    sous $B_k$.
  \end{enumerate}
\end{prop}

\begin{rqe} \label{rqrefl} Dans le cas des réflexions, on a une proposition
  similaire, mais il faut s'assurer que la (permutation associée à la)
  tresse qui transforme $\xi$ en $\xi'$ envoie le numéro de la
  position de $u$ sur celui de la position de $v$.
\end{rqe}

\begin{proof}[Démonstration :]
  (ii) $\Rightarrow$ (i) est clair par définition, puisque, étant donnée
  l'invariance de la longueur par conjugaison, $v$ est nécessairement
  obtenu à partir de $u$ par forte conjugaison (cf. remarque
  \ref{rqhur}). Pour le sens direct, montrons d'abord (ii) lorsque $u$ et
  $v$ sont tels que $ur=rv \< c$, avec $r \in \CR$ et
  $\ell(ur)=\ell(u)+1$. Il suffit dans ce cas de prendre une factorisation
  quelconque de $c$ qui commence par $(u,r)$, et de faire agir la tresse
  $\bm{\sigma_1} ^2$. On conclut par transitivité.
\end{proof}

Ainsi le théorème \ref{thmintro} se déduit du théorème suivant, qui donne
également des propriétés supplémentaires de l'action d'Hurwitz sur
$\Red(c)$ :

\begin{thm}
  \label{thmfconj}
  Soit $u,v \in \NCP_W$. Alors $u$ et $v$ sont fortement conjugués dans
  $\NCP_W$ si et seulement si $u$ et $v$ sont conjugués dans $W$.
\end{thm}

On aura ainsi déterminé les orbites d'Hurwitz de factorisations
primitives. Toutefois, cela ne suffit pas pour comprendre complètement
l'action d'Hurwitz sur les factorisations de forme quelconque. En effet,
pour toute factorisation, le multi-ensemble des classes de conjugaison
forte des facteurs est naturellement invariant par l'action
d'Hurwitz. Cependant, pour les factorisations non primitives, la condition
n'est en général pas suffisante pour que deux factorisations soient dans la
même orbite. Par exemple, dans le cas où $p=2$, l'orbite d'Hurwitz de
$(u_1,u_2) \in \Fact_2 (c)$ est $\{ (u_1^{c^k},u_2^{c^k}),
(u_2^{c^{k+1}},u_1^{c^k}), \linebreak[1] {k \in \BZ} \}$ (où la notation
$u^v$ désigne le conjugué $v^{-1}uv$). Donc dans le type $A$, cela revient
à agir par rotation sur le diagramme des partitions non-croisées. On peut
ainsi facilement trouver un contre-exemple dans $\FS_6$. Posons $u_1=(2\
3)(1\ 5\ 6), u_2=(1\ 3\ 4)$, et $v_1=(3\ 4)(1\ 5\ 6)$, $v_2=(1\ 2\ 4)$ :
alors les factorisations $(u_1,u_2)$ et $(v_1,v_2)$ de $\Fact_2((1\ 2\ 3\
4\ 5\ 6))$ ne sont pas dans la même orbite d'Hurwitz sous $B_2$, alors que
$u_1$ et $v_1$, resp. $u_2$ et $v_2$, sont conjugués. Le problème vient
du fait que lorsqu'on fait agir $B_p$ par action d'Hurwitz, les
conjugateurs qui s'appliquent ne sont pas quelconques, mais doivent être
des éléments qui sont aussi des facteurs (en particulier pas nécessairement
des réflexions).

\medskip

Désormais on supposera toujours que $W$ est irréductible (et bien
engendré). Si $W$ n'est pas irréductible, le théorème \ref{thmfconj} pour
$W$ se déduit du cas irréductible : on vérifie aisément qu'il suffit
d'appliquer le résultat à chacune des composantes irréductibles de $W$.

\section{Le morphisme de Lyashko-Looijenga}
\label{partLL}
Dans cette partie on rappelle la construction du morphisme de
Lyashko-Looijenga de \cite[Part. 5]{BessisKPi1}, et on fixe des notations
et conventions concernant l'espace de configuration de $n$ points.

\medskip

Le morphisme de Lyashko-Looijenga a été introduit par Lyashko en 1973
(selon Arnold \cite{arnold}) et indépendamment par Looijenga dans
\cite{looijenga} : voir \cite[Chap.5.1]{LZgraphs} et \cite{landozvonkine}
pour un historique détaillé.  Bessis, dans \cite{BessisKPi1}, a généralisé la
définition de $\LL$ à tous les groupes de réflexions complexes bien
engendrés, le cas initial correspondant aux groupes de Weyl.

\subsection{Discriminant d'un groupe bien engendré}
\label{subpartLL0}
~

Soit $W \subseteq \GL(V)$ un groupe de réflexions complexe bien engendré,
irréductible. Fixons une base $v_1,\dots, v_n$ de $V^*$, et $f_1,\dots,
f_n$ dans $S(V^*)= \BC[v_1,\dots, v_n]$ un système d'invariants
fondamentaux, polynômes homogènes de degrés uniquement déterminés $d_1 \leq
\dots \leq d_n=h$, tels que $\BC[v_1,\dots, v_n]^W=\BC[f_1,\dots, f_n]$
(théorème de Shephard-Todd-Chevalley). On a l'isomorphisme
\[\begin{array}{rcl}
W \qg V & \overset\sim\to & \BC^n \\
\bar{v} & \mapsto & (f_1(v),\dots, f_n(v)).
\end{array}
\]

Notons $\CA$ l'ensemble des hyperplans de
 réflexions de $W$.  L'équation de $\bigcup_{H \in \CA} H$ peut s'écrire
 comme un polynôme invariant, noté $\Delta$, le discriminant de $W$ :

\[ \Delta = \prod_{H\in \CA} \alpha_H ^{e_H} \in \BC[f_1,\dots,
f_n], \]
où $\alpha_H$ est une forme linéaire de noyau $H$ et $e_H$ est l'ordre du
sous-groupe parabolique cyclique $W_H$.

Lorsque $W$ est bien engendré, il existe un système d'invariants
$f_1,\dots, f_n$ tel que le discriminant $\Delta$ de $W$ s'écrive :
\[\Delta= f_n^n + a_2 f_n ^{n-2} +\dots + a_n ,\]
où les $a_i$ sont des polynômes en $f_1, \dots, f_{n-1}$
\cite[Thm.\ 2.4]{BessisKPi1}. La propriété fondamentale est que $\Delta$ est
un polynôme monique de degré $n$ en $f_n$ (le coefficient de $f_n^{n-1}$
est rendu nul par simple translation de la variable $f_n$).

On note :
\[\CH := \{\bar{v}\in W \qg V \tq \D(\bar{v})=0 \} = p\left(\bigcup_{H \in \CA}
H\right) ,\] où $p$ est le morphisme quotient $V \surj \ W \qg V$. Désormais,
lorsqu'on se placera dans l'espace quotient, on considèrera $(f_1,\dots,
f_n)$ comme les coefficients des points de $W \qg V \simeq \BC^n$.

Le morphisme de Lyashko-Looijenga $\LL$, que l'on définira
précisément en partie \ref{subpartLL}, permet d'étudier l'hypersurface $\CH$ à
travers les fibres de la projection $(f_1,\dots,f_n) \mapsto (f_1,\dots,
f_{n-1})$. Ensemblistement, $\LL$ associe à $(f_1,\dots, f_{n-1})$ le
multi-ensemble des racines de 
%\hspace{\stretch{1}} 
$\D (f_1,\dots, f_n)$
vu comme polynôme en $f_n$. Dans la partie suivante on fixe les notations
et définitions concernant l'espace d'arrivée de $\LL$.

\subsection{L'espace des configurations de $n$ points}
\label{subpartEn}
~

On note $E_n$ l'ensemble des configurations
centrées de $n$ points, avec multiplicités, \ie
\[E_n := H_0/ \FS_n\] où $H_0$ est l'hyperplan de $\BC^n$ d'équation $\sum
x_i =0$. 

Considérons le morphisme quotient ${f: H_0 \to E_n}$. C'est un
morphisme fini (donc fermé), correspondant à l'inclusion
$\BC[e_1,\dots,e_n]/(e_1)\ \inj \
\BC[x_1,\dots,x_n]/(\sum x_i)$, où $e_1,\dots,e_n$ désignent les fonctions
symétriques élémentaires en les $x_i$. L'espace $E_n$ est une variété
algébrique, son anneau des fonctions régulières sera noté
$\BC[e_2,\dots,e_n]$.

\medskip

On peut stratifier naturellement $E_n$ par les partitions de l'entier
$n$. Pour $\lambda \vdash n$, on note $E_{\lambda}^0$ l'ensemble des
configurations $X\in E_n$ dont les multiplicités sont distribuées selon
$\lambda$. Ainsi, $E_n$ est l'union disjointe des $E_{\lambda}^0$.

Une partie $E_{\lambda}^0$ n'est pas fermée ;
des points peuvent fusionner et on obtiendra une configuration
correspondant à une partition moins fine :

\begin{defn}
  Soient $\lambda, \mu$ deux partitions de $n$. On note $\mu \leq \lambda$
  la relation «~$\mu$ est \emph{moins fine} que $\lambda$~», \ie $\mu$ est
  obtenue à partir de $\lambda$ après un ou plusieurs regroupements de
  parts.

  On pose $E_{\lambda} :=\bigsqcup_{\mu \leq \lambda} E_{\mu}^0$
  (le symbole $\bigsqcup$ désigne une union disjointe).
\end{defn}

Ainsi, pour $\eps:=1^n \vdash n$, $E_{\eps}$ est l'espace $E_n$
entier, et $E_{\eps}^0$ est l'ensemble des points réguliers de $E_n$,
que l'on notera $\Enreg$. On a $\Enreg=H_0^{\mathrm{reg}}/\FS_n$, où
$H_0^{\mathrm{reg}}=\{(x_1,\dots,x_n)\in H_0 \tq \forall i\neq j,\ x_i\neq
x_j \}$. Pour $\alpha:=2^1 1^{n-2} \vdash n$, on a $E_{\alpha}=E_n-\Enreg$.

Pour la topologie classique, comme pour la topologie de Zariski,
$E_{\lambda}$ est un fermé qui est l'adhérence de
$E_{\lambda}^0$.

Fixons une configuration de référence $X^\circ$ dans $\Enreg$, par exemple
sur la droite réelle. On note $B_n:=\pi_1(\Enreg,X^\circ)$ le groupe de
tresses à $n$ brins. On a la présentation classique :
\[B_n  \simeq \left< \bm{\sigma_1},\dots,\bm{\sigma_{n-1}}
 \left| \bm{\sigma_i}\bm{\sigma_{i+1}} \bm{\sigma_i} =
\bm{\sigma_{i+1}}\bm{\sigma_i} \bm{\sigma_{i+1}},
\bm{\sigma_i}\bm{\sigma_j} = \bm{\sigma_j}\bm{\sigma_i} \text{ pour } |i-j| >
1 \right. \right> ,\] 
où par convention $\bm{\sigma_i}$ est représenté par le chemin suivant dans
le plan complexe :

{\shorthandoff{?;:}%
\[\xy
(-10,0)="1", (-2,0)="2", (7,0)="3", 
(14,0)="4", (22,0)="5", (32,0)="6", (40,0)="7",
"1"*{\bullet},"2"*{\bullet},"3"*{\bullet},"4"*{\bullet},
"5"*{\bullet},"6"*{\bullet},"7"*{\bullet},
(14,-3)*{_{x_i}},(24,-3)*{_{x_{i+1}}},
(-10,-3)*{_{x_1}},(40,-3)*{_{x_{n}}},
"4";"5" **\crv{(18,-3)}
 ?(.5)*\dir{>},
"5";"4" **\crv{(18,3)}
 ?(.5)*\dir{>}
\endxy \]
}

Tout chemin dans $\Enreg$ peut être vu comme un élément du groupe de
tresses $B_n$. Précisons la construction, qui sera importante par la
suite. On définit l'ordre lexicographique sur $\BC$ : $z \lex z'\ \ssi \
[ \re(z) \leq \re(z')] \mathrm{\ et\ } [\re(z)=\re(z')\ \Rightarrow \ \im(z)
\leq \im(z')] $.

 Soit $X$ une configuration quelconque de $\Enreg$, et posons
$\widetilde{X}=(x_1,\dots,x_n )$ son support ordonné pour
$\lex$. Considérons un chemin $t\mapsto (x_1(t),\dots, x_n(t))$ dans
$H_0^{\mathrm{reg}}$ de $\widetilde{X}$ vers $\widetilde{X^\circ}$ (le support
ordonné de $X^\circ$), tel que pour tout $t$, $x_1(t)\lex \dots \lex
x_n(t)$. Il détermine une unique classe d'homotopie de chemin de $X$ vers
$X^\circ$, que l'on note $\gamma_X$.  Si $X,X' \in \Enreg$, on fixe ainsi
une bijection entre le groupe $\pi_1(\Enreg, X^\circ)$ et l'ensemble des
classes d'homotopie de chemin de $X$ vers $X'$ par : $\tau \mapsto
\gamma_X \cdot \tau \cdot \gamma_{X'}^{-1}$.

Via ces conventions, on pourra désormais considérer tout chemin dans
$\Enreg$ comme un élément de $B_n$. De la même façon, si $\lambda$ est une
partition de $n$, on peut associer à tout chemin de $E_\lambda ^0$ un
élément de $B_r$, où $r=\#\lambda$ est le nombre de parts de $\lambda$ : on
considère simplement le mouvement du support de la configuration.  On
utilisera ces conventions dans la proposition \ref{propcompat} et le lemme
\ref{lemcompat}.

\subsection{\texorpdfstring{Définitions et propriétés du morphisme
    $\LL$}{Définitions et propriétés du morphisme LL}}
\label{subpartLL}
~ 

On pose $Y:=\Spec \BC[f_1,\dots, f_{n-1} ] \simeq \BC^{n-1}$. D'autre part,
$E_n=\Spec \BC[e_2,\dots, e_n]$. Rappelons que $a_2,\dots,a_n$, éléments de
$\BC[f_1,\dots,f_{n-1}]$, désignent les coefficients du discriminant
$\Delta$ en tant que polynôme en $f_n$.

\begin{defn}[{\cite[Def.5.1]{BessisKPi1}}]

  \label{defLL}
  Le morphisme de Lyashko-Looijenga (généralisé) est le morphisme de $Y$
  dans $E_n$ défini algébriquement par :
  \[ \begin{array}{rcl}
    \BC[e_2,\dots,e_n] & \to & \BC[f_1,\dots,f_{n-1}]\\
    e_i & \mapsto & (-1)^i a_i
  \end{array}\]
\end{defn}

Ensemblistement, $\LL$ envoie $y=(f_1,\dots,f_{n-1}) \in Y$ sur le
multi-ensemble (élément de $E_n$) des racines du polynôme $\D = f_n^n + a_2
(y) f_n^{n-2} +\dots + a_n(y)$, \ie les intersections, avec multiplicité,
de $\CH$ avec la droite $L_y:=\{(y,f_n) \tq f_n \in \BC \}$. On peut voir
aussi $\LL$ comme le morphisme algébrique qui envoie $y \in Y$ sur $(a_2,
\dots, a_{n})$ : il est alors quasi-homogène pour les poids : $\deg f_i =
d_i$, $\deg a_i = ih$.

\medskip

On pose $\CK := \{ y\in Y \tq \Disc (\D(y,f_n); f_n) =0 \}$, appelé lieu de
bifurcation de $\D$.  C'est le lieu où $\D$ a des racines multiples en tant
que polynôme en $f_n$. Ainsi $\CK=\LL^{-1}(E_\alpha)$ et $Y-\CK= \LL ^{-1}
(\Enreg)$, avec les notations de la partie \ref{subpartEn}. Les propriétés
fondamentales de $\LL$ sont données par le théorème suivant :

\begin{thm}[{\cite[Thm.\ 5.3]{BessisKPi1}}]
  \label{thmLLBessis}
  Les polynômes $a_2,\dots,a_n \in \BC[f_1,\dots,f_{n-1}]$ sont
  algébriquement indépendants, et $\BC[f_1,\dots,f_{n-1}]$ est un
  $\BC[a_2,\dots,a_n]$-module libre gradué de rang $\frac{n!h^n}{|W|}$. Par
  conséquent, $LL$ est un morphisme fini. De plus, sa restriction $Y-\CK \surj
  \Enreg$ est un revêtement non ramifié de degré $\frac{n!h^n}{|W|}$.
\end{thm}

La dernière propriété permet de définir, pour chaque $y\in Y - \CK$, une
action de Galois (ou action de monodromie) de $\pi_1(\Enreg,\LL(y))$ sur la
fibre de $\LL(y)$. D'après les conventions de la partie \ref{subpartEn},
cela détermine une action du groupe de tresses $B_n$ sur l'espace $Y- \CK$ :
les orbites sont exactement les fibres de $\LL$ (car $Y-\CK$ est connexe
par arcs). Pour $\beta \in B_n$, on note $y\cdot \beta$ l'image de $y$ par
l'action de $\beta$.
\medskip

Dans la partie suivante, on va construire pour chaque $y\in Y$ une
factorisation de $c$ ; on verra plus loin que l'action d'Hurwitz sur les
factorisations (définition \ref{defhur}) est compatible avec l'action de Galois
définie ici.

\section{Factorisations géométriques}

\label{partlbl}

En adaptant \cite[Part.6]{BessisKPi1}, on construit pour chaque $y\in Y$
une factorisation par blocs de $c$, dont la partition associée correspond à
la distribution des multiplicités de la configuration $\LL(y)$.

\subsection{Tunnels}
~ 

Soit $y \in Y$. Rappelons que $L_y$ désigne la droite complexe $\{(y,x) \tq
x \in \BC \}$. Notons $U_y$ le complémentaire dans $L_y$ des demi-droites
verticales situées au-dessous des points de $\LL(y)$, \ie :
\[ U_y:= \{ (y,z) \in L_y \tq \forall x \in \LL(y),\ \re(z)= \re(x)
\Rightarrow \im(z) > \im (x) \} \]

La partie $\CU := \bigcup_{y\in Y} U_y$ est un ouvert dense et contractile
de $W \qg \Vreg$ \cite[Lemme 6.2]{BessisKPi1}. On peut donc utiliser $\CU$
comme «~point-base~» pour le groupe fondamental de $W \qg \Vreg$. On
définit ainsi le groupe de tresses de $W$ :
\[B(W):=\pi_1(W \qg \Vreg, \CU) \]

\medskip

On rappelle ci-dessous \cite[Def.6.5]{BessisKPi1} :

\begin{defn}
  Un \emph{tunnel} est un triplet $T=(y,z,L)$ tel que $(y,z)\in \CU$,
  $(y,z+L)\in \CU$, et le segment $[(y,z),(y,z+L)]$ est contenu dans $W \qg
  \Vreg$. On associe à $T$ l'élément $b_T$ du groupe de tresses $\pi_1(W
  \qg \Vreg, \CU)$, représenté par le chemin horizontal $t\mapsto
  (y,z+tL)$.
\end{defn}

{\shorthandoff{;:}%
 \[\xy
<5pt,0pt>:<0pt,3.5pt>::
(0,0)="A", (0,-17)="A1", 
(-12,6)="B", (-12,-3)="B1", (-12,-17)="B2", 
(15,12)="C", (15,3)="C1", (15,-9)="C2", (15,-17)="C3", 
(3,9)*{_{U_y}},
"B"*{\bullet},"B1"*{\bullet},"A"*{\bullet},"C"*{\bullet},"C1"*{\bullet},"C2"*{\bullet},
"B";"B1" **@{-},
"B2";"B1" **@{-},
"A";"A1" **@{-},
"C";"C1" **@{-},
"C1";"C2" **@{-},
"C2";"C3" **@{-},
(-18,-6);(24,-6) **@{-}, *\dir{>},
(-21,-6)*{_{z}}, (30,-6)*{_{z+L}}
\endxy \]
}

Soit $y\in Y$, et $(x_1,\dots,x_k)$ le support du multi-ensemble $\LL(y)$,
ordonné selon l'ordre lexicographique $\lex$. Notons $\pr_{\BC}$ la
projection de $W \qg V \simeq Y \times \BC$ sur la dernière
coordonnée. Considérons l'espace $\pr_{\BC} (L_y - U_y) - \{x_1,\dots, x_k
\}$. C'est une union disjointe d'intervalles ouverts, bornés ou non. Pour
$j \in \{1,\dots,k\}$, on note $I_j$ l'intervalle situé sous $x_j$.

{\shorthandoff{;:}%
 \[\xy
<5pt,0pt>:<0pt,3.5pt>::
(0,0)="A", (0,-17)="A1", 
(-12,6)="B", (-12,-3)="B1", (-12,-17)="B2", 
(15,12)="C", (15,3)="C1", (15,-9)="C2", (15,-17)="C3", 
"B"*{\bullet},"B1"*{\bullet},"A"*{\bullet},"C"*{\bullet},"C1"*{\bullet},"C2"*{\bullet},
(-15,6)*{_{x_2}},(-15,-3)*{_{x_1}},(-3,0)*{_{x_3}},(12,-9)*{_{x_4}},(12,3)*{_{x_5}},(12,12)*{_{x_6}},
(-10,-10)*{_{I_1}},(-10,1.5)*{_{I_2}},(2,-8.5)*{_{I_3}},(17,-13)*{_{I_4}},(17,-3)*{_{I_5}},(17,7.5)*{_{I_6}},
"B";"B1" **@{-},
"B2";"B1" **@{-},
"A";"A1" **@{-},
"C";"C1" **@{-},
"C1";"C2" **@{-},
"C2";"C3" **@{-}
\endxy \]
}
Pour chaque $x_j$, on choisit un tunnel $T_j$ qui traverse
$I_j$ et pas les autres intervalles. On note $s_j:=b_{T_j}$ l'élément
de $B(W)$ associé.
{\shorthandoff{;:}%
 \[\xy
<5pt,0pt>:<0pt,3.5pt>::
(0,0)="A", (0,-17)="A1", 
(-12,6)="B", (-12,-3)="B1", (-12,-17)="B2", 
(15,12)="C", (15,3)="C1", (15,-9)="C2", (15,-17)="C3", 
"B"*{\bullet},"B1"*{\bullet},"A"*{\bullet},"C"*{\bullet},"C1"*{\bullet},"C2"*{\bullet},
(-15,6)*{_{x_2}},(-15,-3)*{_{x_1}},(-3,0)*{_{x_3}},(12,-9)*{_{x_4}},(12,3)*{_{x_5}},(12,12)*{_{x_6}},
(-8,-12)*{_{s_1}},(-8,-0.5)*{_{s_2}},(4,-10.5)*{_{s_3}},(19,-15)*{_{s_4}},(19,-5)*{_{s_5}},(19,5.5)*{_{s_6}},
(-16,-10);(-8,-10) **@{-}, *\dir{>},(-16,1.5);(-8,1.5) **@{-}, *\dir{>},
(-4,-8.5);(4,-8.5) **@{-}, *\dir{>},
(11,-13);(19,-13) **@{-}, *\dir{>},(11,-3);(19,-3) **@{-},
*\dir{>},(11,7.5);(19,7.5) **@{-}, *\dir{>},
"B";"B1" **@{-},
"B2";"B1" **@{-},
"A";"A1" **@{-},
"C";"C1" **@{-},
"C1";"C2" **@{-},
"C2";"C3" **@{-}
\endxy \]
}

Dans \cite[Def.6.7]{BessisKPi1}, le $k$-uplet $(s_1,\dots,s_k)$ est
 appelé \emph{label} de y et noté $\lbl(y)$. Ici on va associer à $y$ un
 $k$-uplet légèrement différent, mieux adapté au problème. 

\subsection{Factorisations géométriques}
\label{subpartfactogeom}

\begin{defn}
  \label{deflbl}
  Soit $y\in Y$, $x_1,\dots,x_k$ et $s_1,\dots,s_k$ définis comme
  ci-dessus. On note $\fact_B(y)$ le $k$-uplet $(s_1',\dots, s_k')$
  d'éléments de $B(W)$ où :
  \[ s_i'=\left\{
    \begin{array}{ll}
      s_i s_{i+1}^{-1} & \mathrm{si\ }\re(x_{i+1})=\re(x_i)\\
      s_i & \mathrm{sinon.}
    \end{array}
  \right.\]

  On appelle \emph{factorisation associée à $y$}, et on note $\fact(y)$, le
  $k$-uplet $(\pi(s_1'),\dots, \pi(s_k'))$ d'éléments de $W$ où $\pi:B(W)
  \surj W$.
\end{defn}

La raison de l'utilisation du terme «~factorisation~» sera clarifiée plus
loin.

\begin{rqe}
  \label{rqpi}
  L'application $\pi$ est une projection naturelle $B(W)
  \stackrel{\pi}{\surj}W$. Le revêtement $\Vreg \surj W \qg \Vreg$ permet
  en effet d'obtenir la suite exacte
  \[ 1 \to P(W)=\pi_1(\Vreg) \to B(W)=\pi_1 (W \qg \Vreg)
  \stackrel{\pi}{\to} W \to 1 .\] Pour définir précisément $\pi$, on doit
  choisir une section $\widetilde{\CU}$ du point-base $\CU$ dans
  $\Vreg$. Il y a $|W|$ choix possibles, qui donnent des morphismes
  conjugués (cf. \cite[Rq.6.4]{BessisKPi1}).
\end{rqe}

Géométriquement, $s_i'$ est représenté par un chemin qui, vu depuis
$\CU$, passe sous $x_i$ mais pas sous $x_j$ pour $j\neq i$ :
{\shorthandoff{;:}%
 \[\xy
<5pt,0pt>:<0pt,3.5pt>::
(0,0)="A", (0,-17)="A1", 
(-12,6)="B", (-12,-3)="B1", (-12,-17)="B2", 
(15,12)="C", (15,3)="C1", (15,-9)="C2", (15,-17)="C3", 
"B"*{\bullet},"B1"*{\bullet},"A"*{\bullet},"C"*{\bullet},"C1"*{\bullet},"C2"*{\bullet},
% (-15,6)*{_{x_2}},(-15,-3)*{_{x_1}},(-3,0)*{_{x_3}},(12,-9)*{_{x_4}},(12,3)*{_{x_5}},(12,12)*{_{x_6}},
(-16,-2)*{_{s_1'}},(-8,6)*{_{s_2'}},(4,-10.5)*{_{s_3'}},(11,-8)*{_{s_4'}},(11,4)*{_{s_5'}},(19,12)*{_{s_6'}},
(-11.5,-7);(-16,-7) **@{-},(-11.5,1);(-16,1) **@{-}, *\dir{>},
"B1"*\cir<14pt>{r^l},%flèche s_1'
(-16,4);(-8,4) **@{-}, *\dir{>},%flèche s_2'
(-4,-8.5);(4,-8.5) **@{-}, *\dir{>},%flèche s_3'
(15.5,-5);(11,-5) **@{-}, *\dir{>},(15.5,-13);(11,-13) **@{-},
"C2"*\cir<14pt>{r^l},%flèche s_4'
(15.5,7);(11,7) **@{-},*\dir{>},(15.5,-1);(11,-1) **@{-},
"C1"*\cir<14pt>{r^l},%flèche s_5'
(11,10);(19,10) **@{-}, *\dir{>},%flèche s_6'
"B";"B1" **@{-},
"B2";"B1" **@{-},
"A";"A1" **@{-},
"C";"C1" **@{-},
"C1";"C2" **@{-},
"C2";"C3" **@{-}
\endxy \]
}

\begin{rqe}
\label{rqzariski}
Si $y \in Y-\CK$, $\fact_B(y)$ est un $n$-uplet de «~réflexions tressées~»
(terminologie de Broué, cf. \cite{broumarou}), et détermine des générateurs
de la monodromie de $W \qg \Vreg=\BC^n - \CH$. Plus précisément, par
$\pi_1$-surjectivité de l'inclusion ${L_y - L_y\cap \CH } \ \inj \ {\BC^n -
  \CH}$ (cf. \cite[Thm.\ 2.5]{zariski}), les facteurs de $\fact_B(y)$
engendrent le groupe de tresses $B(W)$.
\end{rqe}

Dans le cas où le support de $\LL(y)$ est «~générique~» (parties réelles
distinctes), les $k$-uplets $\fact_B(y)$ et $\lbl(y)$ coïncident. Dans le
cas général, on peut toujours perturber $y$ en un $y'$ tel que
$\LL(y')=e^{-i\theta}\LL(y)$, avec $\theta>0$ assez petit pour que le
support de $\LL(y')$ soit générique et que
$\fact_B(y)=\fact_B(y')=\lbl(y')$. Les propriétés de $\lbl$ énoncées dans
\cite{BessisKPi1} s'adaptent ainsi aisément à l'application $\fact$ ; par
la suite (\ref{propfacto}, \ref{propcompat}, \ref{thmbij}), on rappelle ces
propriétés, en les reformulant dans ce nouveau cadre, et on ne démontre que
ce qui change de façon non triviale.

\medskip

Si $\LL(y)=\{0\}$ (avec multiplicité $n$), alors $y=0$ (d'après le théorème
\ref{thmLLBessis}). On note alors $\delta$ l'élément de $B(W)$ tel que
$\fact_B(0)=(\delta)$. Il est représenté par l'image dans $W \qg \Vreg$ du
chemin dans $\Vreg$ :
\begin{eqnarray*}
  [0,1] & \longrightarrow & V^{\reg}\\
  t & \longmapsto & v\exp(2i\pi t/h).
\end{eqnarray*}
où $v$ est tel que pour $i=1,\dots,n-1$, $f_i(v)=0$ (\ie $v\in
L_0$). Notons que $\delta$ est une racine $h$-ième du «~tour complet~» de
$P(W)$ («~full-twist~» noté $\bm{\pi}$ dans \cite[Not.2.3]{broumarou}).

On pose $c=\pi(\delta)$, image de $\delta$ dans $W$. Par construction, $c$
est $\eh$-régulier, donc est un élément de Coxeter. Les autres éléments de
Coxeter de $W$ (conjugués) sont obtenus pour d'autres choix du morphisme
$\pi$ (cf. remarque \ref{rqpi}).

\medskip

Le lemme 6.14 dans \cite{BessisKPi1} permet de comparer les tunnels lorsqu'on
change de fibre $L_y$ :

\begin{lemme}[«~Règle d'Hurwitz~»]
  \label{reglehur}
  Soit $T=(y,z,L)$ un tunnel, qui représente un élément $s\in B(W)$. Soit
  $\Omega$ un voisinage connexe par arcs de $y$, tel que pour tout $y'\in
  \Omega$, $(y',z,L)$ soit encore un tunnel. Alors, pour tout $y'\in
  \Omega$, $(y',z,L)$ représente $s$.
\end{lemme}

On rassemble dans la proposition suivante quelques conséquences
utiles de ce lemme, adaptées de \cite[Lemmes 6.16, 6.17]{BessisKPi1} :

\begin{prop}
\label{propfacto}
  Soit $y\in Y$, $(x_1,\dots, x_k)$ le support ordonné de $\LL(y)$, et
  $(s_1,\dots,s_k)=\fact_B(y)$. Alors :
  \begin{enumerate}[(i)]
  \item $s_1\dots s_k=\delta$ et $\pi(s_1)\dots \pi(s_k)=c$ ;
  \item pour tout $i$, la longueur de $\pi(s_i)\in W$ est la multiplicité
    de $x_i$ dans $\LL(y)$.
  \end{enumerate}
  Par conséquent, $\fact(y)\in \Fact(c)$, et si $y \in \LL ^{-1}
  (E_{\lambda}^0)$, $\fact(y) \in \Fact_\lambda (c)$.
\end{prop}

Une autre conséquence est la compatibilité des actions d'Hurwitz et de
Galois :
\begin{prop}[d'après {\cite[Cor.\ 6.18]{BessisKPi1}}]
  \label{propcompat}
  Soit $y\in Y-\CK$, et $\beta \in B_n$. Alors :
  \[ \fact(y \cdot \beta) = \fact(y)\cdot \beta \] où la première action
  est l'action de Galois définie en \ref{subpartLL}, et la seconde est
  l'action d'Hurwitz (cf. définition \ref{defhur}).
\end{prop}

\begin{proof}[Démonstration :]
  La preuve de \cite[Cor.\ 6.18]{BessisKPi1} s'adapte au cas non générique
  (parties réelles non distinctes) sans difficultés. On verra plus loin
  (lemme \ref{lemcompat}) une généralisation de cette propriété de
  compatibilité à tous les éléments de $Y$.
\end{proof}

\section{\texorpdfstring{Etude de $\LL$ sur les strates $E_\lambda$}{Etude de LL sur les strates}}

\label{partLLstrat}

Ci-dessous on reformule \cite[Thm.\ 7.9]{BessisKPi1} en utilisant le produit
fibré $E_n \times_{\cp(n)} \Fact(c)$, où $\cp(n)$ désigne l'ensemble des
compositions de $n$. A toute configuration $X$ de $E_n$ on peut associer
une composition de $n$, constituée des multiplicités des points du support
de $X$ pris dans l'ordre $\lex$ sur $\BC$. D'autre part, toute
factorisation de $\Fact(c)$ détermine une composition de $n$, en
considérant les longueurs des facteurs dans l'ordre. Ces constructions
définissent deux applications : $\comp_1 : E_n \to \cp(n)$ et $\comp_2 :
\Fact(c) \to \cp(n)$. On pose :
\[ E_n \times_{\cp(n)} \Fact(c) := \{(X,\xi) \in E_n \times \Fact(c)
\tq \comp_1(X)=\comp_2(\xi) \} \]

\begin{thm}[d'après {\cite[Thm.\ 7.9]{BessisKPi1}}]
  \label{thmbij}
  L'application $\LL \times \fact$ :
  \[ \begin{array}{lcl}
    Y & \to & E_n  \times_{\cp(n)}  \Fact(c)\\
    y & \mapsto &(\LL(y) , \quad \fact(y))
  \end{array}
  \]
  est bijective.
\end{thm}

Donnons ici les grandes lignes de la démonstration (détails dans
\cite[Part.7]{BessisKPi1}). L'essentiel est de montrer la bijectivité de
$\LL\times \fact : Y-\CK \to \Red(c)$ (le reste se fait en dégénérant les
points réguliers). Cela revient à prouver que, pour $y \in Y-\CK$, le
morphisme de $B_n$-ensembles
\[ y\cdot B_n \xrightarrow{\fact} \fact(y)\cdot B_n \] est un isomorphisme
\cite[Thm.\ 7.4]{BessisKPi1}. Pour cela, on montre que ${|\fact(y)\cdot
  B_n|}={|y\cdot B_n|}=\frac{n!h^n}{|W|}$ en utilisant les deux propriétés
suivantes, prouvées au cas par cas \cite[Prop.\ 7.5]{BessisKPi1} :
  \begin{enumerate}[(i)]
  \item l'action d'Hurwitz sur $\Red(c)$ est transitive ;
  \item $|\Red(c)|= \frac{n!h^n}{|W|}$.
  \end{enumerate}

\bigskip

Posons maintenant $Y_{\lambda} := \LL ^{-1} (E_{\lambda})$, et $Y_\lambda^0:=\LL
^{-1} (E_{\lambda}^0)$. Ainsi, $Y_{\lambda}$ est un fermé de Zariski
dans $Y\simeq \BC^{n-1}$, et est l'adhérence de $Y_{\lambda}^0$. Pour
tout $\lambda \vdash n$, $Y_\lambda = \bigsqcup_{\mu \leq \lambda}
Y_\mu ^0$. En particulier, $Y=\bigsqcup_{\mu \vdash n} Y_\mu ^0$.

\begin{thm}
  \label{thmrevetement}
  Pour tout $\lambda \vdash n$, la restriction de $\LL$ :
  \[ \LL_{\lambda} :Y_{\lambda}^0 \ \surj \ E_{\lambda}^0 \] est un
  revêtement non ramifié.
\end{thm}

\begin{rqe}
  La définition de «~revêtement~» est bien sûr à prendre ici au sens large
  : comme $Y_{\lambda}^0$ n'est pas nécessairement connexe par arcs, le
  théorème signifie que l'application $\LL_\lambda$ est un revêtement
  connexe par arcs sur chacune des composantes connexes par arcs de
  $Y_{\lambda}^0$.
\end{rqe}

\begin{proof}[Démonstration :]
  Soit $X_0 \in E_{\lambda}^0$. Supposons dans un premier temps que les
  parties réelles des éléments du support de $X_0$ sont distinctes (on dira
  que $X_0$ est relativement générique). Soit $\Omega$ un ouvert connexe
  par arcs de $E_{\lambda}^0$, contenant $X_0$, et de diamètre assez petit
  pour que tous les éléments de $X_0$ soient relativement génériques.

  Soit $\mu=\comp_1(X_0)$ la composition de $n$ associée à $X_0$. Il est
  clair que pour tout $X \in \Omega$, la composition associée à $X$ est
  encore $\mu$. Posons $F_\mu:=\Fact_{\mu} (c)=\{\xi \in \Fact(c) \tq
  \comp_2(\xi)=\mu \}$. Le théorème \ref{thmbij} implique que l'application
  \[ \LL_{\lambda}^{-1}(\Omega) \xrightarrow{\LL_{\lambda}\times \fact}
  \Omega \times F_\mu \] est une bijection.

    Si $X_0$ n'est pas relativement générique, on peut refaire toutes les
  constructions précédentes en tournant légèrement la direction de la
  verticale dans le sens trigonométrique direct. On doit modifier alors la
  définition de l'application $\fact$ (sauf pour les $y$ de la fibre de
  $X_0$), mais les factorisations construites restent toujours dans
  $\Fact_{\mu} (c)$ pour $\mu = \comp_1 (X_0)$.

  Pour conclure, il suffit de remarquer que toutes les fibres $F_\mu$
  construites sont en bijection. En effet, les compositions $\mu$ associées
  aux éléments de $Y_\lambda ^0$ correspondent nécessairement à la
  partition $\lambda$, et si deux compositions $\mu$ et $\mu'$ sont égales
  à permutation des parts près, alors $\Fact_\mu (c)$ et $\Fact_{\mu'}(c)$
  sont en bijection (utiliser l'action d'Hurwitz par une tresse adaptée).
\end{proof}

Soit $\lambda$ une partition de $n$, et $X,X'\in E_\lambda ^0$. On a vu en
partie \ref{subpartEn} que toute classe d'homotopie de chemin de $X$ vers
$X'$ dans $E_\lambda ^0$ définit un élément de $B_p$ où $p=\# \lambda$. On
en déduit en particulier une action de Galois de $B_p$ sur chacune des
fibres de $\LL$ au-dessus de $E_{\lambda}^0$, et on a une propriété de
compatibilité des actions plus générale :

\begin{lemme}[Compatibilité des actions de Galois et d'Hurwitz sur
  les strates]
  \label{lemcompat}
  Soit $\lambda \vdash n$, avec $p=\#\lambda$.
  \begin{enumerate}[(i)]
  \item Soient $y,y' \in Y_\lambda^0$, reliés par un chemin $\gamma$ dans
    $Y_\lambda ^0$. Notons $X=\LL(y)$, $X'=\LL(y')$, et $\beta$ la tresse
    de $B_p$ représentée par l'image de $\gamma$ par $\LL$. Alors :
    $\fact(y')=\fact(y) \cdot \beta$.
  \item Soient $\beta \in B_p$, et $y,y' \in Y_\lambda ^0$ tels que
    $\fact(y')=\fact(y) \cdot \beta$. Notons $X= \LL(y)$, $X'=\LL(y')$ et $
    \widetilde{\beta}$ l'unique relevé (par $\LL$) d'origine $y$ de la
    tresse $\beta$ vue comme classe de chemin (dans $E_\lambda^0$) de $X$
    vers $X'$. Alors $y'$ est le point d'arrivée de $\widetilde{\beta}$.
  \end{enumerate}
\end{lemme}

\begin{proof}[Démonstration :]
  \begin{enumerate}[(i)]
  \item C'est essentiellement la même preuve que pour
    \cite[Cor.6.18]{BessisKPi1}. Il suffit de considérer le cas
    $\beta=\bm{\sigma_i}$ ($i$-ème tresse élémentaire de $B_p$), pour $i
    \in \{1,\dots, p-1\}$. En vertu du théorème \ref{thmrevetement}, on
    peut déplacer $X,X'$ (sans changer l'ordre des points des
    configurations), et $y,y'$ (sans modifier $\fact(y)$ et $\fact(y')$),
    de sorte que la tresse $\bm{\sigma_i}$ soit représentée par le chemin
    suivant (en pointillés) : {\shorthandoff{;:}%
      \[\xy
      (-10,0)="1", (-3,0)="2", (4,0)="3", 
      (14,-6)="4", (22,0)="5", (36,0)="6", (44,0)="7",
      (-10,-15)="11", (-3,-15)="22", (4,-15)="33", 
      (14,-15)="44", (22,-15)="55", (36,-15)="66", (44,-15)="77",
      "1";"11" **@{-}, "2";"22" **@{-}, "3";"33" **@{-}, "4";"44" **@{-},
      "5";"55" **@{-}, "6";"66" **@{-}, "7";"77" **@{-},
      "1"*{\bullet},"2"*{\bullet},"3"*{\bullet},"4"*{\bullet},
      "5"*{\bullet},"6"*{\bullet},"7"*{\bullet},
      (11,-6)*{_{x_i}},(27,0)*{_{x_{i+1}}},
      (-13,0)*{_{x_1}},(41,0)*{_{x_{p}}},
      "5";(10,0) **@{.}, *\dir{>},
      (7,-3);(27,-3) **@{-}, *\dir{>},
      (7,-10);(27,-10) **@{-}, *\dir{>},
      (18,-12)*{_{T_-}},
      (18,-5)*{_{T_+}}
      \endxy \] } Considérons les tunnels $T_+$ et $T_-$ représentés sur le
    schéma. Notons $(w_1,\dots, w_p)=\fact(y)$ et $(w_1',\dots,
    w_p')=\fact(y')$. Alors $T_+$ représente $w_{i+1}$ dans $L_y$, et
    $w_i'$ dans $L_{y'}$. D'où, d'après le lemme \ref{reglehur},
    $w_i'=w_{i+1}$. De même, avec $T_-$, on obtient $w_i'w_{i+1}'=w_i
    w_{i+1}$, et on a vérifié que $\fact(y')=\fact(y) \cdot \bm{\sigma_i}$.
  \item On reprend les notations de l'énoncé. Notons $y''$ le point
    d'arrivée de $\beta$. Alors, d'une part on a
    $\LL(y'')=X'=\LL(y)$. D'autre part, en appliquant le point (i) à $y$,
    $y'$, et $\widetilde{\beta}$, on obtient : $\fact(y'')=\fact(y)\cdot
    \beta=\fact (y')$. D'où, par le théorème \ref{thmbij}, $y''=y'$.
  \end{enumerate}
\end{proof}

Le théorème suivant est une conséquence directe du lemme :

\begin{thm}
  \label{thmcomposantes}
  Soit $\lambda \vdash n$, et $p=\#\lambda$. L'application $ Y_\lambda^0
  \xrightarrow{\fact} \Fact_\lambda(c)$ induit une bijection entre
  l'ensemble des composantes connexes par arcs de $Y_\lambda^0$ et
  l'ensemble des orbites d'Hurwitz de $\Fact_\lambda(c)$ sous $B_p$.

  Autrement dit, si $ Y_{\lambda}^0 = \bigsqcup_{i} Y_{\lambda,i}^0$ est la
  décomposition de $Y_{\lambda}^0$ en ses composantes connexes par arcs,
  alors
  \[ \Fact_{\lambda}(c)=\fact(Y_{\lambda}^0) = \bigsqcup_{i}
  \fact(Y_{\lambda,i}^0) \] est la décomposition de $\Fact_{\lambda}(c)$ en
  orbites d'Hurwitz sous $B_p$.
\end{thm}

\section{Stratification de $\CH$ et éléments de Coxeter paraboliques}
\label{partecp}

Dans cette partie on étudie la géométrie de $\CH$, afin de pouvoir
déterminer en partie \ref{partirred} les composantes connexes par arcs de
$Y_\lambda ^0$ lorsque $\lambda$ est une partition primitive, en appliquant
le théorème \ref{thmcomposantes} ci-dessus.

On rappelle qu'on a fixé un élément de Coxeter $c=\pi(\delta)$.

\subsection{Stratification de $V$}
~ 

L'arrangement d'hyperplans associé à $W$ est noté $\CA$. On
considère la stratification de $W$ par les \emph{plats}, \ie
le treillis d'intersection :

\[ \CL:=\CL(\CA)=\left\{\bigcap_{H\in \CB} H \tq \CB \subseteq \CA \right\} .\]

Pour $L \in \CL$, on pose 
\[L^0:=L - \bigcup_{L'\in \CL, L' \subsetneq L} L'.\]
On obtient la stratification ouverte de $V$ associée à
$\CL$. Pour tout $L\in \CL$, $L$ est l'adhérence de $L^0$.

La stratification de $V$ par les plats correspond à la stratification du
groupe $W$ par ses sous-groupes paraboliques : on rappelle le théorème de
Steinberg \cite[Thm.\ 1.5]{steinberg2} et ses conséquences.

\begin{thm}[Steinberg]
\label{thmsteinberg}
  Si $L$ est un plat, le groupe 
  \[W_L:=\{w\in W \tq \forall x \in L, wx=x\} \]
  est encore un groupe de réflexions, appelé sous-groupe parabolique
  associé à $L$. De plus :
  \begin{enumerate}[(i)]
  \item L'application $L \mapsto W_L$ est une bijection de $\CL$ vers
    l'ensemble des sous-groupes paraboliques ; sa réciproque est
    \[ G \mapsto V^{G}=\{x\in V \tq \forall w \in G, wx=x \}=\bigcap_{r
      \in \CR\cap G} H_r \ .\]
  \item Le rang de $W_L$ est égal à la codimension de $L$.
  \item Soit $v \in V$. Notons $V_v:=\displaystyle{\bigcap_{H\in \CA, v\in H} H}$
    et $W_v:= \{w\in W \tq wv=v\}$.\\ Alors, pour $L\in \CL$ :
    \[ v \in L^0 \ \ssi \  V_v=L \ \ssi \ W_v=W_L .\]
  \end{enumerate}
\end{thm}

\subsection{Éléments de Coxeter paraboliques}
~
\label{subpartcoxparab}

Notons $f :\CH \to \NCP_W(c)$ l'application qui à $(y,x) \in \CH$ associe
le facteur de $\fact(y)$ correspondant à $x$ (\ie défini par un tunnel
élémentaire passant sous le point $x$ de $\LL(y)$). Ainsi,
$(f(y,x_1),\dots, f(y,x_p))=\fact(y)$ lorsque le support ordonné de
$\LL(y)$ est $(x_1,\dots, x_p)$ .
 
On reformule ci-dessous un lemme fondamental de \cite{BessisKPi1} :

\begin{lemme}[d'après {\cite[Lemme 7.3]{BessisKPi1}}]
  \label{lemcoxpar}
  Soit $y\in Y$. Soient $x \in \LL(y)$, de multiplicité $p$, et $w=f(y,x)$.

  Alors il existe une préimage $v\in V$
  de $(y,x)\in W \qg V$ telle que $w$ soit un élément
  de Coxeter dans le sous-groupe parabolique $W_{v}$. De plus :
  \[ p=\ell(w)=\rg W_{v}=\dim V/V_{v} = \codim \Ker (w-1). \]
\end{lemme}

Par conséquent, en utilisant la surjectivité de l'application $\fact$, on
peut déduire de ce lemme que tout diviseur $w$ de $c$ est un élément de
Coxeter d'un sous-groupe parabolique, groupe que l'on peut déterminer à
l'aide d'une factorisation qui contient $w$. D'où la proposition-définition
suivante :

\begin{prop}
  \label{propcoxpar}
  Soit $W$ un groupe de réflexions complexe bien engendré, et $w \in
  W$. Les propriétés suivantes sont équivalentes :
  \begin{enumerate}[(i)]
  \item $w$ est un élément de Coxeter d'un sous-groupe parabolique de $W$ ;
  \item il existe un élément de Coxeter $c_w $ de $W$, tel que $w \< c_w$ ;
  \item $w$ est conjugué à un élément de $\NCP_W(c)$.
  \end{enumerate}
  On dit alors que $w$ est un \emph{élément de Coxeter parabolique}.
\end{prop}

\begin{rqe}
Dans le cas d'un groupe de Coxeter fini, cette propriété est bien connue et
démontrée de manière uniforme, cf. \cite[Lemme 1.4.3]{Bessisdual}.
\end{rqe}

\begin{proof}[Démonstration :]
  (ii) $\ssi$ (iii) provient directement de la théorie de Springer (les
  éléments de Coxeter de $W$ forment une seule classe de conjugaison) et de la
  propriété $\NCP_W(a c a ^{-1}) = a \NCP_W(c) a ^{-1}$.

  (i) $\Rightarrow$ (ii) : soit $G$ un sous-groupe parabolique (non
  trivial), et $w$ un élément de Coxeter de $G$. Soit $L=V^G$ et $v \in
  L^0$, de sorte que $W_v=G$. Notons $(y,x)=\bar{v}$ et $w_0=f(y,x) \in
  \NCP_W(c)$. D'après le lemme \ref{lemcoxpar}, il existe $v_0\in V$ tel
  que $(y,x)=\overline{v_0}$ et que $w_0$ soit un élément de Coxeter de
  $W_{v_0}$. Comme $v_0$ et $v$ sont dans la même orbite sous $W$, $G=W_v$
  est conjugué à $W_{v_0}$, donc tous leurs éléments de Coxeter sont
  conjugués dans $W$. En particulier $w$ est conjugué à $w_0$.

  (iii) $\Rightarrow$ (i) : il suffit de montrer l'implication pour $w \in
  \NCP_W(c)$, puisque la propriété (i) est invariante par conjugaison. Si
  $w \<c$, la surjectivité de $\fact$ (théorème \ref{thmbij}) donne
  l'existence de $(y,x)\in W \qg V $ tel que $f(y,x)=w$. Le lemme
  \ref{lemcoxpar} permet alors de conclure.
\end{proof}

Comme dans le cas d'un groupe de Coxeter \cite[Cor.1.6.2]{Bessisdual}, on
peut retrouver, à partir d'un élément de Coxeter parabolique, le
sous-groupe parabolique associé :

\begin{prop}
  \label{propcoxpar2}
  Soit $w$ un élément de Coxeter parabolique, et $W_w$ le sous-groupe
  parabolique fixateur du plat $\Ker(w-1)$. Alors :
  \begin{enumerate}[(i)]
  \item le groupe $W_w$ est l'unique sous-groupe parabolique duquel $w$ est
    un élément de Coxeter ;
  \item si $(r_1,\dots,r_k)\in \Red (w)$, alors $\left< r_1,\dots, r_k
    \right>=W_w$.
  \end{enumerate}
\end{prop}

\begin{proof}[Démonstration :]
  (i) Soit $G$ un sous-groupe parabolique tel que $w$ soit un élément de
  Coxeter de $G$. Soit $L$ le plat associé à $G$, \ie $L=V^{G}$. Alors
  comme $w\in G$, on a : $L \subseteq \Ker(w-1)$. Or $\codim L =\rg G$ par
  théorème \ref{thmsteinberg}, et $\rg G=\ell_{G} (w)$ (où $\ell_G$ désigne la
  longueur relativement aux réflexions de $G$) car $w$ est un élément de
  Coxeter. D'autre part $\ell(w)=\codim \Ker(w-1)$ par la proposition
  \ref{proplongueur}. Donc pour conclure il suffit d'utiliser le résultat
  suivant :
  \[ \mathrm{Si}\ G\ \text{est\ un\ sous-groupe\ parabolique,\ alors\ :\ }
  \forall g \in \NCP_W,\ g\in G \Rightarrow \ell_G(g)=\ell(g). \] Pour cela, on
  va vérifier que pour $r\in \CR$, si $r\< g$, alors $r \in G$. Si $G=W_L$,
  cela revient à montrer que $L \subseteq \Ker(r-1)$. Or, comme $g\in G$,
  on a $L\subseteq \Ker(g-1)$, et d'après la proposition \ref{propordre},
  $\Ker(g-1) \subseteq \Ker(r-1)$.

  (ii) Il suffit de le démontrer dans le cas où $w$ est un élément de
  Coxeter $c$ d'un groupe de réflexions (bien engendré) irréductible
  $W$. D'après le théorème \ref{thmbij}, toute décomposition réduite
  $(r_1,\dots, r_n)$ de $c$ provient d'une factorisation de $\delta$ en
  générateurs de la monodromie, \ie de $\fact_B (y)=(s_1,\dots, s_n)$, avec
  $y \in Y-\CK$ et $\pi(s_i)=r_i$. Or on sait qu'alors $s_1,\dots, s_n$
  engendrent $B(W)$ (cf. remarque \ref{rqzariski}), donc $r_1,\dots, r_n$
  engendrent $W$.
\end{proof}

Dans le cas présent, on a fixé un élément de Coxeter $c$ de $W$, et on ne
va considérer que les éléments paraboliques qui sont dans $\NCP_W$.  On
n'obtient donc pas tous les sous-groupes paraboliques, mais seulement les
\emph{sous-groupes paraboliques «~non croisés~»}, \ie ceux qui possèdent un
élément de Coxeter qui divise $c$. Cependant d'après les propositions
\ref{propcoxpar} et \ref{propcoxpar2}, quitte à les conjuguer, on obtient
tous les sous-groupes paraboliques :

\begin{prop}
\label{propsgpnc}
Soit $W$ un groupe de réflexions complexe bien engendré, et $c$ un élément
de Coxeter fixé de $W$. Soit $W_0$ un sous-groupe parabolique de $W$.

Alors $W_0$ est conjugué à un sous-groupe parabolique «~non croisé~» de
$W$, \ie un sous-groupe de la forme $W_L$ où $L=\Ker(w-1)$ et $w \< c$.
\end{prop}

Ce un résultat non trivial utilise ainsi de façon essentielle le lemme
\ref{lemcoxpar}. On peut dire que les sous-groupes paraboliques non croisés
jouent ici le rôle des sous-diagrammes (ou des paraboliques standards) de
la théorie de Coxeter (voir aussi \cite[p.3]{broumarou1} sur les diagrammes
et sous-diagrammes pour les groupes complexes).

\begin{rqe}
  Dans un sous-groupe parabolique non croisé, il y a unicité de l'élément
  de Coxeter divisant $c$. En effet, il faut trouver un diviseur de $c$
  dont le plat associé est donné, et la solution est unique par le théorème
  de Brady-Watt (\ref{thmBW}). L'ensemble des sous-groupes paraboliques non
  croisés, ordonné par inclusion, est donc isomorphe au treillis
  $(\NCP_W(c), \<)$.
\end{rqe}

\subsection{\texorpdfstring{Stratification de $W \qg V$}{Stratification de
    W \textbackslash V}}
\label{subpartstratquotient}
~

Le groupe $W$ agit sur $\CA$, donc sur $\CL$. On peut ainsi définir des
orbites de plats, qui forment une stratification notée $\Lb$ de $W \qg V$. 

Soit $p$ la projection $V \rightarrow W \qg V$,
$v\mapsto \bar{v}=W\cdot v$. On a : $\Lb= W \qg \CL =(p(L))_{L \in \CL}=
(W\cdot L)_{L \in \CL} $.

\medskip

On notera par la suite les strates de $\Lb$ par la lettre $\Lambda$. Pour $
\Lambda \in \Lb $ , posons :
\[\Lambda^0:=\Lambda - \bigcup_{\Lambda' \in \Lb, \Lambda'
  \subsetneq \Lambda} \Lambda' \ .\]

Si $\Lambda=W \cdot L$, alors $\Lambda^0=W \cdot L^0$. Les ouverts
$\Lambda^0$, pour $\Lambda\in \Lb$, forment la stratification ouverte de $W
\qg V$ associée à $\Lb$, appelée \emph{stratification
  discriminante}. Notons que $(W \qg V)^0=W \qg V - \CH = W \qg \Vreg$.

\bigskip

Les strates de $\Lb$ correspondent aux classes de conjugaison de
sous-groupes paraboliques, puisque $W_{w\cdot L} = w W_L w ^{-1}$. On peut
également les associer aux classes de conjugaison d'éléments de Coxeter
paraboliques. En effet, considérons l'application $F: W\to \CL$, $w \mapsto
\Ker(w-1)$. Si $w$ et $w'$ sont conjugués, alors $F(w)$ et $F(w')$ sont
dans la même orbite sous $W$, donc $F$ induit une application $\bar{F}$ de
l'ensemble des classes de conjugaison de $W$ vers $\Lb$.

\begin{prop}
  L'application $\bar{F}$ définie ci-dessus induit une bijection entre
  l'ensemble $\Lb$ des strates de $W \qg V$ et :
  \begin{itemize}
  \item l'ensemble des classes de conjugaison d'éléments de Coxeter
    paraboliques ;
  \item l'ensemble des classes de conjugaison d'éléments de $\NCP_W$.
  \end{itemize}
\end{prop}

\begin{proof}[Démonstration :]
  En utilisant les propositions \ref{propcoxpar} et \ref{propcoxpar2}, le
  premier point est clair : deux éléments de Coxeter d'un même sous-groupe
  parabolique sont conjugués d'après la théorie de Springer, donc deux éléments
  de Coxeter paraboliques $w_1$ et $w_2$, associés à deux sous-groupes
  paraboliques $W_1$ et $W_2$, sont conjugués si et seulement si les
  groupes $W_1$ et $W_2$ sont conjugués.

  Le second point est direct en utilisant la proposition \ref{propcoxpar}.
\end{proof}

\begin{defn}
  \label{deftype}
  Soit $\Lambda$ une strate de $\Lb$, $w \in \NCP_W$, et  $\bar{v} \in W \qg
  V$. On dit que :
  \begin{itemize}
  \item «~$w$ est de \emph{type} $\Lambda$~» si la classe de conjugaison de $w$
    correspond à $\Lambda$ par la bijection ci-dessus, \ie si $\Lambda = W \cdot
    \Ker(w-1)$.
  \item «~la \emph{strate} de $\bar{v}$ est $\Lambda$~» si $\Lambda$ est la
    strate \emph{minimale} de $\Lb$ contenant $\bar{v}$, \ie si $\bar{v}
    \in \Lambda^0$, ou encore $\Lambda=W \cdot V_v$.
  \end{itemize}
\end{defn}

On peut ainsi reformuler le lemme \ref{lemcoxpar} :

\begin{lemme}
  \label{lemtype}
  Soit $(y,x) \in \CH$. Alors le type de l'élément de Coxeter parabolique
  $f(y,x)$ est la strate du point $(y,x)$.
\end{lemme}

Une conséquence du lemme est que si $(y,x)$ est dans une
strate ouverte de dimension $n-k$, alors $\ell(f(y,x))=k$ et la
multiplicité de $x$ dans $\LL(y)$ est $k$. Soit $\Lb_k$ l'ensemble des
strates fermées de dimension exactement $n-k$, et posons :

\[ \CH_k:= \bigcup_{\Lambda \in \Lb_k} \Lambda \ . \]

Ainsi $W \qg V= \CH_0 \supsetneq \CH_1=\CH \supsetneq \CH_2 \supsetneq
\dots \supsetneq \CH_n = \{0 \}$. 

\bigskip

Désormais on suppose $k\geq 1$. Soit la projection $\varphi : \CH \to Y,\
(y,x)\mapsto y$. Notons $Y_k := \varphi(\CH_k)$, et $\alpha_k := k^1
1^{n-k} \vdash n$. D'après le lemme \ref{lemcoxpar}, pour $(y,x)\in \CH$,
la longueur du facteur $f(y,x)$ est donnée par la codimension de la strate
de $(y,x)$, d'où la propriété suivante :

\begin{lemme}
  Soit $(y,x)\in W \qg V$. Alors  $(y,x)$ est dans $\CH_k$ si et
  seulement si la multiplicité de $x$ dans $\LL(y)$ est supérieure ou égale
  à $k$.

  Par conséquent : $Y_k= \LL ^{-1} (E_{\alpha_k})$.

\end{lemme}

La partie $Y_k$ est donc ce que l'on avait noté $Y_\lambda$ dans la partie 5,
dans le cas où $\lambda$ est la partition primitive $\alpha_k$. En
particulier, comme $\CK=\LL ^{-1} (E_{\alpha_2})$, on obtient $Y_2=\CK$.

\section{Composantes connexes par arcs de $Y_k^0$}
\label{partirred}

D'après le théorème \ref{thmcomposantes}, pour étudier les orbites
d'Hurwitz primitives, il est intéressant d'identifier les composantes
connexes par arcs de $Y_k^0$. Pour cela, on commence par déterminer les
composantes irréductibles de $Y_k$.

\subsection{Composantes irréductibles de $\CH_k$ et de $Y_k$}
~
\label{subpartirred}
On va d'abord montrer que les strates de $\Lb_k$ sont les composantes
irréductibles de $\CH_k$.

\begin{lemme}
  \label{lemfini}
  Les morphismes $p: V \surj W \qg V$ et $\varphi:\CH \rightarrow Y$ sont
  finis. Par conséquent ils sont fermés, pour la topologie de Zariski.
\end{lemme}

\begin{proof}[Démonstration :] 
  Pour $p$, c'est le théorème de Chevalley : $\CO_V \simeq \BC[v_1,\dots,
  v_n]$, $\CO_{W \qg V} \simeq \BC[f_1,\dots, f_n]$, et $\CO_V$ est un
  $\CO_{W \qg V}$-module libre de rang $|W|$.

  Pour $\varphi$, on a $\CO_Y=\BC[f_1,\dots, f_n]$ et $\CO_\CH =
  \BC[f_1,\dots, f_n] / (\D)$, avec $\D=f_n^n+a_2 f_n^{n-2} + \dots + a_n$
  et $a_i \in \BC[f_1,\dots, f_{n-1}]$. En particulier, $f_n$ est entier sur
  $\varphi^*(\CO_Y)$. Donc $\varphi^*$ fait de $\CO_\CH$ un $\CO_Y$-module
  libre de rang $n$.
\end{proof}

Par conséquent, les strates de $\Lb$, images par $p$ des plats dans $V$,
sont fermées, et irréductibles (car l'image d'un irréductible par un
morphisme algébrique est irréductible). D'où :

\begin{cor}
  Pour tout $k \geq 1$, les composantes irréductibles du fermé $\CH_k$ sont
  les strates $\Lambda$ de $\Lb_k$.
\end{cor}

Par $\varphi$, continue et fermée, on peut envoyer ces strates dans
$Y_k=\varphi(\CH_k)$. On a ainsi $Y_k=\bigcup_{\Lambda \in \Lb_k}
\varphi(\Lambda)$, avec les $\varphi(\Lambda)$ fermés irréductibles dans
$Y$.

\medskip

Commençons par le cas $k=1$, qui correspond aux classes de conjugaison
de réflexions de $\NCP_W$. En utilisant le théorème \ref{thmbij}, on
obtient le résultat suivant.

\begin{prop}
  Soit $\Lambda \in \Lb_1$. Alors, pour tout $y \in Y - \CK$, au moins un
  des facteurs de $\fact(y)$ a pour classe de conjugaison $\Lambda$.

  Par conséquent, pour tout $\Lambda \in \Lb_1$, on a : $\varphi(\Lambda) = Y$.

\end{prop}

\begin{proof}[Démonstration :]

  D'après la surjectivité de l'application $\fact$, toute réflexion de
  $\NCP_W$ apparaît dans une factorisation $\fact(y)$. Donc, par
  transitivité de l'action d'Hurwitz sur $\Red(c)$, si $\xi \in
  \fact(Y-\CK)$, alors toutes les classes de conjugaison de réflexions de
  $\NCP_W$ apparaissent dans $\xi$. En effet, l'ensemble des classes de
  conjugaison de réflexions faisant partie d'une décomposition réduite est
  invariant par l'action d'Hurwitz.

  Soit $y \in Y$. Quitte à désingulariser, on peut trouver $y' \in Y-\CK$
  tel que $\fact(y')$ soit un raffinement de $\fact(y)$. Soit $\Lambda \in
  \Lb_1$, alors $\fact(y')$ contient un facteur de type $\Lambda$, donc il
  existe $x$ dans $\LL(y)$ tel que $f(y,x)$ soit multiple (pour $\<$) d'une
  réflexion de type $\Lambda$, d'où $(y,x) \in \Lambda$, et $y \in
  \varphi(\Lambda)$.
\end{proof}

Désormais on suppose $k \geq 2$.

\begin{prop}
\label{propcompirr}
  Les $\varphi(\Lambda)$, pour $\Lambda \in \Lb_k$, sont distincts deux à deux,
  et sont les composantes irréductibles de $Y_k$.
\end{prop}

\begin{proof}[Démonstration :]
  D'après la discussion précédente, il suffit de montrer que si $\Lambda,
  \Lambda' \in \Lb_k$, avec $\Lambda \neq \Lambda'$, alors $\varphi(\Lambda)
  \nsubseteq \varphi(\Lambda')$.
  
  La strate ouverte $\Lambda ^0$ correspond à une classe de conjugaison de
  sous-groupe parabolique de rang $k$. Soit $c_\Lambda$ un élément de Coxeter
  parabolique, divisant $c$, de type $\Lambda$. Complétons avec $n-k$ réflexions
  pour obtenir une factorisation complète : $\xi=(c_\Lambda,
  s_{k+1},\dots,s_n)$. Alors, par \ref{thmbij}, il existe $y$ dans $Y$,
  tel que $\fact(y)=\xi$, et que $\LL(y)$ ait $n-k+1$ points
  distincts. Soit $x$ le point multiple dans $\LL(y)$ ; alors $(y,x)\in
  \Lambda^0$ puisque l'élément de Coxeter parabolique associé à $x$ est de
  type $\Lambda$. D'où : $y\in \varphi(\Lambda)$.

  Supposons que $y\in \varphi(\Lambda')$ ; alors il existe $x'$ tel que
  $(y,x')\in \Lambda'$. Donc, dans $\fact(y)=\xi$, on doit trouver un
  élément de type $\Lambda''\subseteq \Lambda'$, de longueur supérieure ou
  égale à $k$ ; or, dans $\fact(y)$, seul $c_\Lambda$ convient, et il est
  de type $\Lambda \nsubseteq \Lambda'$. D'où $y \notin \varphi(\Lambda')$,
  ce qui conclut la preuve.
\end{proof}

\subsection{Connexité par arcs}
\label{subpartcpa}
~

Notons comme plus haut $Y_k^0:= \LL ^{-1} (E_{\alpha_k}^0)$, et pour $\Lambda \in
\Lb_k$, $\varphi(\Lambda)^0:=\varphi(\Lambda)\cap Y_k^0$. 

Pour $\Lambda \in \Lb_k$, notons $\Fact_{\alpha_k}^{\Lambda}(c)$ les factorisations
primitives de $c$, de forme $\alpha_k$, et dont l'élément long est de type
$\Lambda$. Alors, en vertu du lemme \ref{lemtype}, on a :
\[ \varphi(\Lambda)^0 = \fact ^{-1} (\Fact_{\alpha_k}^{\Lambda}(c)). \]

Ainsi $Y_k ^0=\bigsqcup_{\Lambda \in \Lb_k} \varphi(\Lambda)^0$. D'après la
surjectivité de $\fact$, $\varphi(\Lambda)^0$ est un ouvert (de Zariski) non
vide de $\varphi(\Lambda)$. 

\begin{prop}
\label{propcpa}
Pour tout $\Lambda \in \Lb_k$, $\varphi(\Lambda)^0$ est connexe par arcs.
\end{prop}

\begin{proof}[Démonstration :]
  C'est un fait général, mais non trivial, que tout ouvert de Zariski d'une
  variété algébrique complexe irréductible est connexe par arcs. Dans le
  cas présent on peut cependant donner une démonstration explicite. Soit
  $\Lambda \in \Lb_k$, et soit $L \in \CL$ tel que $p(L)=\Lambda$. Notons $\Omega=
  L\cap (\varphi \circ p)^{-1}(\varphi(\Lambda)^0)$. Comme $p: V \to W \qg V$ et
  $\varphi : W \qg V \to Y$ sont des morphismes algébriques, $\Omega$ est
  un ouvert de Zariski de $L$. Comme $L$ est un espace vectoriel sur $\BC$,
  la connexité par arcs de $\Omega$ est alors évidente. D'autre part,
  $\varphi(\Lambda)^0=\varphi \circ p (\Omega)$, avec $\varphi \circ p$
  continue, donc $\varphi(\Lambda)^0$ est connexe par arcs.
\end{proof}

En utilisant le théorème \ref{thmcomposantes}, on va en déduire aisément
que les $\varphi(\Lambda)^0$, pour $\Lambda \in \Lb_k$, sont bien les
composantes connexes par arcs de $Y_k^0$.

\section{Forte conjugaison, et cas des réflexions}
\label{partrefl}

\begin{thm}
  \label{thmfconj2}
  Soit $k \in \{2,\dots , n \}$. Alors :
  \begin{itemize}
  \item les parties $\varphi(\Lambda)^0$, pour $\Lambda \in
    \Lb_k$, sont les composantes connexes par arcs de $Y_k^0$ ;
  \item les ensembles
    $\fact(\varphi(\Lambda)^0)=\Fact_{\alpha_k}^{\Lambda}(c)$ sont les
    orbites d'Hurwitz de $\fact(Y_k^0)=\Fact_{\alpha_k}(c)$ sous $B_{n-k+1}$.
  \end{itemize}
  Par conséquent, deux éléments de $\NCP_W$ de longueur $k$ sont fortement
  conjugués si et seulement s'ils sont conjugués.
\end{thm}

\begin{proof}[Démonstration :]
  Soient $y\in \varphi(\Lambda)^0$ et $y'\in \varphi(\Lambda')^0$, avec
  $\Lambda,\Lambda' \in \Lb_k$. Si $y$ et $y'$ sont reliés par un chemin
  dans $Y_k^0$, alors, par le théorème \ref{thmcomposantes}, $\fact(y)$ et
  $\fact(y')$ sont dans la même orbite d'Hurwitz, donc leurs éléments longs
  sont conjugués, \ie $\Lambda=\Lambda'$ (par lemme \ref{lemtype}). Donc la
  proposition \ref{propcpa} implique que les $\varphi(\Lambda)^0$ sont les
  composantes connexes par arcs de $Y_k^0$.  Les orbites d'Hurwitz de
  $\fact(Y_k^0)$ sont alors directement données par le théorème
  \ref{thmcomposantes}.

  Enfin, la propriété de forte conjugaison vient du fait qu'une
  factorisation de forme $\alpha_k$ est dans $\fact(\varphi(\Lambda)^0)$ si et
  seulement si son facteur long est de type $\Lambda$.
\end{proof}

Pour conclure la preuve du théorème \ref{thmfconj}, il reste à déterminer
les classes de conjugaison forte de réflexions :

\begin{thm}
  \label{thmfconjrefl}
  Soient $r,r'$ deux réflexions de $\NCP_W$. Si $r$ et $r'$ sont
  conjuguées, alors $r$ et $r'$ sont fortement conjuguées dans $\NCP_W$.
\end{thm}

\begin{rqe}
  Cette propriété apporte une précision intéressante concernant l'action
  d'Hurwitz de $B_n$ sur $\Red(c)$ : si $(r_1,r_2,\dots, r_n)$ et
  $(r_1',r_2',\dots, r_n')$ sont deux décompositions réduites de $c$, avec
  $r_1, r_1' \in \NCP_W$ deux réflexions conjuguées, alors il existe une
  tresse de $B_n$, \emph{pure par rapport au premier brin}, qui transforme
  l'une en l'autre (cf. remarque \ref{rqrefl} sur le lien entre forte
  conjugaison et action d'Hurwitz).
\end{rqe}

\begin{proof}[Démonstration :]
  Considérons $\varphi : \CH \to Y$, $(y,x) \mapsto y$. Notons $\CH' :=
  \varphi ^{-1} (Y- \CK)$. Alors la restriction $\varphi ' : \CH' \to Y -
  \CK$ est un revêtement non ramifié à $n$ feuillets (continuité des
  racines d'un polynôme à racines simples).

  Soient $r,r'$ deux réflexions de $\NCP_W$ conjuguées. Par surjectivité de
  $\fact$, il existe $y,y' \in Y -\CK$, tels que
  $\fact(y)=(r,r_2,\dots,r_n)$ et $\fact(y')=(r',r_2',\dots,r_n')$. On peut
  supposer que $\LL(y)=\LL(y')$ ; soit $x$ leur élément minimal pour $\lex$
  (correspondant à $r$ et $r'$). Soit $\Lambda$ la strate de $\Lb_1$
  correspondant à la classe de conjugaison de $r$ et $r'$. D'après le lemme
  \ref{lemtype}, $(y,x)$ et $(y',x)$ sont dans $\Lambda ^0$. Plus
  précisément, $(y,x)$ et $(y',x)$ sont dans $\Lambda \cap \CH'$, que l'on
  va noter $\Lambda'$.

  Notons que $\Lambda'$ est un ouvert de Zariski de $\Lambda$, donc est
  connexe par arcs, par le même argument que pour la proposition
  \ref{propcpa}. Par conséquent, on peut relier $(y,x)$ et $(y',x)$ par un
  chemin $\gamma$ dans $\CH '$. Celui-ci se projette en un chemin dans $Y-
  \CK$, et détermine via $\LL$ un lacet dans $\Enreg$. Ce lacet
  représente une tresse $\beta$ qui, par construction, stabilise le premier
  brin $(x)$. Ainsi $y'=y \cdot \beta$ par l'action de monodromie, et
  $\fact(y')=\fact(y)\cdot \beta$ par l'action d'Hurwitz. Comme $\beta$
  stabilise le brin $(x)$, on en déduit que $r$ et $r'$ sont fortement
  conjugués, en vertu de la remarque \ref{rqrefl}.
\end{proof}

\selectlanguage{english}
% Jacobian of..., 2e et 3e chapitre 
%%%%%%%%%%%%%%%%%%%%%%%%%%%%%%%%%%%%%%%%%%%%%%%%%%%%%%%%%%%%%%%%%%%%%%%%%%%
% Chapitre thèse d'après "Discriminants and Jacobians of virtual reflection
% groups", à compiler avec les en-têtes de manuscrit.tex
%%%%%%%%%%%%%%%%%%%%%%%%%%%%%%%%%%%%%%%%%%%%%%%%%%%%%%%%%%%%%%%%%%%%%%%%%%%

\chapter{Discriminants and Jacobians of virtual reflection groups}

\label{chapjac}

\section*{Introduction}

This chapter is a non-Galois version of the first few steps of the
classical invariant theory of reflection groups. We will deal with
questions of commutative algebra, that were at first motivated by empirical
observations on the extensions defined by Lyashko-Looijenga morphisms. 

We consider a finite polynomial ring extension $A\subseteq B$, where $A$ is not
necessarily the ring of invariants of $B$ under a group action. Thus, we
cannot simply imitate the classical proofs of invariant theory, as they
really make use of the group action. However, in our setting, many
properties seem to work the same way as for Galois extensions, particularly
for Jacobian and ``discriminant'' of the extension.

\bigskip

Note that we use only elementary commutative algebra, and that the
properties derived here are presumably folklore. The situation that we
describe is in fact surprisingly basic and universal, yet apparently
written nowhere from this perspective. The extensions usually studied in the
litterature, are either much too general, or of the form $A=B^G \subseteq
B$, where $B$ is a polynomial algebra; here we are rather interested in
extensions $A\subseteq B$ where $B$ \emph{and} $A$ are polynomial algebras,
but where we \emph{do not require} $A$ to be the ring of invariants of $B$
under a group action.

\bigskip

The key ingredients to describe and understand the situation are:
\begin{itemize}
\item a notion of \emph{``well-ramified''} polynomial extensions (this very
  natural property ought to be standard, but I could not find any
  references for that);
\item properties of the \emph{different} of an extension, that enable to
  apprehend the Jacobian of the extension.
\end{itemize}

\section{Motivations and main theorem}

Let $V$ be an $n$-dimensional complex vector space, and $W\subseteq \GL(V)$
a finite complex reflection group, with fundamental system of invariants
$f_1, \dots, f_n$ of degrees $d_1 \leq \dots \leq d_n$. From
Chevalley-Shephard-Todd theorem, we have the equality $\BC[V]^W=\BC[f_1,\dots, f_n]$,
and the isomorphism
\[ \begin{array}{lcl}
  W \qg V & \to & \BC^n \\
\bar{v} &\mapsto& (f_1(v),\dots, f_n(v)).
\end{array}\]

Let us denote by $\CA$ the set of all reflection hyperplanes, and consider the
discriminant of $W$ defined by
\[ \Delta_W := \prod_{H\in \CA} \alpha_H^{e_H} \ ,\] where $\alpha_H$ is an
equation of $H$ and $e_H$ is the order of the parabolic subgroup $W_H$. The
discriminant is the equation of the hypersurface $\displaystyle{\CH:=W \qg \bigcup_{H\in
  \CA}H}$ in $\BC^n=\Spec \BC[f_1,\dots, f_n]$.

Let us also consider the Jacobian $J_W$ of the morphism $(v_1,\dots, v_n)
\mapsto (f_1(v),\dots, f_n(v))$:
\[ J_W := \det \left( \frac{\partial f_i}{\partial v_j}
\right)_{\substack{1 \leq i\leq n\\1\leq j \leq n}} .\] It is well known
(see for example \cite[Sect.\ 21]{Kane}) that the Jacobian satisfies the
following factorisation:
\[ J_W \doteq \prod_{H\in \CA} \alpha_H^{e_H -1}, \] where $\doteq$ denotes
equality up to a nonzero scalar; thus we have $\Delta_W / J_W = \prod_{H\in
  \CA} \alpha_H$, \ie it is the product of the ramified polynomial of the
extension $\BC[f_1,\dots, f_n] \subseteq \BC[V]$.

\bigskip

One can construct a stunningly similar situation related to the morphism
$\LL$ defined in \ref{defLL}. Let us define an $\LL$-discriminant:
\[ D_{\LL} := \Disc(\Delta_W(f_1,\dots, f_n);f_n) \ \in \BC[f_1,\dots,
f_{n-1}] \] (it is the equation of the bifurcation locus $\CK$). Consider
also the $\LL$-Jacobian $J_{\LL}$ (the Jacobian determinant of the morphism
$\LL$). As we can observe empirically, and will prove in all generality in
the next chapter, it turns out that the couple of polynomials $(J_{\LL},
D_{\LL})$ behaves similarly to the couple $(J_W, \Delta_W)$: the quotient
$D_{\LL}/J_{\LL}$ is the product of the ramified polynomials of the
extension associated to $\LL$, and their valuations in $D_{\LL}$
correspond to their ramification indices.

\bigskip

Although we mainly have in mind applications to $D_{\LL}$, we devote this
chapter to the following general setup. Let us consider a \emph{finite
  graded polynomial extension} $A\subseteq B$ (see Definition
\ref{defext}): we have a graded polynomial algebra $B$ in $n$
indeterminates over $\BC$, and a polynomial subalgebra $A$
generated by $n$ weighted homogeneous elements of $B$, such that the
extension is finite. The two key examples are:
\begin{itemize}
\item the Galois extensions $\BC[f_1,\dots, f_n]\subseteq
  \BC[v_1,\dots, v_n]$, defined by a quotient morphism $V \to W \qg V$, where
  $w$ is a reflection group and $\BC[f_1,\dots, f_n]=\BC[V]^W$;
\item the Lyashko-Looijenga extensions $\BC[a_2,\dots, a_n]\subseteq
  \BC[f_1,\dots, f_{n-1}]$, given by a morphism $\LL$ (with the notations
  of section \ref{subpartLL}); these extensions are indeed finite according
  to Thm.\ \ref{thmLLBessis}.
\end{itemize}

In the first section we give the precise definitions, and use the notion of
\emph{different ideal} of an extension to describe a factorisation of the
Jacobian. In section \ref{partgeom}, about the geometry of such extensions,
we recall the relations between the ramification locus and the branch locus
of a branched covering. In section \ref{partwellram} we define the
\emph{well-ramified property} for a finite graded polynomial extension
(Def.\ \ref{defwellram}), and we give several characterisations of this
property (Prop.\ \ref{propwell}): this is a slightly weaker property than
the normality of the extension, and is also equivalent to the equality
between the preimage of the branch locus and the ramification locus.

\medskip

The main result of this chapter is:

\begin{theo}
  \label{thmintrojac}
  Let $W=(A \subseteq B)$ be a finite graded polynomial extension. Then the
  Jacobian $J$ of the extension verifies:
  \[ J \doteq \prod_{Q \in \Spram(B)} Q^{e_Q-1}\] where $\Spram(B)$ is the set
  of ramified polynomials in $B$ (up to association), and the $e_Q$ are the
  ramification indices.

  Moreover, if the extension $W$ is \emph{well-ramified} (according to
  Def.\ \ref{defwellram}), then:
\[ (J)\cap A = \left( \prod_{Q \in \Spram(B)} Q^{e_Q} \right) \qquad
\text{(as ideal of } A\text{).}\]
\end{theo}

\begin{remark}
  The notation $W$ for the extension is intentionally chosen to emphasize
  the analogy with the case when the extension is Galois, \ie when it is
  given by the action of a reflection group $W$ on the polynomial algebra
  $B$. Here there is not necessarily a group acting, but some features of
  the Galois case remain. That is why David Bessis proposes to call the
  extension $W$ a \emph{virtual reflection group}\footnote{David Bessis,
    personal communication.}. The polynomial $\prod Q^{e_Q}$ in the theorem
  above could then be called the \emph{discriminant} of the virtual
  reflection group $W$. One can wonder whether the analogies can go
  further, and to what extent it is possible to construct an ``invariant
  theory'' for virtual reflection groups.

  In the next chapter, we will prove that $\LL$ extensions are
  well-ramified, and give a list
  of the analogies between the $\LL$ and the Galois situations
  (Sect.\ \ref{partLLvirtual}).
\end{remark}

\begin{remark}
  The first part of Thm.\ \ref{thmintrojac} is a quite easy consequence of
  commutative algebra properties. It is probably folklore, but I could not
  find the formula stated anywhere. I thank Raphaël Rouquier for his
  suggestion to use the different ideal.
\end{remark}

\section{Jacobian and different of a finite graded polynomial
 extension}

\subsection{General setting and notations}
~\\
\label{subpartsetting}
Let $n$ be a positive integer, and denote by $B$ the graded polynomial
algebra $\BC[X_1,\dots, X_n]$, where $X_1,\dots, X_n$ are indeterminates of
respective weights $b_1,\dots, b_n$.

Let us consider $n$ weighted homogeneous polynomials $f_1,\dots, f_n$ in $B$ (of
respective weights $a_1,\dots,a_n$), and the resulting
(quasi-homogenenous) mapping  
\[\begin{array}{lccc}
    f : & \BC^{n} & \to & \BC^{n}\\
& (x_1,\dots, x_n) & \mapsto & (f_1(x_1,\dots, x_n),\dots, f_n(x_1,\dots,
x_n)).
  \end{array} \]

We denote by $A$ the
algebra $\BC[f_1,\dots, f_n]$, so that we have a ring extension $A\subseteq B$.

\begin{defi}
\label{defext}
In the above situation, if $B$ is an $A$-module of finite type, we will
call $A\subseteq B$ a \emph{finite graded polynomial extension}.
\end{defi}

\begin{remark}
  In this setting, the extension is finite if and only if
  $f^{-1}(\{0\})=\{0\}$, because $f$ is a quasi-homogeneous morphism (see
  for example \cite[Thm.5.1.5]{LZgraphs}). Moreover, the rank of $B$ over
  $A$ (or the \emph{degree} of $f$) is then equal to $r:=\prod a_i / \prod
  b_i$.
\end{remark}

\bigskip

The algebra $B$ is Cohen-Macaulay, and is finite as an $A$-module, so $B$
is also a free $A$-module of finite type.  Thus $A\subseteq B$ is a finite
free extension of UFDs. We denote by $K$ and $L$ the fields of fractions of
$A$ and $B$. Let us recall some notations and properties about the
ramification in this context (for example see \cite[Chap. 3]{benson}).

\medskip

If $\fq$ is a prime ideal of $B$, then $\fp=\fq \cap A$ is a prime ideal of
$A$. In this situation we say that $\fq$ lies \emph{over} $\fp$. By the
Cohen-Macaulay theorem \cite[Thm.1.4.4]{benson}, $\fq$ has height one if
and only if $\fp$ has height one. In this case, we can write $\fq=(Q)$ and
$\fp=(P)$, where $P\in A$ and $Q \in B$ are irreducible, and $(Q)\cap
A=(P)$.

The ramification index of $\fq$ over $\fp$ is:
\[ e(\fq,\fp):= v_Q (P)\]
(we will rather write simply $e_\fq$ or $e_Q$).

\medskip

We denote by $\Spec_1(B)$ (resp. $\Spec_1(A)$) the set of prime ideals of
$B$ (resp. $A$) of height one. It is also the set of irreducible
polynomials in $B$, up to association. By abusing the notation, in indices
of products, we will write ``$Q \in \Spec_1(B)$'' instead of ``$\fq\in
\Spec_1(B)$ and $Q$ is one polynomial representing $\fq$'', and ``$Q$ over
$P$'' instead of ``$(Q)$ over $(P)$''.

We have the following elementary property:

\begin{lemma}
  Let $(P) \in \Spec_1(A)$ and $(Q)\in\Spec_1(B)$. Then, $(Q)$ lies
  over $(P)$ if and only if $Q$ divides $P$ in
  $B$. Thus, for $(P)\in\Spec_1(A)$, we have:
  \[ P \doteq \prod_{Q \text{\ over\ } P} Q^{e_Q} \ .\]
\end{lemma}

\subsection{Different ideal and Jacobian}

As our setup is compatible with that of \cite[Part.3.10]{benson}, we can
construct the \emph{different $\FD_{B/A}$} of the extension. Let us recall
that it is defined from the inverse different :
\[ \FD_{B/A}^{-1} := \{x \in L \tq \forall y \in B, \Tr_{L/K}(xy) \in A
\}\ ,\]
where $K$ and $L$ are the fields of fractions of $A$ and $B$, $L$ is
regarded as a finite vector space over $K$, and $\Tr_{L/K}(u)$
denotes the trace of the endomorphism $(x\mapsto ux)$.

The different $\FD_{B/A}$ is then by definition the inverse fractional
ideal to $\FD_{B/A}^{-1}$. It is a (homogeneous) divisorial ideal; in our
setting, as $B$ is a UFD, $\FD_{B/A}$ is thus a principal ideal. We will
see below that the different is simply generated by the Jacobian $J_{B/A}$
of the extension. For now, let us denote by $\theta_{B/A}$ a homogeneous
generator of $\FD_{B/A}$.

\medskip

The different satisfies the following:

\begin{propo}[{\cite[Thm.3.10.2]{benson}}]
\label{propcaracdiff}
If $\fq$ and $\fp=A\cap \fq$ are prime ideals of height one in $B$ and $A$,
then $e(\fq,\fp)>1$ if and only if $\FD_{B/A} \subseteq \fq$.

In other words: if $Q$ is an irreducible polynomial in $B$, then $e_Q>1$ if
and only if $Q$ divides $\theta_{B/A}$.
\end{propo}

We define the set of ramified ideals:
 \[ \Spram(B):=\{\fq \in \Spec_1(B) \tq e_\fq >1 \}, \]
 which can also be seen as a
system of representatives of the irreducible polynomials $Q$ in $B$ which are
ramified over $A$. By the above theorem, we have:
\[ \theta_{B/A} \doteq \prod_{Q \in \Spram(B)} Q ^{v_Q(\theta_{B/A})} \ . \]

This can be refined as:

\begin{propo}
\label{propdiff}
For all irreducible $Q$ in $B$, we have: $v_Q(\theta_{B/A})=e_Q -1$. Thus:
\[ \theta_{B/A} \doteq \prod_{Q \in \Spram(B)} Q ^{e_Q -1}. \]
\end{propo}

\begin{proof}
  We localize at $(Q)$ in order to obtain local Dedekind domains. Then we
  can use directly Prop.\ 13 in \cite[Ch. III]{serre}. 
\end{proof}

Let us show that the different is actually generated by the Jacobian
determinant. For this, we need to introduce the \emph{Kähler different} of
the extension $A \subseteq B$. According to \cite{diff} (end of first
section), when $B$ is a polynomial algebra, the Kähler different can be
defined as the ideal generated by the Jacobians of all $n$-tuples of
elements of $A$, with respect to $X_1,\dots,X_n$. But here we are in an
even more specific situation, where $A$ is also a polynomial ring. Thus,
whenever we take $g_1,\dots, g_n \in A=\BC[f_1,\dots,f_n]$, we have
\[ \det \left( \frac{\partial g_i}{\partial X_k} \right)_{\substack{1 \leq
    i\leq n\\1\leq k \leq n}} = \det \left( \frac{\partial g_i}{\partial f_j}
  \right)_{\substack{1 \leq i\leq n\\1\leq j \leq n}} \det \left( \frac{\partial
        f_j}{\partial X_k} \right)_{\substack{1 \leq j\leq n\\1\leq k \leq n}} \ ,\] so
      the Kähler different is simply the principal ideal of $B$ generated
      by the polynomial
\[ J_{B/A}:=\Jac ((f_1,\dots, f_n) / (X_1,\dots, X_n))= \det \left(
  \frac{\partial f_i}{\partial X_j} \right)_{\substack{1 \leq i\leq
    n\\1\leq j \leq n}} .\]

\begin{propo}
\label{propdiffjac}
With the hypothesis above, we have:
\[ \theta_{B/A} \doteq J_{B/A} .\]
\end{propo}

\begin{proof}
  In \cite{diff}, Broer studies several notions of different ideals, and
  proves that under certain hypothesis they are equal. We are here in the
  hypothesis of his Corollary 1, which states in particular that the Kähler
  different is equal to the different $\FD_{B/A}$.
\end{proof}

Note that we use here a strong result (which applies in much more
generality than what we need). It should be possible to give a simpler
proof. For the sake of completeness, we add here a more explicit proof
of the fact that the polynomials $J_{B/A}$ and $\theta_{B/A}$ have the same
degree (using a ramification formula of Benson).

\begin{lemma}
With the hypothesis and notations above, we have:
\[\deg (\theta_{B/A}) =\deg (J_{B/A}). \]
\end{lemma}

\begin{proof}
 First we need to recall some notations for graded algebras, see
 \cite[Ch.2.4]{benson}. If $A$ is a graded algebra, and $M$ a graded
 $A$-module, we denote the usual Hilbert-Poincaré series of $M$ by $\grdim
 M$ (for ``graded dimension''): $\grdim M := \sum_k \dim M_k t^k $. If
 $n$ is the Krull dimension of $A$, we define rational numbers $\deg(M)$
 and $\psi(M)$ by the Laurent expansion about $t=1$:
 \[ \grdim(M)= \frac{\deg(M)}{(1-t)^n} + \frac{\psi(M)}{(1-t)^{n-1}}+
 o\left(\frac{1}{(1-t)^{n-1}}\right).\]

 \medskip 

 Now if we return to our context we can use the following ramification
 formula from \cite[Thm 3.12.1]{benson}:
 \[ |L:K|\psi(A) - \psi(B) = \frac12 \sum_{\fp \in \Spec_1(B)} v_\fp
 (\FD_{B/A}) \psi (B/\fp) .\] We have $\grdim A= \prod_{i=1}^n
 \frac{1}{1-t^{a_i}}$, so by computing the derivative of $(1-t)^n\grdim A$
 at $t=1$ we get:
 \[ \psi(A)=\frac{1}{\prod_i a_i} \sum_i \frac{a_i-1}{2} \ , \]
 and similarly for $\psi(B)$. As $|L:K|=\prod a_i / \prod b_i$, we obtain:
 \[ |L:K|\psi(A) - \psi(B)= \frac{1}{\prod_i b_i} \sum_i \frac{a_i-b_i}{2}
 = \frac{1}{2\prod_i b_i} \deg J_{B/A}.\]
 On the other hand, for $\fp \in \Spec_1(B)$, if $d$ denotes the degree of a
 homogeneous polynomial $P$ generating $\fp$, we have
 \[ \grdim B/ \fp= (1-t^d)\prod_i \frac1{1-t^{b_i}} \ ,\]
 so after computation we get
 \[ \psi(B/ \fp)= \frac{d}{\prod_i b_i} \ . \]
 Thus, the rhs of the ramification formula becomes 
 \[ \frac1{2\prod_i b_i} \sum_{P \in \Spec_1(B)} v_P
 (\theta_{B/A}) \deg P = \frac1{2\prod_i b_i} \deg \theta_{B/A} \]
 and we can conclude that $\deg J_{B/A}=\deg \theta_{B/A}$.
\end{proof}

\bigskip

As a direct consequence of Propositions \ref{propdiff} and
\ref{propdiffjac}, we obtain the factorisation of $J_{B/A}$.

\begin{theo}
\label{thmjac}
If $A\subseteq B$ is a finite graded
polynomial extension, then we have, with the notations above:
\[ J_{B/A} \doteq \prod_{Q \in \Spram(B)} Q ^{e_Q -1}. \]
\end{theo}

This formula settles the first part of Thm.\ \ref{thmintrojac}. In section
\ref{partwellram} we will define the well-ramified property, in order to
deal with its second part.

\section{Geometric properties}

\label{partgeom}

In this section we recall some definitions and elementary facts about the
ramification locus and branch locus of a branched covering. These will be
used in section \ref{partwellram}, in order to give geometric interpretations of the
well-ramified property (Def.\ \ref{defwellram} and Prop.\ \ref{propwell}).

\subsection{Ramification locus and branch locus}
\label{subpartlocusnew}

Let us define the varieties $U=\Spec A$ and $V=\Spec B$, so that to the
extension $A \subseteq B$ (as above), of degree $r$, corresponds an
algebraic quasi-homogoneous finite morphism $f:V \to U$. We
denote by $J$ the Jacobian of $f$.

\begin{propo}
In $V$, we have equalities between:
\begin{enumerate}[(i)]
\item the set of points where $f$ is not étale, \ie zeros of the Jacobian $J$;
\item the union of the sets of zeros of the ramified polynomials of $B$;
\item the set of zeros of the generator $\theta_{B/A}$ of the different
  $\FD_{B/A}$.
\end{enumerate}
\end{propo}

This set is called the \emph{ramification locus} $\Vram$.

\begin{proof}
  (ii)=(iii) comes from Prop.\ \ref{propcaracdiff}, and (iii)=(i) from
  Prop.\ \ref{propdiffjac}.
\end{proof}

\medskip

The following (well-known) proposition gives an upper bound for the
cardinality of the fibers of $f$.

\begin{propo} For all $u\in U$, $|f^{-1}(u)| \leq r$, where $r$ is the
  degree of $f$.
\end{propo}

We recall the proof for the convenience of the reader.

\begin{proof}
  (After \cite[II.5.Thm.6]{shaf}.) Let $u$ be in $U$, and write
  $f^{-1}(u)=\{v_1,\dots, v_m\}$. One can easily find an element $b$ in
  $B=\CO_V$ such that all the values $b(v_i)$ are distinct.
  
  As $f$ is finite, $B$ is a module of finite type (of rank $r$) over
  $A$. Thus, every element in $B$ is a root of a unitary polynomial with
  coefficients in $A$, and degree less than or equal to $r$. Let $P\in
  A[T]$ be such a polynomial for $b$, and write $P=\sum_{i=0}^{d}a_iT^i$,
  with $d \leq r$ and $a_d=1$. We have
  \[ \sum_{i=0}^{d}a_ib^i =0 .\] For $j\in \{1,\dots, m\}$, as $f(v_j)=u$,
  specializing the above identity at $v_j$ gives
  \[ \sum_{i=0}^{d}a_i(u) b(v_j)^i =0 .\] So the polynomial $\sum_i a_i(u)
  T^i$ has $m$ distinct roots $b(v_1),\dots , b(v_m)$, and has degree $d$,
  so we obtain $m \leq d \leq r$.
\end{proof}

Points in $U$ whose fiber does not have maximal cardinality are called
\emph{branch points}, and form the \emph{branch locus} of $f$ in $U$:
\[ \Ubr := \left\{ u \in U \ ,\  \left|f^{-1}(u)\right| <r \right\} \ .\]

It is easy to show that $\Ubr$ is closed for the Zariski topology
and is not equal to $U$ (cf. \cite[II.5.Thm.~7]{shaf}), so that
$U-\Ubr$ is dense in $U$.

\begin{propo}
\label{proprambr}
  With the notations above, we have the following equality:
  \[ f(\Vram) = \Ubr \ .\] 

  So $\Vram\subseteq f^{-1}(\Ubr)$. Moreover, the
  restriction of $f$: \[V- f^{-1}(\Ubr)\ \surj \ U-
  \Ubr \] is a topological $r$-fold covering (for the
  transcendental topology).
\end{propo} 

\begin{proof}
  We set $U':= U- \Ubr$ and $V':=V-f^{-1}(\Ubr)$;
  these are Zariski-open.

  \medskip

  First, let $u$ be in $\Ubr$. As $U'$ is dense in $U$, we can
  find a sequence $u^{(k)}$ of elements in $U'$, whose limit is $u$. Let us
  write $f^{-1}(u)=\{v_1,\dots, v_p\}$ (with $p <r$), and
  $f^{-1}(u^{(k)})=\{v_1^{(k)},\dots, v_r^{(k)}\}$. Up to extracting
  subsequences, we can assume that each sequence $(v_i^{(k)})_{k \in \BN}$
  converges towards one of the $v_j$. As $r<n$, we have at least two
  sequences, say $v_1^{(k)}$ and $v_2^{(k)}$, whose limit is the same
  element of $f^{-1}(u)$, say $v_1$. But for all $k$, we have $v_1^{(k)}
  \neq v_2^{(k)}$ and $f(v_1^{(k)})= f(v_2^{(k)})=u^{(k)}$. If $J(v_1)\neq
  0$, this contradicts the inverse function theorem. So $J(v_1)=0$, and $u
  \in f(Z(J))=f(\Vram)$.
  
  \medskip

  By \cite[II.5.3.Cor.~2]{shaf}, if $f$ is unramified at $u$, then for all
  $v\in f^{-1}(u)$, the tangent mapping of $f$ at $v$ is an
  isomorphism. That means exactly that $f^{-1}(U- \Ubr) \subseteq
  V-Z(J)$, \ie $f(Z(J)) \subseteq \Ubr$.

  \medskip
  
  Let us fix an element $u$ in $U'$. For all $v$ in
  $f^{-1}(u)$, we know that $J(v)\neq 0$, so, by the inverse function
  theorem, there exists a neighbourhood $E_v$ of $v$ in $V'$ such that the
  restriction $f : E_v \to f(E_v)$ is an isomorphism. Let us write
  $f^{-1}(u)=\{v_1,\dots, v_r \}$. Obviously we can suppose that the
  $E_{v_i}$'s are pairwise disjoint. Then $\Omega_u:=f(A_{v_1})\cap
  \dots \cap f(A_{v_r})$ is a neighbourhood of $u$ in $U'$, and for $x$ in
  $f^{-1}(\Omega_u)$ there exists a unique $i$ such that $x$ is in
  $A_{v_i}$. Thus we get natural map $f^{-1}(\Omega_u) \to \Omega_u \times
  \{1,\dots, r\}$, $x \mapsto (f(x),i)$, which is clearly a homeomorphism.
\end{proof}

\subsection{Ramification indices of a branched covering}
\label{subpartbranched}

As explained in the book of Namba\footnote{Thanks to José Ignacio Cogolludo
  for having suggested this reference to me.}(\cite[Ex. 1.1.2]{branched}),
the map $f$ is more precisely a \emph{finite branched covering} of $U$
(according to \cite[Def.\ 1.1.1]{branched}), of degree $r$.

\medskip

Because of the inclusion $Z(J)=\Vram \subseteq f^{-1}(\Ubr)$, we know that
the irreducible components of $f^{-1}(\Ubr)$ are the $Z(Q)$, for $Q \in
\Spram(B)$, plus possibly some other components associated to unramified
polynomials. We recall here, for the sake of completeness, some classical
properties, thanks to which the ramification can be understood via the
cardinality of the fibers (this will be needed in the next chapter).

\begin{propo}[After Namba]
  \label{propfibers}
  Let $u$ be a non-singular point of $\Ubr$, and $v\in f^{-1}(u)$. Then:
  \begin{enumerate}[(a)]
  \item $v$ is non-singular in $f^{-1}(\Ubr)$. In particular,
    there is a unique irreducible component $C_v$ of
    $f^{-1}(\Ubr)$ containing $v$.
  \item There exists a connected open neighbourhood $\Omega_v$ of $v$ such
    that, for the restriction of $f$:
    \[ \widetilde{f}: \Omega_v \to f(\Omega_v) \ , \] for each $u'$ in
    $f(\Omega_v)\cap (U-\Ubr)$, the cardinality of the fiber
    $\widetilde{f}^{-1}(u')$ is the ramification index $e_\fq$ of the ideal
    $\fq \in \Spec_1(B)$ defining $C_v$.
  \end{enumerate}
\end{propo}

\begin{proof}
  We refer to Theorem 1.1.8 and Corollary 1.1.13 in \cite{branched}.
\end{proof}

\section{Well-ramified extensions}
\label{partwellram}
\subsection{The well-ramified property}
\label{subpartwellram}

As above we consider a finite graded polynomial extension $A \subseteq B$,
given by a morphism $f$. We denote by $J$ its Jacobian, and by $e_Q$ the
ramification index of a polynomial $Q\in\Spec_1(B)$.

\begin{defi}\label{defwellram}
We say that the extension  $A \subseteq B$, or the morphism $f$, is
\emph{well-ramified}, if:
\[ \left( J  \right) \cap A = \left( \prod_{Q \in \Spram(B)} Q
^{e_Q} \right) \ . \]
\end{defi}

Flagrantly, this definition makes the second part of Thm.\ \ref{thmintrojac}
a tautology. So our terminology might at this point seem quite
mysterious. Actually the remaining of this chapter will be concerned with
giving several equivalent characterizations of well-ramified extensions
(Prop.\ \ref{propwell}), that ought to make this terminology (and the
usefulness of this notion) much more transparent.

\subsection{Characterisations of the well-ramified property}

Most of the characterisations given below are very elementary, but are
worth mentioning so as to get a full view of what is a well-ramified
extension. 

\begin{propo}
  \label{propwell}
  Let $A \subseteq B$ a finite graded polynomial extension, and $f : V \to
  U$ its associated morphism. The following properties are equivalent:
  \begin{enumerate}[(i)]
  \item the extension $A\subseteq B$ is well-ramified
    (as defined in \ref{defwellram});
  \item $\displaystyle{\left( \prod_{Q \in \Spram(B)} Q \right) \cap A = \left(
      \prod_{Q \in \Spram(B)} Q ^{e_Q} \right)}$;
  \item the polynomial $\prod_{Q \in \Spram(B)} Q ^{e_Q} $ lies in $A$;
  \item for any $\fp \in \Spec_1(A)$, if there exists $\fq_0 \in
    \Spec_1(B)$ over $\fp$ which is ramified, then any other $\fq \in
    \Spec_1(B)$ over $\fp$ is also ramified;
  \item if $P$ is an irreducible polynomial in $A$, then, as a polynomial
    in $B$, either it is reduced, or it is completely non-reduced, \ie any of
    its irreducible factors appears at least twice;
  \item $f^{-1}(\Ubr )=\Vram$;
  \item $f(\Vram) \cap f(V- \Vram) = \varnothing$.
  \end{enumerate}
\end{propo}

\begin{proof}
  Let $S:=\prod_{Q \in \Spram(B)} Q $, and $R:=\prod_{Q \in \Spram(B)} Q
  ^{e_Q}$. From Thm.\ \ref{thmjac}, we have also: $J = \prod_{Q \in \Spram(B)} Q
  ^{e_Q-1}$.
  
  We begin with some general elementary facts. We have $(S)\cap
  A=\bigcap_{Q \in \Spram(B)} (Q)\cap A$. For $Q \in \Spram(B)$, denote by
  $\widetilde{Q}$ an irreducible in $A$ such that $Q$ lies over
  $\widetilde{Q}$ (\ie $(Q)\cap A = \widetilde{Q}$). As we work in UFDs, we
  get that $(S)\cap A$ is principal, generated by
  \[ \widetilde{S} := \lcm \left(\widetilde{Q} \tq Q\in \Spram(B)\right)\
  .\]

  For $Q \in \Spram(B)$, $Q^{e_Q}$ divides $\widetilde{Q}$, so $R$ divides
  $\widetilde{S}$. Moreover, $J$ divides $R$, so: \\${\widetilde{S} \subseteq
    (R)\cap A \subseteq (J)\cap A}$. Conversely, $S$ divides $J$, so:
  $(J)\cap A \subseteq (S)\cap A= (\widetilde{S})$.
  
  Thus we always have $(J)\cap A = (S)\cap A$, and the statement $(ii)$ is
  just an alternate definition of the well-ramified property:
  $(i)\Leftrightarrow (ii)$. 

  \medskip

  $(iii)\Leftrightarrow (ii)$: we have $(S)\cap A \subseteq (R)$ and $R\in
  (S)$. So $R$ lies in $A$ if and only if $(S)\cap A =(R)$.

  \medskip 

  $(v)\Leftrightarrow (iv)$, since $(v)$ is just a ``polynomial''
  rephrasing of $(iv)$.

  \medskip
  
  $(iv)\Leftrightarrow (iii)$: let us denote by
  $\Spram(A)$ the set of primes $\fp$ in $\Spec_1(A)$ such that there
  exists at least one prime $\fq$ over $\fp$ which is ramified. Then:
  \begin{equation*} R=\prod_{Q \in \Spram(B)} Q ^{e_Q} =\prod_{P \in \Spram(A)}
  \prod_{\substack{Q \in \Spram(B)\\Q \text{ over }P}} Q
  ^{e_Q} \ .\tag{*} \end{equation*} 
  If we suppose $(iv)$, then whenever $P$ is in $\Spram(A)$,
  all the $Q$ over $P$ are ramified. Thus:
  \[R= \prod_{P \in \Spram(A)} \ \prod_{Q \text{\ over\ } P} Q ^{e_Q} =
  \prod_{P \in \Spram(A)} P \] and the polynomial $R$ lies in $A$.

  Conversely, suppose that $R$ lies in $A$. Consider $P$ in $\Spram(A)$,
  and $Q$ in $\Spec_1(B)$ lying over $P$. Then there exists $Q_0$ in
  $\Spram(B)$ such that $(Q_0)\cap A = (P)=(Q)\cap A$. As $(R)$ is
  contained in $(Q_0)\cap A$, we obtain that $Q$ also divides $R$, so is
  among the factors of the product (*) above. Thus $(Q)$ is ramified, and
  $(iv)$ is verified.

  \medskip
  
  $(vii)\Rightarrow (vi)$: if $v$ lies in $f^{-1}(\Ubr)$, we have
  $f(v)\in \Ubr=f(\Vram)$, so $(vii)$ implies $v\in
  \Vram$. Thus $f^{-1}(\Ubr) \subseteq
  \Vram$. The other inclusion is from Prop.\ \ref{proprambr}.

  \medskip
  
  $(i) \Rightarrow (vii)$: we have $(\widetilde{S})=(J)\cap A$, so that
  $Z(\widetilde{S}) = f(Z(J))$. If the extension is well-ramified, we
  obtain: $\widetilde{S}\doteq \prod_{Q \in \Spram(B)} Q^{e_Q}$. Thus $J$
  and $\widetilde{S}$ have the same irreducible factors in $B$, which
  implies $Z(J)=f^{-1} (Z(\widetilde{S}))$. So
  $f(\Vram)=Z(\widetilde{S})$, whereas $f(V- \Vram)=U -
  Z(\widetilde{S})$.

  \medskip

  $(vi)\Rightarrow (v)$: suppose that there exists $Q$ in $\Spram(B)$,
  such that $\widetilde{Q}$ has one irreducible factor (in $B$) which is not a
  ramified polynomial, say $M$. Then we can choose $v$ in $Z(M)-Z(J)$. As
  $M(v)=0$, we have $\widetilde{Q}(f(v))=0$ and $\widetilde{S}(f(v))=0$. So
  $f(v)\in Z(\widetilde{S})=f(Z(J))=\Ubr$, which contradicts $(i)$.
\end{proof}

\subsection{Examples and counterexamples}

A fundamental case when the extension is well-ramified is the Galois
case. If $A=B^G$, with $G$ a (reflection) group, all the ramification
indices of the ideals over a prime of $A$ are the same. In our setting, the
well-ramified property is somewhat a weak version of the normality (the
extension is always separable since we work in characteristic zero).

\medskip

Of course the notion is strictly weaker than being a Galois extension. Take
for example the extension $\BC[X^2+Y^3,X^2Y^3] \subseteq \BC[X,Y]$, with
Jacobian $J=6XY^2(X^2-Y^3)$: the ramified polynomials are $(X)$ (index $2$)
and $(Y)$ (index $3$), both above $(X^2Y^3)$, and $(X^2-Y^3)$ (index $2$),
above $((X^2-Y^3)^2)$.

\bigskip

A simple example of a not well-ramified extension is:
\[A=\BC[X^2 Y,X^2 +Y] \subseteq \BC[X,Y]=B,\] which is free of rank $4$.
Here the ideal $(X^2 Y)$ in $A$ has two ideals above in $B$: $(X)$ which
is ramified and $(Y)$ which is not, so the extension is not
well-ramified. We compute $\theta_{B/A}=\Jac(f)=X(Y-X^2)=S$ (using the
notations of the proof of Prop.\ \ref{propwell}). So $R=X^2
(Y-X^2)^2$ is not in $A$, actually $(S)\cap A$ is generated by $X^2 Y
(Y-X^2)^2$. 

\bigskip

In the next chapter we will show that the Lyashko-Looijenga
extensions are well-ramified but not Galois (Thm.\ \ref{thmLL}).

\chapter[Geometry of LL and submaximal factorisations]{Geometry of the Lyashko-Looijenga morphism and submaximal
  factorisations of a Coxeter element}
\label{chapsubmax}

In this chapter, we apply the results of Chapter \ref{chapjac} to the
Lyashko-Looijenga morphism $\LL$. We prove that $\LL$ is a well-ramified
morphism (according to Def.\ \ref{defwellram}), and describe the
ramified polynomials in terms of the geometry of the discriminant
hypersurface (see Thm.\ \ref{thmLL}). Then we use the constructions of
Chapter \ref{chaphur} to deduce combinatorial results about the
factorisations of a Coxeter element in $(n-1)$ blocks.

\section{Lyashko-Looijenga morphisms and factorisations of a 
Coxeter element}
\label{partLL1}~

Although we refer to Chapter \ref{chaphur} for the precise definitions of
the morphism $\LL$ and the block factorisations of a Coxeter element $c$,
we do recall here a few useful constructions about the strata of $\CH$, the
map $\fact$, and the relations between $\fact$ and $\LL$, so that this
chapter should be intelligible without reading Chapter \ref{chaphur},
except for the definitions of section \ref{partLL}.

\bigskip

The notations are the same as in Chapter \ref{chaphur}. We fix a
well-generated, irreducible (finite) complex reflection group $W$, and we
choose invariant polynomials $f_1,\dots, f_n$, homogeneous of degrees
$d_1\leq \dots \leq d_n=h$, such that the discriminant of $W$ has the form:
\[\Delta_W= f_n^n + a_2 f_n ^{n-2} +\dots + a_n \ ,\]
with $a_i \in \BC[f_1,\dots, f_{n-1}]$. The Lyashko-Looijenga morphism is then:
\[\begin{array}{lccc}
    \LL : & \BC^{n-1} & \to & \BC^{n-1}\\
    &(f_1,\dots, f_{n-1}) & \mapsto & (a_2,\dots, a_n)
  \end{array} \]
We also use the notation $\LL$ to denote the set-theoretical incarnation of
this algebraic morphism, \ie the map
\[\begin{array}{rcl}
Y & \to & E_n\\
y=(f_1,\dots, f_{n-1}) & \mapsto & \text{multiset\ of\ roots\ of\ }
\Delta_W(f_1,\dots, f_n) \ ,
\end{array}\]
where $E_n$ is the set of centered configurations of $n$ points in $\BC$.

\subsection{Discriminant stratification}
The space $V$, together with the hyperplane arrangement $\CA$, admits a
natural stratification by the \emph{flats}, elements of the intersection
lattice $ \CL:=\left\{\bigcap_{H\in \CB} H \tq \CB \subseteq \CA \right\} $.

As the $W$-action on $V$ maps flats to flats, this stratification gives
rise to a quotient stratification $\Lb$ of $W \qg V$: \[\Lb= W \qg \CL
=(p(L))_{L \in \CL}= (W\cdot L)_{L \in \CL} \ ,\] where $p$ is the projection
$V \ \surj \ W \qg V $ . 
For each stratum $\Lambda$ in $\Lb$, we denote by $\Lambda^0$ the
complement in $\Lambda$ of the union of the strata strictly included in
$\Lambda$. The $(\Lambda^0)_{\Lambda \in \Lb}$ form an open stratification
of $W \qg V$, called the \emph{discriminant stratification}.

\medskip

There is a natural bijection between the set of flats in $V$ and the set of
parabolic subgroups of $W$ (Steinberg's theorem); this leads to a
bijection between the stratification $\Lb$ and the set of conjugacy classes
of parabolic subgroups. Moreover, $\Lb$ is in bijection with the set of
conjugacy classes of \emph{parabolic Coxeter elements} (which are Coxeter
elements of parabolic subgroups). Through these bijections, the codimension
of a stratum $\Lambda$ corresponds to the rank of the associated parabolic
subgroup, and to the length of the parabolic Coxeter element. We refer to
Sect.\ \ref{partecp} for details and proofs.

\subsection{Geometric factorisations and compatibilities}
\label{subpartstrat}

In Sect.\ \ref{partlbl}, we constructed, using the topology of $\CH
\subseteq W \qg V \simeq Y \times \BC$, a map:
\[\begin{array}{rcl}
  \CH & \to & W \\
  (y,x) & \mapsto & c_{y,x} \ ,
\end{array}
\]
which satisfies the two fundamental properties (note that $(y,x)$ lies in $\CH$
if and only if the multiset $\LL(y)$ contains $x$):
\begin{itemize}
\item[(P1)] if $(x_1,\dots, x_p)$ is the ordered support of $\LL(y)$ (for
  the lexicographical order on $\BC \simeq \BR^2$), then the $p$-tuple
  $(c_{y,x_1}, \dots, c_{y,x_p})$ lies in $\Fact_p(c)$;
\item [(P2)] for all $x \in \LL(y)$, $c_{y,x}$ is a parabolic Coxeter
  element; its length is equal to the multiplicity of $x$ in $\LL(y)$, and
  its conjugacy class corresponds to the unique stratum $\Lambda$ in $\Lb$
  such that $(y,x)\in \Lambda^0$.
\end{itemize}

According to property (P1), we call \emph{factorisation of $c$ associated
  to $y$}, and denote by $\fact(y)$, the tuple $(c_{y,x_1}, \dots,
c_{y,x_p})$ (where $(x_1,\dots, x_p)$ is the ordered support of $\LL(y)$).

Any block factorisation determines a composition of $n$. We can also
associate to any configuration of $E_n$ a composition of $n$, formed by the
multiplicities of its elements in the lexicographical order. Then property
(P2) implies that for any $y$ in $Y$, the compositions associated to
$\LL(y)$ and $\fact(y)$ are the same. The third fundamental property
(Thm.\ \ref{thmbij}) is:
\begin{itemize}
\item[(P3)] the map $\LL \times \fact : Y \to E_n \times \Fact(c)$ is
  injective, and its image is the entire set of compatible pairs (\ie with
  same associated composition).
\end{itemize}

In other words, for each $y\in Y$, the fiber $\LL^{-1}(\LL(y))$ is in
bijection (via $\fact$) with the set $\Fact_\mu (c)$ where $\mu$ is the
composition of $n$ associated to $\fact(y)$.

\section{Lyashko-Looijenga extensions}
\label{partLL2}
\subsection{\texorpdfstring{Ramification locus for $\LL$}{Ramification locus for LL}} 
~

Let us first explain the reason why $\LL$ is étale on $Y-\CK$ (as stated in
Thm.\ \ref{thmLLBessis}), where  
\[\CK = \{y\in Y \tq \text{ the multiset } \LL(y) \text{ has multiple
  points} \} \ .\] The argument goes back to Looijenga (in
\cite{looijenga}), and is used without details in the proof of Lemma 5.6 of
\cite{BessisKPi1}.

\medskip

We begin with a more general setting. Let $n\geq 1$, and $P \in
\BC[T_1,\dots, T_n]$ of the form:
\[ P= T_n^n + a_2 (T_1,\dots, T_{n-1}) T_n^{n-2} + \dots + a_n(T_1,\dots
T_{n-1})  \] 
(here the polynomials $a_i$ do not need to be quasi-homogeneous).
As in the case of $\LL$ we define the hypersurface $\CH := \{P=0\}
\subseteq \BC^n$, and a map $\psi : \BC^{n-1} \to E_n$, sending
$y=(T_1,\dots, T_{n-1})\in \BC^{n-1}$ to the multiset of roots of
$P(y,T_n)$ (as a polynomial in $T_n$). This map can also be considered as
the morphism $y\mapsto (a_2(y),\dots, a_n(y))$.

We set:
\[J_\psi(y)=\Jac((a_2,\dots,a_n) / y)=\det \left(
  \frac{\partial a_i}{\partial T_j} \right)_{\substack{2 \leq i \leq n\\ 1\leq j \leq
  n -1}} \]

\begin{propo}[after Looijenga]
  \label{propgeneralposition}
  With the notations above, let $y$ be a point in $\BC^{n-1}$, with
  $\psi(y)$ being the multiset $\{x_1,\dots, x_n \}$. Suppose that the
  $x_i$'s are pairwise distinct.

  Then the points $(y,x_i)$ are regular on $\CH$. Moreover, the $n$
  hyperplanes tangent to $\CH$ at $(y,x_1), \dots , (y,x_n)$ are in
  general position if and only if $J_\psi(y)\neq 0$ (\ie $\psi$ is
  étale at $y$).
\end{propo}

\begin{proof}
  Let $\alpha$ be a point in $\CH$. If it exists, the hyperplane tangent to
  $\CH$ at $\alpha$ is directed by its normal vector: $ \grad_{\alpha}P =
  \left( \frac{\partial P}{\partial T_1}(\alpha), \dots, \frac{\partial
      P}{\partial T_n}(\alpha) \right)$.
 
  Let $y$ be a point in $\BC^{n-1}$ such that the $x_i$'s associated are
  pairwise dictinct. Then the polynomial in $T_n$ $P(y,T_n)$ has the
  $x_i$'s as simple roots, so for each $i$, $\frac{\partial P}{\partial
    T_n}(y,x_i) \neq 0$, and the point $(y,x_i)$ is regular on $\CH$.

  The tangent hyperplanes associated to $y$ are in general position if and
  only if $\det M_y \neq 0$, where $M_y$ is the matrix with columns:
  \[ \left(\grad_{(y,x_1)} P \ ;\ \dots \ ;\ \grad_{(y,x_n)} P\right).\]
  
  After computation, we get: $M_y= A_y \times V_y$, where

\[
A_y=\left[
  \begin{array}{c|c}
    \begin{array}{c} 0 \\ \vdots \\ 0 \end{array}
    & 
    \displaystyle{\left( \frac{\partial a_j}{\partial T_i} \right)_{\substack{1\leq i
            \leq n-1 \\ 2\leq j \leq n}} }
        \\
        \hline
     n  
     &
     \begin{array}{cccc}  0  &  (n-2)a_2(y)  &  \dots  &  a_{n-1}(y) \end{array}
   \end{array}
 \right]
 \text{\ and\ }
 V_y=\left[
   \begin{array}{ccc}
     x_1^{n-1} & \dots & x_n^{n-1}\\
     \vdots & \ddots & \vdots \\
     x_1 & \dots & x_n\\
     1 & \dots & 1
   \end{array}
 \right] .
\]

  As the $x_i$'s are distinct, the Vandermonde matrix $V_y$ is
  invertible. As $\det C_y = n J(y)$, we can conclude that $\det M_y \neq
  0$ if and only if $J_\psi(y)\neq 0$.
\end{proof}

If the $x_i$'s are not distinct, nothing can be said in general. But if
$\psi$ is a Lyashko-Looijenga morphism $\LL$, then we can deduce the
following property.

\begin{coro}
  \label{coroetale}
  Let $y$ be a point in $\BC^{n-1}$, and suppose that $\LL(y)$ contains $n$
  distinct points. Then $J_{\LL}(y)\neq 0$.

  In other words, $\LL$ is étale on (at least) $Y-\CK$.
\end{coro}

\begin{proof}
  Set $\LL(y)=\{x_1,\dots, x_n\}$. As the $x_i$'s are distinct, from lemma
  \ref{propgeneralposition} one has to study the hyperplanes tangent to
  $\CH$ at $(y,x_1), \dots , (y,x_n)$. By using their characterization in
  terms of basic derivations of $W$, it is straightforward to show that the
  $n$ hyperplanes are always in general position: we refer to the proof of
  \cite[Lemma 5.6]{BessisKPi1}.

\end{proof} 

In the following we will prove the equality $Z(J_{\LL})=\CK$, \ie that $\LL$ is
étale exactly on $Y-\CK$.

\subsection{\texorpdfstring{The well-ramified property for $\LL$}{The well-ramified property for LL}}
\label{subpartLLwell}

We know that $\LL$ is a finite quasihomogeneous map of degree
$n!h^n/|W|$. So we have, as in Chapter \ref{chapjac}, a graded finite
polynomial extension
\[ A=\BC[a_2,\dots,a_n] \subseteq \BC[f_1,\dots, f_n]=B ,\]
and we will use the same notations. In particular, $\LL$ is a branched
covering (of $E_n$, or of $\BC^{n-1}$), and we can use the terminology and the
properties of section \ref{partgeom}.

Theorem \ref{thmjac} gives
\[ J_{\LL}=\Jac ((a_2,\dots, a_n) / (f_1,\dots, f_{n-1}))= \prod_{Q \in
  \Spram(B)} Q ^{e_Q -1} \ . \]

Let us define
\[ D:= \Disc(f_n^n + a_2 f_n^{n-2} + \dots + a_n; f_n)\ , \] so that
$\CK=\LL^{-1}(E_n - \Enreg)$ is the zero locus of $D$ in $Y$. The
irreducible components of $\CK$ are naturally indexed by the conjugacy
classes of parabolic subgroups of $W$ of rank $2$. Let us recall from
Sect.\ \ref{subpartstratquotient} and \ref{subpartirred} some useful
properties.

We denote by $\Lb_2$ the set of all closed strata in $\Lb$ of codimension
$2$, and we define the map
\[\begin{array}{lccc}
\varphi :& W \qg V \simeq Y \times \BC & \to & Y\\
& \bar{v}=(y,x) & \mapsto & y
\end{array}\]

Then, using the notations and properties of section \ref{subpartstrat}, we
have:
\[\begin{array}{lll}
y \in \CK & \Leftrightarrow & \exists x \in \LL(y), \text{ with
  multiplicity} \geq 2 \\
& \Leftrightarrow & \exists x \in \LL(y), \text{ such that } \ell(c_{y,x})\geq 2\\
& \Leftrightarrow & \exists x \in \LL(y), \text{ such that } (y,x)\in \Gamma^0
\text{ for some stratum } \Gamma \in \Lb \text{ of codim.} \geq 2\\
 & \Leftrightarrow & \exists x \in \LL(y),\ \exists \Lambda \in \Lb_2 ,
 \text{ such that } (y,x)\in \Lambda\\
& \Leftrightarrow & \exists \Lambda \in \Lb_2 , \text{ such that } y\in
\varphi(\Lambda).
\end{array} \]

So the hypersurface $\CK$ is the union of the $\varphi(\Lambda)$, for $\Lambda
\in \Lb_2$; moreover, they are its irreducible components, according to
Prop.\ \ref{propcompirr}. Thus we can write
\[D= \prod_{\Lambda \in  \Lb_2} D_{\Lambda}^{r_{\Lambda}},\]
for some $r_\Lambda \geq 1$, where the $D_{\Lambda}$ are irreducible
polynomials in $B$ such that $\varphi(\Lambda)=Z(D_{\Lambda})$.

\medskip

We can now give an important interpretation of the integers $r_\Lambda$,
and prove that $\LL$ is a well-ramified morphism, according to
Def.\ \ref{defwellram}. 

\begin{theo}
  \label{thmLL}
  Let $\LL$ be the Lyashko-Looijenga extension associated to a
  well-generated, irreducible complex reflection group, together with the
  above notations.

  For any $\Lambda$ in $\Lb_2$, let $w$ be a (length $2$) parabolic
  Coxeter element of $W$ in the conjugacy class corresponding to
  $\Lambda$. Then $r_\Lambda$ is the number of reduced decompositions
  of $w$ in two reflections (in the case when $W$ is a $2$-reflection
  group, it is also the order of $w$).

  Moreover, we have:
  \begin{enumerate}[(a)]
  \item $\LL$ is a well-ramified extension;
  \item $\displaystyle{D = \prod_{\Lambda\in \Lb_2 } D_{\Lambda}
    ^{r_{\Lambda}} }$ is a generator for the ideal $(J_{\LL})\cap A$;
  \item $\displaystyle{J_{\LL}\doteq \prod_{\Lambda\in \Lb_2 } D_{\Lambda}
      ^{r_{\Lambda} -1}}$, and the ramified polynomials of $B$ are the
    $D_\Lambda$.
  \end{enumerate}
\end{theo}

\begin{proof}
The polynomial $D$ is irreducible in $A$ since, as a polynomial in
$a_2,\dots, a_n$, it is the discriminant of a reflection group of type
$A_{n-1}$. Therefore, for all $\Lambda$ in $\Lb_2$, the inclusion 
\[ (D_{\Lambda})\cap A \supseteq (D) \] is an inclusion between prime
ideals of height one in $A$. So we have $(D_{\Lambda})\cap A =(D)$, and
$e_{D_\Lambda}=v_{D_\Lambda}(D)=r_\Lambda$.

According to Corollary \ref{coroetale}, if $J_{\LL}(y)=0$, then
$\LL(y)\notin \Enreg$. So the variety of zeros of $J_{\LL}$ (defined by the
ramified polynomials in $B$) is included in the preimage \[\LL^{-1}(Z(D)) =
\bigcup_{\Lambda \in \Lb_2} Z(D_{\Lambda})\ .\] Thus all the ramified
polynomials in $B$ are in $\{D_{\Lambda}, \Lambda \in \Lb_2 \}$.

\medskip

Let $\Lambda \in \Lb_2$, and $\mu$ be the composition $(2,1,\dots, 1)$
of $n$. Choose $\xi=(w,s_3,\dots, s_n)$ in $\Fact_\mu (c)$ such that the
conjugacy class of $w$ (the only element of length $2$ in $\xi$)
corresponds to $\Lambda$. Fix $e\in E_n$, with composition type $\mu$, and
such that the real parts of its support are distinct. There exists a unique
$y_0$ in $Y$, such that $\LL(y_0)=e$ and $\fact(y_0)=\xi$ (Property (P3) in
Sect.\ \ref{subpartstrat}). Moreover this $y_0$ lies in $\varphi(\Lambda)$
(property (P2)). Using Definition \ref{deflbl} of the map
$\fact$, and the ``Hurwitz rule'' (Lemma \ref{reglehur}), we deduce that
for a sufficiently small connected neighbourhood $\Omega_0$ of $y_0$, if
$y$ is in $\Omega_0 \cap (Y- \CK)$, then $\fact(y)$ is in
\[ F_w:= \{(s_1',s_2',\dots, s_n ')\in \Red_\CR (c) \tq s_1's_2'=w
\text{ and } s_i'=s_i \ \forall i\geq 3 \} .\]
Let us fix $y$ in $\Omega_0 \cap (Y- \CK)$. Then, because of
property (P3), we get an injection
\[ \fact : \LL^{-1}(\LL(y)) \cap \Omega_0 \ \inj \ F_w \ .\] But this map
is also surjective, thanks to the covering properties of $\LL$ and the
transitivity of the Hurwitz action on $w$ (cf. Chapter
\ref{chaphur}). Indeed, we can ``braid'' $s_1'$ and $s_2'$ (by cyclically
intertwining the two corresponding points of $\LL(y)$, while staying in the
neighbourhood) so as to obtain any factorisation of $w$. Thus:
\[|\LL^{-1}(\LL(y))\cap \Omega_0| = |F_w|\ .\]

Using the characterization of the ramification index
(Prop.\ \ref{propfibers}), we infer
\[ r_\Lambda = e_{D_\Lambda}= |F_w| \ , \] which is the number of reduced
decompositions of $w$, \ie the Lyahko-Looijenga number for the parabolic
subgroups in the conjugacy class $\Lambda$.

For any rank $2$ parabolic subgroup with degrees $d_1',h'$, the
$\LL$-number is $2h'/d_1'$. In the particular case when $W$ is a
$2$-reflection group, such a subgroup is a dihedral group, hence
$r_\Lambda$ is also the order $h'$ of the associated parabolic Coxeter
element $w$.

\medskip

Consequently, for all $\Lambda \in \Lb_2$, $e_{D_\Lambda}$ is strictly
greater than $1$, so the $D_\Lambda$, $\Lambda \in \Lb_2$ are exactly
the ramified polynomials for the extension $\LL$. This directly implies
statement (c).

Moreover, we obtain that
\[ \prod_{Q \in \Spram(B)} Q ^{e_Q} = \prod_{\Lambda \in \Lb_2} D_\Lambda
^{e_{D_\Lambda}} = D_{\LL} \ , \] so it lies in $A$.  We recognize one of
our characterization of a well-ramified extension (namely
Prop.~\ref{propwell}(iii)), from which we deduce (a) and (b).
\end{proof}

\subsection{A more intrinsic definition of the Lyashko-Looijenga Jacobian}
~
\label{subpartintrinsic}

In this subsection we give an alternate definition for the Jacobian
$J_{\LL}$, which is more intrinsic, and which allows to recover a formula
observed by K. Saito.

\medskip

We will use the following elementary property.

Suppose $P\in \BC[T_1,\dots, T_{n-1},X]$ has the form:
\[ P = X^n + b_1 X^{n-1} + \dots + b_n \ , \]
with $b_1,\dots, b_n \in \BC[T_1,\dots, T_{n-1}]$. Note that we do not
require $b_1$ to be zero. Let us denote by $J(P)$ the polynomial:
\[ J(P):= \Jac \left( \left( P, \DP{P}{X}, \dots, \DP[n-1]{P}{X} \right)
  \middle/  (T_1,\dots, T_{n-1},X) \right) \ . \]

\begin{lemma}
  Let $P$ be as above. We set $Y=X+\frac{b_1}{n}$ and denote by $Q$ the
  polynomial in $\BC[T_1,\dots, T_{n-1},Y]$ such that $Q(T_1,\dots,
  T_{n-1},Y)=P(T_1,\dots, T_{n-1},X)$, so that
  $ Q = Y^n + a_2 Y^{n-2} + \dots + a_n$,  with $a_2,\dots, a_n \in
  \BC[T_1,\dots, T_{n-1}]$. 

  We define $J(P)$ as above and $J(Q)$ similarly ($Y$ replacing $X$). Then:
  \begin{enumerate}[(i)]
  \item $J(P)=J(Q)$;
  \item $J(P)$ does not depend on $X$, and $J(P)\doteq \Jac((a_2,\dots, a_n) /
    (T_1,\dots, T_{n-1}))$.
  \end{enumerate}
\end{lemma}

\begin{proof}
  \begin{enumerate}[(i)]
  \item Let us denote by $A$ and $B$ the $n\times n$ matrices corresponding
    to $J(P)$ and $J(Q)$ respectively. For $j=1,\dots, n$, the $j$-th
    column of $A$ is 
    \[ A_j := \left( \DP{}{T_j} \left( \DP[i]{P}{X} \right) \right) \]
    with $i$ running from $0$ to $n-1$ (here $T_n:=X$).

    We define similarly $B_j$ (replacing $P$ and $X$ by $Q$ and $Y$). Then,
    we compute that $B_n=A_n$, and for $j=1,\dots, n-1$:
    \[ B_j = A_j -\frac1n \DP{b_1}{T_j} A_n \ . \]
    So $B$ and $A$ have the same determinant and  $J(P)=J(Q)$.
  \item Let us differentiate $J(P)$ with respect to $X$. If $L_i$
    ($i=1,\dots, n$) is the $i$-th line of the matrix $A$, we have
    \[ \DP{L_n}{X}=0 \text{ , and for } i=1,\dots, n-1,\
    \DP{L_i}{X}=L_{i+1}\ . \]
    So $\DP{(J(P))}{X}=0$.

    Similarly, $J(Q)$ does not depend on $Y$. As a consequence,
    $J(P)(T_1,\dots, T_{n-1},X)=J(Q)(T_1,\dots, T_{n-1},0)$, which is by
    definition clearly equal to
    \[ \pm \ n! \left(\prod_{k=0}^{n-2} k!\right)  \Jac((a_2,\dots, a_n) /
    (T_1,\dots, T_{n-1})) \ .\]
  \end{enumerate}
\end{proof}

Consequently, we have an intrinsic definition for the Lyashko-Looijenga
Jacobian: 
\[ J_{\LL} \doteq  J(\Delta_W)=\Jac\left( \left( \Delta_W, \DP{\Delta_W}{f_n},
    \dots, \DP[n-1]{\Delta_W}{f_n} \right) \middle/ (f_1,\dots, f_{n}) \right) \
. \]
where $f_1,\dots, f_n$ do not need to be chosen such that the coefficient
of $f_n^{n-1}$ in $\Delta_W$ is zero.

\medskip

Note that for the computation of $D_{\LL}$ also, the fact that the
coefficient $a_1$ is zero in $\Delta_W$ is not important, because of
invariance by translation.

\medskip

\begin{remark}
  With these alternative definitions, the factorisation of the Jacobian
  given by Thm.\ \ref{thmLL} has already been observed (for real groups) by
  Kyoji Saito: it is Formula 2.2.3 in \cite{saitopolyhedra}. He uses this
  formula in his study of the semi-algebraic geometry of the quotient $W
  \qg V$. 

  His proof was case-by-case and detailed in an unpublished extended
  version of the paper (\cite[Lemma 3.5]{saitopolyhedra2}).

\end{remark}

\section{The Lyashko-Looijenga extension as a virtual reflection group}
\label{partLLvirtual}
%\newpage

In Table \ref{analogies} we list the first analogies between the setting of
Galois extensions (polynomial extension with a reflection group acting) and
that of the Lyashko-Looijenga extensions, which are an example of ``virtual
reflection groups'', in the sense of Bessis. This is not an exhaustive
list, and we may wonder if the analogies can be made further.

%\newcolumntype{M}[1]{>{\raggedright}m{#1}}
\begin{table}[!h]
%\noindent
{\renewcommand{\arraystretch}{1.2}
\begin{tabular}{@{\vrule width 1pt\,}m{2.4cm}@{\,\vrule width 1pt\,}m{6.2cm}@{\,\vrule width 1pt\,}m{6.7cm}@{\,\vrule width 1pt}|}
  \hlinewd{1pt}

  &\begin{center}Complex reflection group\end{center} & \begin{center}Lyashko-Looijenga extension\end{center}
  \tabularnewline \hlinewd{1pt}
  Morphism: & 
%   \hspace{-0.15em}
$\begin{array}{rcl}
    p :\quad  V & \to & W \qg V \\
    (v_1,\dots, v_n)  & \mapsto & (f_1(v),\dots, f_n(v))
  \end{array}$
  &
%  \hspace{-0.1em}
$\begin{array}{rcl}
    \LL : \quad Y & \to &\BC^{n-1} \\
     (y_1,\dots,y_{n-1}) & \mapsto & (a_2(y),\dots, a_n(y))
   \end{array}$
   % $p: V \to W \qg V$ & $\LL : Y \to \BC^{n-1}$\tabularnewline
   % & $(v_1,\dots, v_n) \mapsto (f_1(v),\dots, f_n(v))$ & $(y_1,\dots,
   % y_{n-1})\mapsto (a_2(y),\dots, a_n(y))$
  \tabularnewline\hline
  Extension: & \begin{center}$\BC[f_1,\dots, f_n]=\BC[V]^W \subseteq \BC[V]$\end{center} & \begin{center}
  $\BC[a_2,\dots, a_n] \subseteq \BC[y_1, \dots, y_{n-1}]$ \end{center}\tabularnewline\hline
  Free, of rank: & \begin{center}$|W|=d_1\dots d_n$\ ;\  Galois\end{center}&
  \begin{center}$n!h^n / |W|=\prod ih / \prod d_j$; non-Galois\end{center}\tabularnewline\hline
  Weights: & \begin{center}$\deg v_j=1$, $\deg f_i=d_i$ \end{center}& \begin{center}$\deg y_j=d_j$, $\deg a_i=ih$\end{center}\tabularnewline\hline
  Unramified covering: & \begin{center}$\Vreg \ \surj \ W \qg \Vreg$ \end{center}& \begin{center}$Y-\CK \ \surj \ \Enreg$\end{center}\tabularnewline\hline
  Generic fiber: & \begin{center}$\simeq W$ \end{center}& \begin{center}$\simeq  \Red_\CR(c)$\end{center}\tabularnewline\hline
  Ramified part: & \begin{center}$\bigcup_{H \in \CA} H \ \surj \ (\bigcup H) / W=\CH$ \end{center}& \begin{center}$\CK=
  \bigcup_{\Lambda \in \Lb_2} \varphi(\Lambda) \ \surj \ E_{\alpha}$\end{center}
  \tabularnewline\hline
  Discriminant: & \begin{center}$\Delta_W= \prod_{H\in \CA} \alpha_H^{e_H}\in \BC[f_1,\dots, f_n]$ \end{center}&\begin{center}
  $D_{\LL}=\prod_{\Lambda \in \Lb_2} D_{\Lambda}^{r_{\Lambda}} \in
  \BC[a_2,\dots, a_n]$\end{center}\tabularnewline\hline
  Ramification indices: & \begin{center}$e_H=|W_H|$\end{center} & \begin{center}$r_{\Lambda}=$ order of parabolic elements of type
  $\Lambda$ \end{center}\tabularnewline\hline
  Jacobian: & \begin{center}$J_W= \prod \alpha_H^{e_H-1}\in
    \BC[V]$\end{center}& \begin{center}$J_{\LL}=\prod
    D_{\Lambda}^{r_{\Lambda}-1} \in \BC[f_1,\dots,
    f_{n-1}]$\end{center}\tabularnewline   \hlinewd{1pt}

 % & $\deg J_W = \sum (e_H -1)= \sum(d_i -1)=|\CR|$ & $\deg J_{\LL}= \sum (r_\Lambda -1)
%  \deg D_{\Lambda} = \sum_{i=1}^{n-1} (i+1)h -d_i$ \tabularnewline
\end{tabular}}
\caption{Analogies between Galois extensions and Lyashko-Looijenga extensions.}
\label{analogies}
\end{table}

\section{Combinatorics of the submaximal factorisations}
\label{partfact}

In this section we are going to use Thm.\ \ref{thmLL} to count specific
factorisations of a Coxeter element; this will lead to a geometric proof
of a particular instantiation of Chapoton's formula.

\medskip

We call \emph{submaximal factorisation} of a Coxeter element $c$ a
primitive block factorisation of $c$ with partition
$\alpha=2^11^{n-2}\vdash n$, according to
Def.\ \ref{deffactprim}. Submaximal factorisations are thus exactly
factorisations of $c$ in $n-1$ blocks ($(n-2)$ reflections and one factor
of length $2$), and as such are a natural first generalisation of the
set of reduced decompositions $\Red_\CR (c)$.

\bigskip

Let $\Lambda$ be a stratum of $\Lb_2$: it corresponds to a conjugacy class
of parabolic Coxeter elements of length $2$. We say that a submaximal
factorisation is \emph{of type $\Lambda$} if its factor of length $2$ lies
in this conjugacy class. We denote by $\Fact_{n-1}^{\Lambda}(c)$ the set of
such factorisations. Using the relations between $\LL$ and $\fact$, we can
count these factorisations.

\medskip

For $\Lambda$ a stratum of $\Lb_2$, let us define the following restriction
of $\LL$:
\[ \LL_{\Lambda} \ : \ \varphi(\Lambda) \to E_\alpha \ ,\] where
$E_\alpha= E_n - \Enreg$. We recall that $E_\alpha ^0$ is the subset of
$E_\alpha$ constituted by the configurations whose partition is exactly
$\alpha=2^1 1^{n-2}$.

If we define $\varphi(\Lambda)^0=\LL_{\Lambda}^{-1}(E_\alpha ^0)$, and
$\CK^0=\LL^{-1}(E_\alpha ^0)=\cup_{\Lambda \in \Lb_2} \varphi(\Lambda)^0$,
then from Chapter~\ref{chaphur} we have the following properties:
\begin{itemize}
\item the restriction of $\LL$ : $\CK^0 \ \surj\ E_\alpha ^0$ is a
  (possibly not connected) unramified covering (Thm.~\ref{thmrevetement});
\item the connected components of $\CK^0$ are the $\varphi(\Lambda)^0$, for
  $\Lambda \in \Lb_2$;
\item the image, by the map $\fact$, of $\varphi(\Lambda)^0$ is exactly
  $\Fact_{n-1}^{\Lambda}(c)$;
\item via $\fact$, the Hurwitz action on $\Fact_{n-1}(c)$ corresponds to
  the monodromy action on $\CK^0$; so the orbits are the
  $\Fact_{n-1}^{\Lambda}(c)$ (Thm.~\ref{thmfconj2}).
\end{itemize}

\medskip

The map $\LL_\Lambda$ defined above is an algebraic morphism, corresponding
to the extension
\[ \BC[a_2,\dots, a_n] / (D) \subseteq \BC[f_1,\dots, f_{n-1}]/(D_\Lambda)
\ .\]

\begin{theo}
  \label{thmfact}
  Let $\Lambda$ be a strata of $\Lb_2$. Then:
  \begin{enumerate}[(a)]
  \item $\LL_\Lambda$ is a finite quasi-homogeneous morphism of degree
    $\frac{(n-2)!\ h^{n-1}}{|W|} \deg D_\Lambda $;
  \item the number of submaximal factorisations of $c$ of type $\Lambda$ is
    equal to
    \[ |\Fact_{n-1}^{\Lambda}(c)| = \frac{(n-1)!\ h^{n-1}}{|W|} \deg
    D_\Lambda \ .\]
  \end{enumerate}
  \end{theo}

\begin{proof}
  From Hilbert series, we get that $\LL_\Lambda$ is a finite free extension
  of degree
   \[ \left. \frac{\prod \deg (a_i)}{\deg(D)}\  \middle/ \ \frac{\prod
    \deg(f_i)}{\deg(D_\Lambda)} \right. = \frac{n!\ h^n}{|W|}\frac{\deg
    D_\Lambda}{\deg D} .\]
  \emph{(a)} As $D$ is a discriminant of type $A$ for the variables $a_2,\dots, a_n$ of
    weights $2h,\dots, nh$, we have $\deg D= n(n-1)h$. Thus:
    \[ \deg(\LL_\Lambda) = \frac{(n-2)!\ h^{n-1}}{|W|} \deg D_\Lambda
    .\] 
  \emph{(b)} This degree is also the cardinality of a generic fiber of
    $\LL_\Lambda$, \ie $|\LL^{-1}(\eps) \ \cap \ \varphi(\Lambda)|$, for
    $\eps \in E_\alpha^0$. Consequently, from property (P3) in Sect.
    \ref{subpartstrat}, it counts the number of submaximal factorisations
    of type $\Lambda$, where the length $2$ element has a place
    \emph{fixed} (given by the composition of $n$ associated to
    $\eps$). There are $(n-1)$ compositions of partition type
    $\alpha\vdash n$, so we obtain
    $|\Fact_{n-1}^{\Lambda}(c)|=(n-1)\deg (\LL_\Lambda) = \frac{(n-1)!\
      h^{n-1}}{|W|} \deg D_\Lambda$.
  \end{proof}

\begin{remark}
\label{rkconcat}
Let us denote by $\Fact_{(2,1,\dots, 1)}^{\Lambda}(c)$ the set of submaximal
factorisations of type $\Lambda$ where the length $2$ factor is in first
position. By symmetry, formula (b) is equivalent to
\[ |\Fact_{(2,1,\dots, 1)}^{\Lambda}(c)|=\frac{(n-2)!\ h^{n-1}}{|W|} \deg
D_\Lambda \ .\] As $\sum r_{\Lambda} \deg D_{\Lambda}= \deg D =n(n-1)h$,
this implies the equality :
\[ \sum_{\Lambda \in \Lb_2} r_{\Lambda} |\Fact_{(2,1,\dots,
  1)}^{\Lambda}(c)|=  \frac{(n-2)!\ h^{n-1}}{|W|} \deg D =
\frac{n!h^n}{|W|}=|\Red_\CR(c)| \ .\]
This formula is actually not surprising, since for the concatenation map:
\[ \begin{array}{ccc}
\Red_\CR (c) & \surj & \Fact_{(2,1,\dots, 1)}(c)\\
(s_1,s_2,\dots, s_n) & \mapsto & (s_1 s_2, s_3,\dots, s_n) \ ,
\end{array}
\]
the fiber of a factorisation of type $\Lambda$ has indeed cardinality
$r_{\Lambda}$ (which is the number of factorisations of the first factor in
two reflections).
\end{remark}

\begin{remark}
\label{rkkra}
  In \cite{KraMu}, motivated by the enumerative theory of the generalised
  non-crossing partitions, Krattenthaler and Müller defined and computed
  the \emph{decomposition numbers} of a Coxeter element, for all
  irreducible \emph{real} reflection groups. In our terminology, these are
  the numbers of block factorisations according to the Coxeter type of the
  factors. Note that the Coxeter type of a parabolic Coxeter element is the
  type of its associated parabolic subgroup, in the sense of the
  classification of finite Coxeter groups. So the conjugacy class for a
  parabolic elements is a finer characteristic than the Coxeter type: take
  for example $D_4$, where there are three conjugacy classes of parabolic
  elements of type $A_1\times A_1$.

  Nevertheless, when $W$ is real, most of the results obtained from formula
  (b) in Thm.\ \ref{thmfact} are very specific cases of the computations in
  \cite{KraMu}. But the method of proof is completely different, geometric
  instead of combinatorial\footnote{The computation of all decomposition
    numbers for complex groups, by combinatorial means, is also a work in
    progress (Krattenthaler, personal communication).}. Note that another possible
  way to tackle this problem is to use a recursion, to obtain data for the
  group from the data for its parabolic subgroups. A recursion formula (for
  factorisations where the rank of each factor is dictated) is indeed given
  by Reading in \cite{reading}, but the proof is very specific to the real
  case.
\end{remark}

For $W$ non-real, formula (b) implies new combinatorial results on the
factorisation of a Coxeter element. The numerical data for all irreducible
well-generated complex reflection groups are listed in Appendix
\ref{annexfactodisc}. In particular, we obtain (geometrically) general
formulas for the submaximal factorisations of a given type in $G(e,e,n)$.

\section{Chapoton's formula and submaximal factorisations}

If $W$ is a well-generated complex reflection group, we recall that the
noncrossing partition lattice of type $W$ is the poset of divisors of a
fixed Coxeter element $c$:
\[ \NCP_W(c):= \{ w \in W \tq w\< c \} \ ,\]
where $\<$ is the absolute order of $W$ (see Def.\ \ref{defordre}). 

\subsection{Chapoton's formula for the number of multichains}

Chapoton's formula gives the number of multichains of a given
length in the poset $(\NCP_W, \<)$:

\begin{theo}[``Chapoton's formula'']
\label{thmchapo}
  Let $W$ be an irreducible well-generated complex reflection group, of
  rank $n$. Then, for any $N\in \BN$, the number of multichains $w_1 \<
  \dots \< w_{N} \< c$ in the poset $\NCP_W$ is equal to :
  \[ \Cat^{(N)}(W) = \prod_{i=1}^n \frac{d_i + Nh}{d_i} ,\]
 where $d_1\leq\dots \leq d_n=h$ are the invariant degrees of $W$.
\end{theo}

The numbers $\Cat^{(N)}(W)$ are called \emph{Fuss-Catalan numbers of type
  $W$}, and count also other combinatorial objects, for example in cluster
algebras, as we explained in the introduction.

In the real case, this formula was first observed by Chapoton in
\cite[Prop.\ 9]{chapoton}; the proof is case-by-case, and mainly uses
results by Athanasiadis and Reiner \cite{reiner,athareiner}. The remaining
complex cases are checked by Bessis in \cite{BessisKPi1}, using his study of the
infinite series $G(e,e,r)$ with Corran \cite{BessisCorran}. \emph{There is still
no case-free proof of this formula.}

\begin{coro}
\label{corcat}
  Let $W$ be an irreducible well-generated complex reflection group, with
  invariant degrees  $d_1\leq\dots \leq d_n=h$. Then :
  \begin{enumerate}[(i)]
  \item the cardinality of $\NCP_W$ is $\quad \displaystyle{\prod_{i=1}^n \frac{d_i
        + h}{d_i}}$ \emph{(Catalan number of type $W$)};
  \item the number of reduced decompositions of a Coxeter element is
    $\displaystyle{\frac{n!h^n}{|W|}}$ \emph{(Lyashko-Looijenga number of type
    $W$)}.
  \end{enumerate}
\end{coro}

Both formulas are of course consequences of Thm.\ \ref{thmchapo} : (i) is
just the case $N=1$, and (ii) comes from the computation of maximal strict
chains in a poset. But historically, they have been observed before
Chapoton's formula. In the Coxeter types, (ii) has been conjectured by
Looijenga in \cite{looijenga}, and proved later by Deligne
\cite{deligneletter}. Still today, the only known proofs are case-by-case.

\subsection{Number of submaximal factorisations of a Coxeter element}

Using the previous results of this chapter, we can state here a formula for
the number of submaximal factorisations, with a ``geometric'' proof:

\begin{theo}
\label{thmsubmax}
  Let $W$ be an irreducible well-generated complex reflection group, with
  invariant degrees $d_1\leq\dots \leq d_n=h$. Then, the number of
  submaximal factorisations of a Coxeter element $c$ is equal to:
    \[ |\Fact_{n-1}(c)| = \frac{(n-1)!\ h^{n-1}}{|W|}
\left( \frac{(n-1)(n-2)}{2}h + \sum_{i=1}^{n-1} d_i \right) .\] 
\end{theo}

\begin{proof}
  Using Thm.\ \ref{thmfact}(b) and Thm.\ \ref{thmLL}(b)-(c), we compute:
  \[\begin{array}{lcl}
    |\Fact_{n-1}(c)| =|\Fact_{\alpha}(c)| & = &
    \displaystyle{\sum_{\Lambda \in \Lb_2}
      |\Fact_{n-1}^{\Lambda}(c)|}\\ & = &
    \displaystyle{\frac{(n-1)!\ h^{n-1}}{|W|} \sum_{\Lambda \in \Lb_2}
      \deg D_\Lambda }\\ & = &
    \displaystyle{\frac{(n-1)!\ h^{n-1}}{|W|} \left( \deg D_{\LL} - \deg
        J_{\LL} \right) }\ ,
  \end{array} \]
  As $\deg D_{\LL} = n(n-1)h$ and $\deg
  J_{\LL} = \sum_{i=2}^n \deg(a_i) - \sum_{j=1}^{n-1}
  \deg(f_j)=\sum_{i=2}^n ih - \sum_{j=1}^{n-1} d_j$, a quick
  computation gives the conclusion.
\end{proof}

\begin{remark}
  The formula in the above theorem is actually included in Chapoton's
  formula: indeed, there exist easy combinatorial tricks allowing to pass
  from the numbers of multichains to the numbers of strict chains (which
  are roughly the numbers of block factorisations). In the Appendix
  \ref{annexchapo}, we detail these relations, and give the formulas for
  the number of block factorisations predicted by Chapoton's formula.

  \emph{However}, the proof we obtained here is more
  satisfactory (and more enlightening) than the one using Chapoton's
  formula. Indeed, if we recapitulate the ingredients of the proof, we only
  made use of the formula for the Lyashko-Looijenga number
  (Cor.\ \ref{corcat}(b)) --- necessary to prove the first properties of $\LL$
  in \cite{BessisKPi1} ---, the remaining being the geometric properties of
  $\LL$, for which we never used the classification. In other words, we
  travelled from the numerology of $\Red_\CR(c)$ to that of
  $\Fact_{n-1}(c)$, without adding any case-by-case analysis to the setting
  of \cite{BessisKPi1}.
\end{remark}

 \bigskip

Although it seems to be a new interesting avenue towards a geometric
explanation of Chapoton's formula, the method used here to
compute the number of submaximal factorisations is not directly
generalisable to factorisations with fewer blocks. A more promising
approach would be to avoid computing explicitely these factorisations, and
to try to understand globally Chapoton's formula as some ramification
formula for the morphism $\LL$ (see Remark \ref{rkramifLL}).

% Fin chapitre.

\appendix

% Données numériques
%%%%%%%%%%%%%%%%%%%%%%%%%%%%%%%%%%%%%%%%%%%%%%%%%%%%%%%%%%%%%%%%%%%%%%%
% Annexe de thèse sur les données numériques des factorisations de
% D_LL.
% A compiler avec les en-têtes de manuscrit.tex
%
%%%%%%%%%%%%%%%%%%%%%%%%%%%%%%%%%%%%%%%%%%%%%%%%%%%%%%%%%%%%%%%%%%%%%%%

\chapter[Numerical data for the factorisations]{Numerical data
  for the factorisations of the Lyashko-Looijenga discriminants}

\label{annexfactodisc}
\section{Description}

This appendix details explicit numerical data regarding the factorisation of
the discriminant polynomial $D_{\LL}$, where $\LL$ is the Lyashko-Looijenga
morphism associated to a well-generated irreducible reflection group.

\medskip

For such a group $W$, let us recall the definition (from
Sect.\ \ref{subpartLLwell}):
\[ D_{\LL}= \Disc(\Delta_W(f_1,\dots, f_n) ; f_n) \  ,\]
where $\Delta_W$ is the \emph{discriminant} of $W$ (\ie the equation of the
union of the reflecting hyperplanes in the quotient $W \qg V$), and
$f_1,\dots, f_n$ are the fundamental invariant polynomials of $W$ ($f_n$
being the one with maximal degree).

Let us write (as in \ref{subpartLLwell})
\[ D_{\LL} = \prod_{i=1}^r D_i^{p_i} \] 
the factorisation of $D_{\LL}$ in irreducible polynomials of
$\BC[f_1,\dots, f_n]$ (actually, the factorisation holds in $K_W[f_1,\dots,
f_n]$, where $K_W$ is the field of definition
of $W$).

\medskip

In the table \ref{tab}, we give, for each irreducible well-generated group, the weighted degrees $\deg(D_i)$ and the powers $p_i$ which
appear in the factorisation above. It is enough to deal with the
$2$-reflection groups (those generated by reflections of order $2$),
because any irreducible complex reflection group is isodiscriminantal to a
$2$-reflection group (see \cite[Thm.2.2]{BessisKPi1}): it has the same
discriminant $\Delta$, and consequently the same braid group and the
same $D_{\LL}$. Thus we only have to treat the four infinite series $A_n$,
$B_n$, $I_2(e)$, $G(e,e,n)$, and $11$ exceptional types (including the $6$
exceptional Coxeter groups).

\newpage

\subsection*{Notations} 

In the last column of table \ref{tab}, the ``$\LL$-data'':
\[ \boxed{p_1} \centerdot \left( u_1\right) + \boxed{p_2} \centerdot \left(
  u_2\right) + \dots + \boxed{p_r} \centerdot \left( u_r\right) \] means
that the form of the factorisation is $ D_{\LL} = \prod_{i=1}^r D_i^{p_i}$
with $\deg D_i = u_i$. Thus this writing reflects the additive
decomposition of $\deg D_{\LL} = n(n-1)h$ (where $n=\rk(W)$ and $h=d_n$) in terms
of the $u_i$:
\[ \deg D_{\LL} = \sum_i p_i u_i \ . \]

\subsection*{By-products}

These numbers $(p_i,u_i)$ have many combinatorial interpretations. In
particular, thanks to Theorems \ref{thmLL} and \ref{thmfact}, we have:
\begin{itemize}
\item the number of conjugacy classes of parabolic Coxeter elements of
  length $2$ is the number of terms in the sum (each term
  $(p_i,u_i)$ of the sum
  corresponds to one of these classes, say $\Lambda_i$);
\item the order of the elements in $\Lambda_i$ is $p_i$ (provided $W$ is a
  $2$-reflection group);
\item the number $|\Fact_{(2,1,\dots, 1)}^{\Lambda_i}(c)|$ of submaximal
  factorisations of a Coxeter element $c$, whose first factor is in the
  class $\Lambda_i$, equals $\frac{(n-2)!\ h^{n-1}}{|W|} u_i $.
\end{itemize}
As a consequence, the $\LL$-data of the last column, when multiplied by the
scalar $\frac{(n-2)!\ h^{n-1}}{|W|}$ (which is listed in the second
column), gives rise to the following equality:
\[ 
\begin{disarray}{lcl}
\sum_i \ p_i \ |\Fact_{(2,1,\dots, 1)}^{\Lambda_i}(c)| & = & \frac{(n-2)!\ h^{n-1}}{|W|}
\deg D \\
 & = &  \frac{n!\ h^{n}}{|W|} \\
 & = & |\Red_\CR(c)| \ ,
\end{disarray}
\]
which simply reflects the enumeration of fibers of the concatenation map
(see Remark \ref{rkconcat}).
\[ \begin{array}{ccc}
\Red_\CR (c) & \surj & \Fact_{(2,1,\dots, 1)}(c)\\
(r_1,r_2,\dots, r_n) & \mapsto & (r_1 r_2, r_3,\dots, r_n) \ .
\end{array}
\]

\begin{table}[!h]
\rotatebox{90}{
{\normalsize
$   %$
{\renewcommand{\arraystretch}{1.2}
\begin{disarray}{@{\vrule width 1pt\,}c@{\,\vrule width 1pt\,}c@{\,\vrule width 1pt\,}c@{\,\vrule width 1pt}}
  \hlinewd{1pt}
\text{\begin{tabular}{c} Group type \\ {[Isodiscriminantal groups]}
    \end{tabular}} &
  \slfrac{(n-2)!\ h^{n-1}}{|W|} &
  \LL\text{-data} 
%\text{ in irreducibles}
  \\
  \hlinewd{1pt} 
  \begin{array}{c} A_n \ ,\ n \geq 2. \\
    \left[G_4,G_8,G_{16},G_{25},G_{32}\right] \end{array} &
  \slfrac{(n+1)^{n-2}}{(n(n-1))} & \boxed{2} \centerdot \left(\slfrac{n(n-1)(n-2)}{2}\right) +
  \boxed{3} \centerdot \left(n(n-1)\right) \\
  \hline
  \begin{array}{c} B_n \ , \ {n\geq 2}. \\ \left[G(d,1,n),G_5,G_{10},G_{18},G_{26}\right]
    \end{array} &
  \slfrac{n^{n-2}}{(2(n-1))} & 
  \begin{array}{r}
    \boxed{2} \centerdot \left((n-1)(n-2)(n-3)\right) + \boxed{2}
    \centerdot \left(2(n-1)(n-2)\right)  \qquad \qquad \\
    +\: \boxed{3} \centerdot
    \left(2(n-1)(n-2)\right)+ \boxed{4}
    \centerdot \left(2(n-1)\right)
  \end{array}
  \\
  \hline
  \begin{array}{c} 
    I_2(e)\\
    \left[G_6,G_9,G_{17},G_{14},G_{20},G_{21}\right] 
  \end{array}  &
  \slfrac12 & 
  \boxed{e}\centerdot \left(2\right) 
  \\
  \hline
  \begin{array}{c}
    G(e,e,n), \ e \geq 2, n\geq 5 \\
    (=D_n \text{ for } e=2)
  \end{array} &
  \slfrac{(n-1)^{n-2}}n & 
  \boxed{2} \centerdot \left(\slfrac{n(n-2)(n-3)e}{2}\right) +
  \boxed{3} \centerdot \left(n(n-2)e \right) + \boxed{e} \centerdot \left(n\right) 
  \\
  \hline
  G(e,e,3)\ ,\ e\geq 3&
  \slfrac23 & 
  \begin{array}{lcccc}
    \text{If }3\nmid e\ &: &  \boxed{3} \centerdot \left(3e\right) &+& \boxed{e}\centerdot \left(3\right)\\
    \text{If }3\mid e\ &: &  \boxed{3} \centerdot \left(e\right)+ \boxed{3} \centerdot \left(e\right)+
    \boxed{3} \centerdot \left(e\right) &+&\boxed{e}\centerdot \left(3\right) 
  \end{array}
  \\
  \hline
  \begin{array}{c}
    G(e,e,4), \ e \geq 2 \\
    (=D_4 \text{ for } e=2)
  \end{array} &
  \slfrac{9}4 & 
  \begin{array}{lcccc}
    \text{If }e\text{ odd } &: & \boxed{2}\centerdot \left(4e\right) &+& \boxed{3}
    \centerdot \left(8e\right) + \boxed{e} \centerdot \left(4\right)   \\
    \text{If }e\text{ even }& : & \boxed{2}\centerdot \left(2e\right) + \boxed{2}\centerdot \left(2e\right) &+ &\boxed{3}
    \centerdot \left(8e\right) + \boxed{e} \centerdot \left(4\right) 
  \end{array}
  \\
  \hline G_{23} \ (=H_3)& \slfrac56 & \boxed{2} \centerdot \left(6\right) +
  \boxed{3} \centerdot \left( 6 \right) + \boxed{5} \centerdot \left( 6
  \right)
  \\
  \hline G_{24} & \slfrac7{12} & \boxed{3} \centerdot \left(12 \right) +
  \boxed{4} \centerdot \left(12 \right)
  \\
  \hline G_{27} & \slfrac5{12} & \boxed{3} \centerdot \left(12 \right) +
  \boxed{3} \centerdot \left(12 \right) + \boxed{4} \centerdot \left(12
  \right) + \boxed{5} \centerdot \left(12 \right)
  \\
  \hline G_{28} \ (=F_4) & 3 &  \boxed{2} \centerdot \left( 24\right) +
  \boxed{3} \centerdot \left( 8 \right) + \boxed{3} \centerdot \left( 8
  \right) +  \boxed{4} \centerdot \left( 12 \right)
  \\
  \hline G_{29} & \slfrac{25}{12} & \boxed{2} \centerdot \left(24 \right) +
  \boxed{3} \centerdot \left(48 \right) + \boxed{4} \centerdot \left(12
  \right)
  \\
  \hline G_{30} \ (=H_4) & \slfrac{15}4 & \boxed{2} \centerdot
  \left(60\right) + \boxed{3} \centerdot \left( 40 \right) + \boxed{5}
  \centerdot \left( 24 \right)
  \\
  \hline G_{33} & \slfrac{243}{20} & \boxed{2} \centerdot \left(60 \right)
  + \boxed{3} \centerdot \left(80 \right)
  \\
  \hline G_{34} & \slfrac{2401}{30} & \boxed{2} \centerdot \left(270
  \right) + \boxed{3} \centerdot \left(240 \right)
  \\
  \hline G_{35} \ (=E_6) & \slfrac{576}{5} & \boxed{2} \centerdot \left(90
  \right) + \boxed{3} \centerdot \left(60 \right)
  \\
  \hline G_{36} \ (=E_7) & \slfrac{19683}{14} & \boxed{2} \centerdot
  \left(210 \right) + \boxed{3} \centerdot \left(112 \right)
  \\
  \hline G_{37} \ (=E_8) & \slfrac{1265625}{56} & \boxed{2} \centerdot
  \left(504 \right) + \boxed{3} \centerdot \left(224 \right) \\
  \hlinewd{1pt}
\end{disarray}}
$  %$
}
}
\caption{Factorisation of the $\LL$-discriminant for irreducible
  well-generated groups}\label{tab}
\end{table}
\clearpage
\section{Computations}

Let us explain how Table \ref{tab} was obtained.  As mentioned in Remark
\ref{rkkra}, the numerical results about the $|\Fact_{(2,1,\dots,
  1)}^{\Lambda_i}(c)|$, for $W$ real, are particular cases of the
``decomposition numbers'' computed by Krattenthaler-Müller in \cite{KraMu},
namely the $N_W(T_1,A_1,\dots, A_1)$, where $T_1$ is a rank $2$ group type.

For some types we are able to directly recover these numbers just by
explicitely factorising $D_{\LL}$. When this factorisation is not easily
computable, we rather rely on \cite{KraMu}, thus obtaining the geometric
properties of the discriminant through combinatorial means. 

For non-real groups, we cannot rely on \cite{KraMu}, and the corresponding
entries of the table are new results.

\begin{remark}
  For Coxeter groups, the degrees of certain factors of $D_{\LL}$ have also
  been computed by Saito in \cite{saitopolyhedra2}, in order to check
  case-by-case the formula of Thm.\ \ref{thmLL}(c).
\end{remark}

In the following, we detail the methods of computation that we used.

\subsection{Straightforward cases}

\label{subpartstfwd}
Let $W$ be a well-generated irreducible $2$-reflection group. Suppose that
there are only two conjugacy classes of parabolic Coxeter elements of
length $2$, and that we know the orders of their elements $p_1$,
$p_2$. Then, from Thm.\ \ref{thmLL}, the discriminant $D_{\LL}$ and the
Jacobian $J_{\LL}$ have the form:
\[ D_{\LL} = D_1^{p_1} D_2^{p_2} \qquad ; \qquad
J_{\LL}=D_1^{p_1-1}D_2^{p_2-1} \ .\]
In order to find the degrees $u_1=\deg D_1$ and  $u_2=\deg D_2$, we then
just have to solve the following system:
\[\left\{
\begin{disarray}{rcrclcl}
p_1 \ u_1 & + &  p_2 \ u_2& =& \deg D_{\LL}& =& n(n-1)h \\
(p_1-1)\  u_1& +& (p_2-1)\  u_2 &= &\deg J_{\LL}& =& \sum_{i=2}^n (ih) -
\sum_{j=1}^{n-1} d_j \ .
\end{disarray}
\right.
\]

This allows to easily tackle the cases\footnote{The conjugacy classes of
  length $2$ parabolic elements can be found using the package {\sf CHEVIE}
  of {\sf GAP3}.} of $A_n$, $G_{24}$, and the large exceptional types
$G_{33}$ through $G_{37}$.

\subsection{Other exceptional types}

We can explicitly compute the discriminant $\Delta_W$, via the matrix of
basic derivations, following the strategy described in the Appendix B of
\cite{orlikterao}. We use the software {\sf GAP3} \cite{GAP} together with
its packages {\sf CHEVIE} \cite{chevie} and {\sf VKCURVE} \cite{vkcurve}
\footnote{Most of the implementation of \cite{orlikterao} is already
  available in {\sf VKCURVE}.}. Then we compute the discriminant $D_{\LL}$
and its factorisation in irreducibles, which is an easy task in Maple, or
any symbolic computation software (here the rank is never greater than
$5$).

\subsection{Series $B_n$}

For $W=B_n$, we can take for each invariant $f_i$ ($i=1,\dots ,n$) the
$i$-th elementary symmetric polynomial in $x_1^2,\dots, x_n ^2$; then the
discriminant can be written as:
\[ \Delta_W = f_n \Disc(T^n - f_1 T^{n-1} + \dots + (-1)^n f_n ; T) \ .\]

It should be possible to study in general the factorisation in
$\BQ[f_1,\dots,f_n]$ of the polynomial $\Disc (\Delta_W ; f_n)$; however
it is much simpler here to use the numerical results in \cite{KraMu} and to
apply the formula of Thm.\ \ref{thmfact}.

\medskip

Note that \cite{KraMu} gives the number of factorisations according to the
Coxeter type of the parabolic subgroup associated to the factor; here we
rather need these numbers according to the conjugacy classes of the factors. In
this case, this comes back to use the data for the ``combinatorial types''
of $B_n$, see \cite[Sect.\ 2]{KraMu}.

Let us denote by $t,s_2, \dots, s_n$ a set of Coxeter generators for $B_n$
(with $m_{t,s_2}=4$, $m_{s_i,s_{i+1}}=3$, the other ones being $2$). If $n
\geq 4$, there are in $B_n$ four conjugacy classes of parabolic Coxeter
elements of length $2$:
\begin{itemize}
\item the class (denoted $\Lambda_4$) containing $ts_2$ (elements of order $4$);
\item the class ($\Lambda_3$) of the $s_i s_{i+1}$'s (order $3$);
\item the class ($\Lambda_2$) of the $s_i s_j$'s, for $|i-j|>1$ (order $2$);
\item the class ($\Lambda_2'$) of the $ts_i$'s, for $i>2$ (order $2$).
\end{itemize}

Then, with our notations and those of \cite[Sect.\ 5]{KraMu}, we have:
\begin{eqnarray*}
 |\Fact_{(2,1,\dots, 1)}^{\Lambda_4}(c)|    &=&  
 N_{B_n}^{\text{comb}}(B_2,A_1,\dots , A_1) \\
|\Fact_{(2,1,\dots, 1)}^{\Lambda_3}(c)|    &=&   (n-2) \times 
 N_{B_n}^{\text{comb}}(A_2,B_1,A_1,\dots , A_1)  \\
|\Fact_{(2,1,\dots, 1)}^{\Lambda_2}(c)|   & = &  (n-2) \times 
 N_{B_n}^{\text{comb}}(A_1^2,B_1,A_1,\dots , A_1)  \\
|\Fact_{(2,1,\dots, 1)}^{\Lambda_2'}(c)|    & = &   
 N_{B_n}^{\text{comb}}(B_1*A_1,A_1,\dots , A_1) \ . 
\end{eqnarray*}

The results are listed in the table \ref{tab}.

\subsection{Series $G(e,e,n)$}

If $W=G(e,e,n)$ (with $n\geq 3$, $e\geq 2$), we can take for the invariants:
\[ f_1=x_1\dots x_n \quad ; \quad f_i= \sigma_{i-1}(x_1^e,\dots,
x_n^e)\qquad (i=2,\dots, n) \ ,\] 
where $\sigma_i$ is the $i$-th elementary symmetric polynomial (note that
here $f_1$ is not necessarily the invariant of lowest degree, but
$f_n$ is still the one of highest degree $h= (n-1)e$). Then the discriminant is
equal to:
\[ \Delta_W = \Disc (T^n - f_2 T^{n-1}+\dots + (-1)^{n-1} f_n T + (-1)^n
f_1^e ; T) \ . \]

Let us fix $n$, and denote by $\Delta^{[e]}$, $D_{\LL}^{[e]}$ and
$J_{\LL}^{[e]}$ the discriminant, $\LL$-discriminant and $\LL$-Jacobian for
$G(e,e,n)$. They depend only slightly on $e$, namely:
\[ \Delta^{[e]}(f_1,\dots, f_{n})  = P(f_1^e,f_2,\dots, f_n) \]
for a polynomial $P$ independent of $e$, and
 \begin{gather}  D_{\LL}^{[e]} (f_1,\dots, f_{n-1}) =Q(f_1^e,f_2,\dots, f_{n-1} )
  \tag{1} \label{eqnD} \end{gather}
where  $Q(X_1,\dots, X_{n-1})= \Disc( P(X_1,\dots, X_n) ; X_n)$. Moreover,
an easy computation shows:
\begin{gather} J_{\LL}^{[e]}(f_1,\dots, f_{n-1})=e f_1^{e-1} J_1(f_1^e,f_2,\dots, f_n)
  \tag{2} \label{eqnJ} \end{gather}
where $J_1(X_1,\dots, X_n)= \Jac ((P,\DP{P}{X_n},\dots, \DP[n-1]{P}{X_n}) ;
(X_1,\dots ,X_n))$ (see Sect.\ \ref{subpartintrinsic}).

\bigskip

We will need the following lemma:
\begin{lemma}
  \label{lemorder}
  If $n\geq 5$ and $e\geq 4$, then there are in $G(e,e,n)$
  three conjugacy classes of parabolic Coxeter elements of length $2$. The
  orders of the elements in these classes are $2$, $3$, $e$, respectively.
\end{lemma}

\begin{proof}
  We use the representation of elements of $\NCP_W$ as ``noncrossing
  partitions of type $(e,e,n)$'', described by Bessis-Corran. We explain
  roughly the method, strongly relying on their paper \cite{BessisCorran}.

  The lattice $\NCP_W$ is isomorphic to the set $NCP(e,e,n)$ of the
  non-crossing partitions $u$ of $\mu_{e(n-1)} \cup \{0\}$ ---where $\mu_p$
  is the $p$-gon of $p$-th roots of unity in $\BC$---, such that,
  forgetting $\{0\}$, $u$ is $e$-symmetric (\ie fixed by the natural action
  of the group $\mu_e$ on partitions of $\mu_{e(n-1)}$).

  Elements $NCP(e,e,n)$ fall into three types, according to the geometric
  shape of the partition: short symmetric, long symmetric, or asymmetric
  (see Def.\ 1.15\footnote{All the references in this proof are to
    \cite{BessisCorran}.}). Moreover, we have a notion of height for an
  element of $NCP(e,e,n)$ (see Sect.\ 1.8), which is equal to the length of
  the corresponding element of $W$ (Lemma 4.1). There is a natural
  geometric way to see elements of $NCP(e,e,n)$ as braids of $B(W)$ or
  elements of $W$ (Sect.\ 2.3). With these ingredients, we can describe the
  elements of length $2$ of $\NCP_W$ according to their types:
  \begin{enumerate}[(i)]
  \item \emph{Short symmetric.} Let us fix a maximal short symmetric
    element $v$. Its partition contains $e$ blocks of size $n-1$, each of
    its divisors is also short symmetric, and the lattice under $v$ is
    isomorphic to the set of classical noncrossing partitions $\NCP(n-1)$
    (or $\NCP_{A_{n-2}}$). Using the type $A$ case, we deduce that there are two
    conjugacy classes of divisors of $v$ of height $2$, corresponding to
    the orders $2$ and $3$. Moreover, any short symmetric element $u$ is
    conjugate to a divisor of $v$: use the conjugation by the Coxeter
    element $c$, which acts as a rotation of the partition.
  \item \emph{Long symmetric.} If $u$ is a long symmetric element of height
    $2$, then its partition has necessarily only one part which is not a
    singleton, and this part must be a regular $e$-gon (together with
    $\{0\}$). All the elements of this type are obviously of order $e$, and
    conjugated to each other by the action of $c$ (rotation).
  \item \emph{Asymmetric.} If $u$ is an asymmetric element, let us denote
    by $u^\flat$ the partition obtained from $u$ by forgetting the point $0$. Then,
    $u$ has height $2$ when in each $e$-sector of $u^\flat$ there are one
    part of size $2$ and $n-3$ singletons.
    \begin{itemize}
    \item If in the partition of $u$, $0$ is with a $2$-part of $u^\flat$, then
      $u$ has order $3$. It can be brought down to a short symmetric
      element of (i) by conjugation with an adapted asymmetric height $1$
      element (note that it requires $n-1 \geq 3$).
    \item If in the partition of $u$, $0$ is with a singleton of $u^\flat$, then
      $u$ has order $2$. It can also be brought down to a short symmetric
      element of (i) by conjugation with an adapted asymmetric height $1$
      element (providing this time that $n-1 \geq 4$).
    \end{itemize}
  \end{enumerate}
As $e \neq 2,3$, we can conclude.
\end{proof}

\begin{remark} When $e=2$ or $3$, the proof only shows that there are two
  or three conjugacy classes, according to whether the long symmetric
  elements are in a specific class or not. Actually, the result is always
  three, as we will see below.
\end{remark}

If we suppose that $n\geq 5$ and $e\geq 4$, the lemma implies that the
$\LL$-discriminant and Jacobian associated to $G(e,e,n)$ have the form
(using Thm.\ \ref{thmLL}):
\[ D_{\LL}^{[e]}  = D_2^2 D_3^3 D_e^e \qquad ; \qquad   J_{\LL}^{[e]}  = D_2
D_3^2 D_e^{e-1} ,\]
where $D_2,D_3,D_e$ are irreducible in $\BC[f_1,\dots f_{n-1}]$. Thanks to
equation \ref{eqnJ}, we know that $f_1^{e-1}$ divides $J_{\LL}^{[e]}$, so $D_e \doteq f_1$.

\medskip

Because of equation \ref{eqnD}, if $e'$ divides $e$, then the irreducible
factorisation of $D_{\LL}^{[e']}$ is a refinement of the one of
$D_{\LL}^{[e]}$. But, as explained in the remark above, even when $e'=2$ or
$3$, there are at most three conjugacy classes of length $2$ parabolic
Coxeter elements in $G(e',e',n)$, so at most three irreducible factors in
$D_{\LL}^{[e']}$. Using for example $e=6$, it proves that $D_{\LL}^{[e']}$,
for $e'=2$ or $3$, has also an irreducible factorisation of the form:
$D_2^2 D_3^3 f_1^e$. 

\medskip

To complete the table, it remains to compute $\deg D_2$ and $\deg D_3$ in
function of $e$, $n$. This is easily done using the same method as in
section \ref{subpartstfwd}, because we have two equations and only two
unknowns left.

\bigskip

When $n \leq 4$, the factorisation can be finer than the general one.
\begin{itemize}
\item If $n=3$, we compute $D_{\LL}^{[e]}=f_1^e (27f_1^e-f_2^3)^3$. The
  second factor is irreducible when $e\notin 3\BZ$; otherwise, it
  decomposes into three factors, each of degree $e$.
\item If $n=4$, as in the general case, the proof of Lemma \ref{lemorder}
  shows that $D_{\LL}^{[e]}=f_1^e D_2^2 D_3^3$, with $D_3$ irreducible but
  $D_2$ possibly reducible. The computation gives the equality
  $D_2\doteq {64f_1^e -(f_2^2-4f_3)^2}$, so $D_2$ is irreducible if and only
  if $e$ is odd.
\end{itemize}

All the numerical data are gathered in the table \ref{tab}.

% Formulas for the factorisations of a Coxeter element
\chapter{Chains, multichains and block factorisations}
\label{annexchapo}

In this appendix we give the formulas relating the numbers of multichains
and of block factorisations, and we deduce the equalities predicted by
Chapoton's formula for factorisations of a Coxeter element.

\section{Multichains and factorisations in an ordered group}

The relation between multichains and factorisations in $\NCP_W$ can easily
be deduced from the relation between strict chains and multichains, which
is classical material (see Chap. 3.11 in \cite{stanley}). For the sake of
completeness, we give here direct proofs of the formulas relating
multichains and block factorisations, in the general setting of a poset
order given by the divisibility in a generated group $(G,A)$.

\medskip

Let $G$ be a finite group, and $A$ a generating set for $G$. We denote by
$\ell_A$ the associated length function, and $\Aleq$ the associated partial
order on $G$: $u\Aleq v$ if and only if $\ell_A(u) + \ell_A(u^{-1} v) = \ell_A
(v)$. We fix an element $c$ in $G$, and denote by $P_c$ the interval
$[1,c]$ for the order $\Aleq$.

\begin{defi}
  Let $p$ be a positive integer.
  \begin{itemize}
  \item A \emph{$p$-multichain} in $P_c$ is a chain $u_1 \Aleq \dots \Aleq
    u_p \Aleq  c$; the number of such chains is denoted by $\Ch_p$.
  \item A \emph{$p$-block-factorisation} (of $c$) is a $p$-tuple
    $(g_1,\dots,g_p)$ of elements in $G \setminus \{1\}$ such that
    $c=g_1\dots g_p$ and $\ell_A(c)=\ell_A(g_1)+\dots + \ell_A(g_p)$; the
    number of such factorisations is denoted by $\Fact_p$.
  \end{itemize}
\end{defi}

For $N\in \BN^*$, any $N$-multichain $u_1\Aleq \dots \Aleq u_N$ defines a
unique $p$-block-factorisation (for a $p\leq N+1$): first transform the
multichain $u_1\Aleq \dots \Aleq u_N \Aleq c$ in a strict chain $1=v_0
\Aleqst v_1\Aleqst \dots \Aleqst v_p=c$ by erasing the redundant elements, then
define $g_i=v_{i-1}^{-1}v_i$ for $i=1,\dots, p$.

\medskip

Let us count how many multichains give rise to a fixed factorisation
$(g_1,\dots, g_p)$ with this operation. If we define $h_i:=g_1\dots g_i$,
these are the multichains of the form
\[ \underbrace{1 \Aleq \dots \Aleq 1}_{r_0 \text{ times}} \Aleq
\underbrace{h_1 \Aleq \dots \Aleq h_1}_{r_1 \text{ times}} \Aleq
\dots \Aleq 
\underbrace{h_p \Aleq \dots \Aleq h_p}_{r_p \text{ times}} \ , \] with
$r_0\geq 0$, $r_p \geq 0$, and $r_i \geq 1$ for $i=1,\dots, p-1$. It
remains to compute the number of $(r_0,\dots, r_p)$ in $\BN^*\times
\BN^{p-1}\times \BN^*$ such that $\sum r_i=N$. This corresponds to choose $p$
distinct elements in $\{0,\dots, N\}$ (then take for the $r_i$ the lengths of
the intervals which are cut out). So:
\begin{equation*} \Ch_N= \sum_{p\geq 1} \binom{N+1}{p} \Fact_p \ .\tag{1} \end{equation*}

\medskip

We can inverse the formula, using the discrete derivative operator
$\Delta$: for $f:\BC\to \BC$, it is defined by $\Delta f:=(x \mapsto
f(x+1)-f(x))$. It is well known that $f$ is polynomial if and only if
$\Delta^kf=0$ for $k$ large enough, and for $P$ polynomial we have the
following formulas:
\begin{itemize}
\item $\forall n \in \BN$, $\displaystyle{\Delta ^n P (X) = \sum_{k=0}^n
    (-1)^{n-k} \binom{n}{k} P(X+k)}$;
\item $\forall a \in \BC$, $\displaystyle{P(a+X)= \sum_{k=0}^{\deg
      P} \Delta ^k P(a) \binom{X}{k} }$.
\end{itemize}

As a consequence, $\Ch_N$ is a polynomial in $N$. We have
indeed ${\Ch_N=Z(N+1)}$, where $Z$ is the so-called \emph{Zeta} polynomial of the
poset $(P_c,\Aleq)$. We also obtain:
\begin{equation*} \Fact_p=\Delta ^p Z(0)= \sum_{k=1}^p (-1)^{p-k}
  \binom{p}{k}\Ch_k \ .\tag{2} \end{equation*}

\section{Formulas for block factorisations of a Coxeter element}

We consider here $G=W$ a well-generated irreducible complex reflection
group (of rank $n$), and $A=\CR$ the set of all its reflections. We fix a
Coxeter element $c$.

``Chapoton's formula'' states that the number of $N$-multichains in
$P_c=\NCP_W(c)$ is the Fuss-Catalan number $\Cat^{(N)}(W)=\prod_i
\frac{d_i+Nh}{d_i}$, \ie:
\[ Z(X) = \prod_{i=1}^n \frac{d_i + (X-1)h}{d_i} \ ,\]
where $Z$ is the Zeta polynomial of $\NCP_W$, and $d_1\leq \dots \leq
d_n=h$ are the invariant degrees of $W$.

\begin{remark}
\label{rkramifLL}
From Equation (1), we deduce the formulas:
\begin{equation*} \forall N \in \BN,\ \sum_{p = 1}^n \binom{N+1}{p}
  \Fact_p = \prod_{i=1}^n \frac{d_i+Nh}{d_i} \ . \tag{3} \end{equation*} As the
numbers $\Fact_p$ are related to the cardinality of fibers of the
Lyashko-Looijenga morphism (see Thm. \ref{thmbij}), these equalities can be
interpreted as ``ramification formulas'' for $\LL$. Instead of trying to
prove Chapoton's formula by computing explicitely the $\Fact_p$ (which leads
to quite complicated expressions, see Prop. \ref{calculfact} below),
another approach could be to prove directly Equation (3) as a global
property of $\LL$.
\end{remark}
 
We give below the numbers $\Fact_p$ of factorisations of $c$ in $p$ blocks,
as predicted by Chapoton's formula. Note that for $\Fact_{n-1}$ we recover
the formula obtained in Thm. \ref{thmfact}.

\begin{propo}
\label{calculfact}
  We denote by  $\sigma_0^* =1, \sigma_1^*=\sum_{i=1}^{n-1}d_i^*, \dots,
  \sigma_{n-1}=\prod_{i=1}^{n-1}d_i^*$ the elementary symmetric polynomials
  in the  $n-1$ codegrees $d_1^* \geq \dots \geq d_{n-1}^*$ (where
  $d_i^*=h-d_i$). We write $S(p,k)$ for the number of partitions of a
  $p$-set into $k$ nonempty subsets (Stirling number of the second
  kind). Then:

  \[ \forall p\in \{0,\dots, n\},\ \Fact_{n-p}=\frac{(n-p)!\ h^{n-p}}{|W|}
  \sum_{j=0}^{p} (-1)^{p-j} \ \sigma_{p-j}^* \ S(n-p+j,n-p)\ h^j \ .\] 
    
  \[\begin{disarray}{llll}
    \text{In particular:} & \Fact_n & = & \frac{n! \ h^n}{|W|} \ ;\\
    & \Fact_{n-1} & = & \frac{(n-1)!\ h^{n-1}}{|W|} \left( -\sigma_1^* +
      \frac{n(n-1)}{2}h \right) \ .
  \end{disarray}\]
\end{propo}

\begin{proof}
  We develop the expression $Z(X)=\frac{hX}{|W|} \prod_{i=1}^{n-1} (d_i^*
  +hX)$ in the canonical basis. Then we apply Equation (2) and compute
  $\Delta^{n-p}Z(0)$, using the fact that $\Delta^r(X^k)(0)= r!\ S(k,r)$
  (cf. Prop. 1.4.2 in \cite{stanley}).
\end{proof}

\selectlanguage{francais}

% ================================== BIBLIOGRAPHIE =============================

%% Choix du style 
%% En français
\bibliographystyle{alpha-fr} % style alphabétique en français
%
%% En anglais
%\bibliographystyle{alpha} % style alphabétique en anglais
%\bibliographystyle{plain} % style numéroté en anglais
%% Il y a plein d'autres possibilités
%% Fabrication de la biblio
\bibliography{arxiv_totalbibli} % pour afficher la biblio
% utilise le fichier bibliothese.bib

% Remarque : l'ajout de la biblio à la table des matières se fait par le
% paquet tocbibind (car la commande addcontentsline ne fabrique pas le bon
% numéro de page)

\end{document}